\documentclass[11 pt]{article}
\usepackage[utf8]{inputenc}
\pdfoutput=1
\usepackage{amsmath}
\usepackage{amssymb}
\usepackage{graphicx}
\usepackage{setspace}
\usepackage{centernot}
\usepackage{amsthm}
\usepackage{fullpage}
\usepackage{amsmath,amssymb,amsfonts}
\usepackage{float}
\usepackage{mathtools}
\usepackage{array}
\usepackage{multicol}
\usepackage{enumitem}
\usepackage{tcolorbox}
\usepackage{mathrsfs}
\usepackage{tikz-cd}
\usepackage{xypic}
\usepackage{geometry}
\usepackage{bbm}
\usepackage[backend=biber,style=alphabetic]{biblatex}
\usepackage{hyperref}
\hypersetup{
	colorlinks=true,
	linkcolor=blue,
	filecolor=cyan,      
	urlcolor=blue,
	citecolor=red,
}
\usepackage{upgreek}
\usepackage{extpfeil}

\numberwithin{equation}{subsection}
\newtheorem{theorem}{Theorem}[section]
\newtheorem{proposition}[theorem]{Proposition}
\newtheorem{lemma}[theorem]{Lemma}
\newtheorem{corollary}[theorem]{Corollary}
\theoremstyle{definition}
\newtheorem{definition}[theorem]{Definition}
\theoremstyle{remark}
\newtheorem{claim}{Claim}[theorem]
\newtheorem{remark}[theorem]{Remark}

\theoremstyle{conjecture}
\newtheorem{conjecture}{Conjecture}[section]
\DeclareMathOperator{\Gal}{Gal}

\newcommand{\rG}{\mathrm G}

\newcommand{\rP}{\mathrm P}

\newcommand{\rL}{\mathrm L}
\newcommand{\rU}{\mathrm U}

\newcommand{\rT}{\mathrm T}
\newcommand{\rB}{\mathrm B}
\newcommand{\rV}{\mathrm V}

\newcommand{\sA}{\mathscr A}

\newcommand{\X}{\mathfrak X}
\newcommand{\fg}{\mathfrak g}

\newcommand{\ft}{\mathfrak t}

\newcommand{\fu}{\mathfrak u}
\newcommand{\fs}{\mathfrak s}

\newcommand{\bt}{\bar{\mathfrak t}}
\newcommand{\bg}{\bar{\mathfrak g}}
\newcommand{\bu}{\bar{\mathfrak u}}
\newcommand{\bl}{\bar{\mathfrak l}}
\newcommand{\bp}{\bar{\mathfrak p}}

\newcommand{\Res}{\mathrm{Res}}
\newcommand{\res}{\mathrm{res}}
\newcommand{\rv}{\mathrm{v}}
\newcommand{\Lie}{\mathrm{Lie}}
\newcommand{\Irr}{\mathrm{Irr}}
\newcommand{\Ind}{\mathrm{Ind}}
\newcommand{\Ad}{\mathrm{Ad}}
\newcommand{\MP}{\mathrm{MP}}
\newcommand{\DL}{\mathrm{DL}}
\newcommand{\RP}{\mathrm{RP}}
\newcommand{\JH}{\mathrm{JH}}
\newcommand{\Hom}{\mathrm{Hom}}
\newcommand{\IS}{\mathrm{IS}}
\newcommand{\CS}{\mathrm{CS}}
\newcommand{\Id}{\mathrm{Id}}
\newcommand{\GL}{\mathrm{GL}}

\newcommand{\mbm}{\mathbbm}
\newcommand{\Z}{\mathbb Z}
\newcommand{\mbP}{\mathbb P}
\newcommand{\R}{\mathbb R}

\newcommand{\T}{\mathbb T}
\newcommand{\Q}{\mathbb Q}
\newcommand{\U}{\mathbb U}
\newcommand{\mbL}{\mathbb L}
\newcommand{\C}{\mathbb C}

\newcommand{\HH}{\mathcal H}
\newcommand{\F}{\mathbb F}
\newcommand{\A}{\mathcal A}

\newcommand{\mcF}{\mathcal F}
\newcommand{\G}{\mathbb G}
\newcommand{\mf}{\mathfrak}
\newcommand{\ms}{\mathscr}
\newcommand{\M}{\mathcal{M}}
\newcommand{\B}{\mathcal{B}}
\newcommand{\ZZ}{\mathcal{Z}}
\newcommand{\mbk}{\mathbbm k}
\newcommand{\mc}{\mathcal}
\newcommand{\mcC}{\mathcal C}

\newcommand{\blue}{\textcolor{blue}}
\usepackage{tikz}
\usetikzlibrary{cd}
\usetikzlibrary{arrows}
\usetikzlibrary{matrix}
\title{A description of the depth-$r$ Bernstein center for rational depths}
\author{Sarbartha Bhattacharya and Tsao-Hsien Chen}

\date{}

\begin{document}

\maketitle
\begin{abstract}
Let $G$ be a split connected reductive  over a non-archimedean local field $k$.
In this paper we give a description of the depth-$r$ Bernstein center of $G(k)$ for rational depths 
as a limit of depth-$r$ standard parahoric Hecke algebras, extending our previous work in the integral depths case in \cite{cb24}. Using this description, we construct maps from the space of stable functions on depth-$r$ Moy-Prasad quotients to the depth-$r$ center, and attach depth-$r$ Deligne-Lusztig parameters to smooth irreducible representations, with the assignment of parameters to irreducible representations shown to be consistent with restricted Langlands parameters for Moy-Prasad types described in \cite{chendebackertsai}. 
As an application, we  give a decomposition of the category of smooth representations into a product of full subcategories
indexed by restricted depth-$r$ Langlands parameters.


\end{abstract}
\tableofcontents
\section{Introduction}
Let $k$ be a non-archimedean local field with ring of integers $\mf O_k$ and residue field $\kappa_k=\F_q$ of characteristic $p$. Let $\varpi \in \mf O_k$ be a uniformizer and $\mf m_k=\langle\varpi \mf O_k\rangle$ be the unique prime ideal in $\mf O_k$. We fix a separable closure $\bar{k}$ of $k$, and let $K=k^u$ be the maximal unramified extension of $k$ and $k^t$ denote the maximal tamely ramified extension of $k$. 
Let $G$ be a connected reductive algebraic group defined over $k$. We assume throughout that $G$ splits over $k$ and fix a $k$-split maximal tori $T$ in $G$. We denote the reduced Bruhat-Tits building of $G$ by $\B(G,k)$. Further, we fix a Haar measure $\mu$ on $G(k)$ and identify $\HH(G)$ and compactly supported smooth functions on $G(k)$, denoted by $C^\infty_c(G)$. Let $R(G)$ denote the category of smooth complex representations of $G(k)$ and $\mathcal{Z}(G)=\mathrm{End}(\mathrm{Id}_{R(G)})$ denote the Bernstein center. There are several equivalent descriptions of the Bernstein center, including one which interprets elements in $\ZZ(G)$ as essentially compact invariant distributions on $G(k)$ (for details check Section \ref{section: Bernstein center}). We denote the the set of (isomorphism classes of) smooth irreducible representations of $G(k)$ by $\Irr(G)$. Moy and Prasad in \cite{MP94} attached an invariant $\rho(\pi)\in \Q_{\geq 0}$ to each $\pi \in \Irr(G)$, called the depth of $\pi$. For any non-negative rational number $r\in\mathbb Q_{\geq 0}$,
let $R(G)_{\leq r}$ (resp, $R(G)_{> r}$) denote the full subcategory of smooth representations whose irreducible subquotients have depth $\leq r $ (resp, depth $>r$). Let $\Irr(G)_r$ denote the set of smooth irreducible representations of $G(k)$ depth $r$ (similarly $\Irr(G)_{\leq r}$, $\Irr(G)_{<r}$ and $\Irr(G)_{>r}$). Results of Bernstein and Moy-Prasad (\cite{Ber, MP94, MP96}) imply that the category $R(G)$ decomposes as a direct sum
 $R(G) = R(G)_{\leq r} \oplus R(G)_{> r}$, and hence the Bernstein center also decomposes as $\ZZ(G) = \ZZ^r(G) \oplus \ZZ^{>r}(G)$. In the first part of our article, we provide a description of the depth-$r$ Bernstein center $\ZZ^r(G)$. \par

\subsection{A description of the center}
 In \cite{BKV}, Bezrukavnikov-Kazhdan-Varshavsky worked with a categorical analogue of the Bernstein center, and constructed an invariant distribution $E_0 \in \ZZ^0(G)$ which is the projector to the depth-zero part of the center using $l$-adic sheaves on loop groups. In our previous work \cite{cb24}, we used some the ideas developed in \cite{BKV} and \cite{chendepthzerostable} in a more classical setting to give a description of the integral depth-$r$ center $\ZZ^r(G)$ for a split simply-connected $p$-adic group $G(k)$ (see \cite[Theorem 2.8, Theorem 3.11]{cb24}). In particular, this provided a positive answer to the conjecture mentioned in the abstract of \cite{chendepthzerostable}.\par 
 In their subsequent work \cite{BKV2}, Bezrukavnikov-Kazhdan-Varshavsky proved an explicit description for the Bernstein projector $ E_r\in \ZZ^r(G)$ to representations of depth $\leq r$ for all rational depths. In order to deal with fractional depths, they replaced the usual simplicial structure on the building $\B(G,k)$ with a refinement by subdividing the facets into smaller parts depending on $m \in \Z_{>0}$ when the depth $r \in \frac{1}{m}\Z_{\geq 0}$ (check Section \ref{section: subdiving facets}). Let us denote the new structure obtained by $\B_m$. The main idea in the construction of the projectors was to use a $G(k)$-equivariant system of idempotents $\delta_{G_{\sigma,r+}}= \mu(G_{\sigma,r+})^{-1}\mbm 1_{G_{\sigma,r+}} \in \HH(G)$ (see Section \ref{section:MoyPrasadfiltrations} for definition of $G_{\sigma,r+}$) for each refined facet $\sigma$ in $\B_m$, and give a formula for the projector as an Euler-Poincar\'e sum of the idempotents (\cite[Theorem 1.6]{BKV2}). Some of the ideas in this work was motivated by the work of \cite{MeyerSolleveld} who used a system of idempotents in $\mathrm{End}_\C(V_\pi)$ for $(\pi,V_\pi) \in R(G)$ to produce $G(k)$-equivariant (co)sheaves, and used that to give resolutions of certain subspaces of $V_\pi$.\par
 Bernstein gave a decomposition of $R(G)$ into indecomposable full subcategories $R(G)_{\mf a}$ where $\mf a =[L,\varrho]_G$ denotes the inertial equivalence class of $(L,\varrho)$ for a $k$-Levi subgroup $L \subset G$ and $\varrho$ an irreducible supercuspidal representation of $L(k)$ (check Section \ref{section:decomposition} for definitions and \cite{roche2009bernstein} for more details). These subcategories are often called Bernstein components. Barbasch-Ciubotaru-Moy in \cite{barbaschciubotarumoy} established Euler-Poincar\'e formulas for projectors to individual depth-zero Bernstein components from a equivariant system of idempotents produced using cuspidal representations of the reductive quotients of parahoric subgroups, giving a decomposition of the depth-zero projector. Moy and Savin in \cite{moysavin} used similar ideas to produce a partial analogue for positive depths. \par

Let $[\B_m]$ denote the set of refined facets obtained by subdividing each open polysimplex ( check Section \ref{section: subdiving facets} for detailed description), and $G_{\sigma,r} $ be the Moy-Prasad filtration subgroups attached to $\sigma \in [\B_m]$ as described in Sections \ref{section:MoyPrasadfiltrations} and \ref{section: subdiving facets}. For each $\sigma \in [\B_m]$, we define 
\[
\M_\sigma^r:=C_c^\infty\left(\frac{G(k)/G_{\sigma,r+}}{G_{\sigma,0}}\right)
\]
to be the algebra (under convolution with respect to $\mu$) of compactly supported smooth functions on $G(k)$ which are $G_{\sigma,r+}$ bi-invariant and ${G_{\sigma,0}}$ conjugation invariant. Note that the idempotents $\delta_{G_{\sigma,r+}}$ used in the construction of projectors in \cite{BKV2} are elements in $\M_\sigma^r$. Let $\mcC$ be a fixed chamber in the apartment $\A_T$ corresponding to the fixed $k$-split maximal torus. \par
We have a partial order on $[\B_m]$ given by $\sigma' \preceq \sigma $ if $\sigma'$ is contained in the closure of $\sigma$, and this gives a partial order on $[\Bar{\mcC}_m]$. For $\sigma',\:\sigma\in [\Bar{\mcC}_m] $ and $\sigma' \preceq \sigma $, we have a map 
\begin{align*}
\phi^r_{\sigma',\sigma}:\M^r_{\sigma'} &\longrightarrow \M^r_{\sigma}\\
f &\longmapsto f*\delta_{G_{\sigma , r+}}  
\end{align*}
Further, for any element $n\in \mc N:=N_G(T)(k)$ such that $n\mcC=\mcC$, if $n\sigma_1=\sigma_1'\preceq \sigma_2$, we add morphisms $\phi^r_{\sigma_1,\sigma_2,n}:\M^r_{\sigma_1} \longrightarrow \M^r_{\sigma_2}$ in the following way
\[
\phi^r_{\sigma_1,\sigma_2,n}:\M^r_{\sigma_1} \xrightarrow[]{\Ad(n)} \M^r_{n\sigma_1}\xrightarrow[]{\phi^r_{n\sigma_1,\sigma_2}}\M^r_{\sigma_2} 
\]
With the above defined maps, we have an inverse system $\{\M^r_{\sigma}\}_{\sigma\in [\Bar{\mcC}_m] }$ and we define $A^r(G)$ to be the inverse limit of the algebras $\M^r_{\sigma}$.
$$ A^r(G) := \lim_{\sigma\in [\Bar{\mcC}_m] } \M^r_{\sigma}$$
We had a similar inverse system in the description of integral depths in \cite{cb24}, and the second set of maps were not required because we were working with simply-connected groups. Our first main result is a  generalization of \cite[Theorem 1.1]{cb24} to fractional depths and split reductive groups, stated and proved in Theorem \ref{thm: isominverselimit}.
\begin{theorem}
    There is an explicit algebra isomorphism $[A^r]: A^r(G) \xrightarrow{\simeq} \ZZ^r(G)$ for a split reductive group $G$ and $r \in \Q_{\geq 0}$.
\end{theorem}

Given $h=\{ h_\sigma\}_{\sigma\in [\Bar{\mcC}_m]}\in A^r(G)$, we can define $h_{\sigma'}$ for all $\sigma' \in [\B_m]$ and an Euler-Poincar\'e sum $[A^\Sigma_h] = \sum_{\sigma \in \Sigma }(-1)^{\text{dim} \: \sigma} h_\sigma \in \HH(G)$ for each finite convex subcomplex in $\B_m$.  Let $\Theta_m$ denote the set of non-empty finite convex subcomplexes $\Sigma \subset [\B_m]$. The main idea in the proof is to show that for every $f\in \mc H(G)$ and $h\in A^r(G)$, the sequence $\{[A^\Sigma_h] * f\}_{\Sigma \in \Theta_m} $ stabilizes (Theorem \ref{thm stab frac}) and use that to produce an element $[A_h] \in \mathrm{End}_{\HH(G)^{2}}(\HH(G)) \cong \ZZ(G)$. The final step is to show that the map $h \mapsto [A_h]$ gives an algebra isomorphism $A^r(G) \xrightarrow{\simeq} \ZZ^r(G)$. The image of $\delta_r=\{ \delta_{G_{\sigma,r+}}\}_{\sigma\in [\Bar{\mcC}_m]}\in A^r(G)$ is exactly the depth-$r$ projector constructed in \cite{BKV2}. 

\subsection{Stable functions and Deligne-Lusztig parameters}

In \cite{chendebackertsai}, Chen-Debacker-Tsai attached Deligne-Lusztig parameters (see sections \ref{section:positiveDLparameters} and \ref{section:depth0DLparameters} for the definition) to Moy-Prasad types and proved that these parameters are the same for any two Moy-Prasad types contained in a smooth irreducible representation. Similar parameters were studied in \cite{cb24} for split simply-connected groups, where they were called semi-simple part of a minimal $K$-type. Further, they established a connection with the Galois side and showed that these parameters are in bijection to restricted Langlands parameters which are continuous homomorphisms $I_k^r/I^{r+}_k \rightarrow G^\vee$ satisfying certain properties (see Definitions \ref{defn:depth0RPr} and \ref{defn:positiveRPr}), where $I_k^{r}$ denotes the upper numbering filtration of the Weil group $W_k$ of the local field $k$, and $G^\vee$ denotes the complex dual group of $G$. For $r \in \Q_{\geq 0}$, they define a map $\Irr(G)_r\rightarrow \DL_r\xrightarrow{\simeq}\RP_r$ where $\DL_r$ and $\RP_r$ denote the Deligne-Lusztig and restricted Langlands parameters of depth-$r$ respectively. The Deligne-Lusztig parameters decmposes the set $\Irr(G)$ into disjoint sets and conjecturally, these sets are unions of $L$-packets. The main aim of the subsequent sections is to use the description of the center in Section \ref{section:fracdepthdescription} to construct elements in $\ZZ(G)$ which act by the same constant on smooth irreducible representations having the same Deligne-Lusztig(DL) parameter attached to it.  \par

To that end, in section \ref{section:stablefnsonpositivedepth}, we define and study stable functions on positive depth Moy-Prasad quotients. For positive integral depths, they were studied in \cite{cb24}, inspired by similar concepts in \cite{laumon} and \cite{Chenbravermankazhdan}.
 In the case of positive integral depths, the Moy-Prasad filtration quotients at a point in the Bruhat-Tits building are isomorphic to the Lie algebra of the reductive quotient of the parahoric subgroup. The main tool in studying and contructing stable functions in this case was Fourier transforms on finite Lie algebras studied in \cite{Le}. To deal with fractional depths, we develop a theory of Fourier transforms on fractional depth quotients in Section \ref{subsection: Fouriertransform}, extending the results in \cite{Le} and use that to construct stable functions on fractional depth quotients for $r \in \Z_{(p)}\cap \Q_{>0}$. \par
 For the fixed maximal tori $T$, let $\T$ be the reductive quotient of $T(k^t)$, which is an $\F_q$-split $\bar\F_q$-torus canonically identified with the reductive quotient of $T(K)$ since $T$ is $k$-split, and let $\bt= \Lie(\T)$. The space $(\bt^*//W)^F$ where $F$ denotes the (geometric) Frobenius and $W$ the Weyl group of $T$ is the parameter space for depth-$r$ Deligne Lusztig parameters (along with some other data, see Section \ref{section:positiveDLparameters}). In Section \ref{section:stablefnstocenter}, we use the theory of stable functions to construct a elements in the depth-$r$ center via a map described in Theorem \ref{thmtmodWtocenter} whose main statement is the following:
 \begin{theorem}
   There is an algebra homomorphism 
    \begin{equation*}
        \xi^r: \C[(\bt^*//W)^F] \longrightarrow \ZZ^r(G)
    \end{equation*}
 \end{theorem}
This enables us to attach Deligne-Lusztig parameters to smooth irreducible representations of positive depth and define a map $\Theta_r: \Irr(G)_r\rightarrow \DL_r$ for $r \in \Z_{(p)}\cap \Q_{>0} $, which only takes values in non-trivial parameters, which are denoted by $\DL^\circ_r$. We further show in Proposition \ref{PropMP=DLpi} that the parameters we attach are the same as the ones in \cite{chendebackertsai}, and hence our construction produces elements in the center which act by the same constant on smooth irreducible representations having the same DL parameter. We deal with the depth zero case and show similar results in Section \ref{section: Depthzerocase}, constructing a map $\Theta_0: \Irr(G)_0\rightarrow \DL_0$. Most of the ideas regarding stable functions in this depth zero was developed in \cite{cb24}, and we have generalized them to the reductive case.\par 
With the additional assumption that $p\nmid |W|$, we know that smooth irreducible representations have depth in $\Z_{(p)}\cap \Q_{\geq 0}$. Thus we have a map $\Theta: \Irr(G)\rightarrow \DL^t:=\DL_0\coprod \left(\coprod_{r \in \Z_{(p)}\cap \Q_{> 0}}\mathrm{DL}_r^\circ\right)$ and we can partition $\Irr(G)$ into packets $\Pi(\vartheta),\: \vartheta \in \DL^t$ having the same Deligne-Lusztig parameter. The maps in Theorem \ref{thmtmodWtocenter} and \ref{thmTmodWto0center} enable us to construct projectors to the packets (Proposition \ref{prop:projectorpackets})  and decompose the category $R(G)$ into a product of full subcategories (Theorem \ref{thm:decompositionDLpackets})
    \begin{equation}\label{eqn:decompositionDLpackintro}
        R(G) = \prod_{\vartheta\in \DL^t}R(G)_\vartheta
    \end{equation}
where $R(G)_\vartheta $ is the full subcategory of $R(G)$ consisting of those representations all of whose irreducible subquotients are in $\Pi(\vartheta)$. 
The decomposition in \eqref{eqn:decompositionDLpackintro} then gives a decomposition $R(G)_{\leq r}$ in terms of restricted Langlands parameters 
:
\begin{theorem}
For $r \in\Z_{(p)}\cap \Q_{\geq 0}$, we have a decomposition 
\begin{equation*}
  R(G)_{\leq r}= \bigoplus_{\varphi^r\in \RP_r}R(G)_{\varphi_r}  
\end{equation*}
of $R(G)_{\leq r}$ into a product of full subcategories indexed by restricted
depth-$r$ Langlands parameters $\RP_r$.
\end{theorem}
The theorem is restated in 
 Corollary \ref{corollary:decomprestricted and <r}.
In the case of positive depths, the subcategory corresponding to the to the trivial depth-$r$ parameter contains the smooth representations all of whose irreducible subquotients are in $\Irr(G)_{<r}$.

\subsection{A conjecture on stability}
An element in the Bernstein center is called stable if the associated invariant distribution is stable. Let $\ZZ^{st}(G)$ denote the vector subspace of stable elements in the center, and $\ZZ^{st,r}(G)=\ZZ^{st}(G)\cap \ZZ^{r}(G)$. The stable center conjecture asserts that $\ZZ^{st}(G) \subset \ZZ(G)$ is a unital sub-algebra. The stable center and some of its conjectural properties and equivalent interpretations are described briefly Section \ref{section: Bernstein center}. We have the following conjecture about the elements in $\ZZ(G)$ that lie in the image of $\xi^r$ under the assumption that the residue chracteristic is large enough, which we state in detail with some evidence in Section \ref{section:conjectures}. 
\begin{conjecture}
    For $r \in \Z_{(p)}\cap \Q_{\geq 0}$, we have $\mathrm{Im}(\xi^r)\subset \ZZ^{st,r}(G) $.
\end{conjecture}
In \cite{BKV}, Bezrukavnikov-Kazhdan-Varshavsky use the geometry of affine Springer fibers and $\ell$-adic perverse sheaves to 
study the depth-zero stable center conjecture. The results of the paper 
provide a possible  approach to the 
positive depth stable center conjecture  using the geometric framework
developed in \emph{loc. cit.}.

\subsection{Organization and some conventions}
We briefly summarize the main goals of each section. In Section \ref{section: Background and notations}, we give some necessary background about Bruhat-Tits buildings, Moy-Prasad filtrations, the Bernstein center and set up some notations used throughout the article. Section \ref{section:fracdepthdescription} gives a description of the depth-$r$ center for non negative rational depths, generalizing previous results in the integral depth case. In Section \ref{section:stablefnsonpositivedepth}, we define and study stable functions on positive depth Moy-Prasad quotients, using the theory of Fourier transforms on such quotients. In Section \ref{section:stablefnstocenter}, we use the stable functions to construct a map from the space of functions on depth-$r$ Deligne-Lusztig parameters to the depth-$r$ center.  The maps to the depth-$r$ center constructed in the previous section are used to attach Deligne-Lusztig parameters to smooth irreducible representations of positive depth in Section \ref{section: Parameters attached to repns of positive depth}. We also give a brief description of their relation to restricted Langlands parameters. Section \ref{section: Depthzerocase} studies the same results for the depth-zero case. We decompose the category of smooth representations into full subcategories in Section \ref{section:decomposition} using the partition of $\Irr(G)$ into disjoint sets via their Deligne-Lusztig parameters. Finally, in the last section, we have some conjectures about the stability of the elements in the center that we constructed. 
\paragraph{Some conventions:} Throughout the paper, we assume that $p>2$. In some sections, there are additional assumptions on $p$. We use $\mf F$ to denote a (geometric) Frobenius in the local field setting, for example in $W_k$ or in $\Gal(K/k)$, and $F$ to denote the (geometric) Frobenius in $\bar\F_q$-vector spaces or varieties with $\F_q$-structure. We generally always use the geometric Frobenius unless mentioned otherwise

\subsection{Acknowledgement}
The authors thank Connor Bass, Charlotte Chan, Stephen DeBacker and Cheng-Chiang Tsai
 for many useful discussions. T.-H. Chen also thanks the 
 NCTS-National Center for Theoretical Sciences at Taipei
 where parts of this work were done. The research of T.-H. Chen is supported by NSF grant DMS-2143722.

\section{Background and some notations}\label{section: Background and notations}
 
\subsection{Buildings- The split case}
We fix a discrete valuation on $\rv:k \rightarrow \Z \cup \{\infty\}$, which extends uniquely to $\rv:\bar k\rightarrow\Q\cup \{\infty\}$ and $\rv(K) \subseteq \Z \cup \{\infty\}$. Let $S$ be a $k$-split maximal torus in $G$, and $X^*(S)= \Hom_k(S,\G_m)$ and $X_*(S)=\Hom_k(\G_m, S)$ denote the lattice of $k$-rational characters and co-characters of $S$ respectively. Further, let $\Phi(G,S)$ and $\Phi^\vee(G,S)$ denote the root and co-root lattice of $G$ with respect to $S$, which can be identified with the absolute root system since $S$ is $k$-split. Henceforth, we will denote them by $\Phi$ and $\Phi^\vee$. We have a perfect pairing 
\begin{equation}\label{eqn:pairingchar co-char}
  <,>:  X_*(S)\times X^*(S) \rightarrow \Z 
\end{equation}
where $<\lambda, \chi>$ denotes the integer such that $\chi \circ \lambda(s)= s^{<\lambda, \chi>}$. Let $V_1= X_*(S)\otimes_\Z \R$, and we identify $V_1^*$ and $X^*(S)\otimes_\Z \R$. The pairing in \eqref{eqn:pairingchar co-char} canonically extends to $<,>:V_1^*\times V_1\rightarrow \R$. There is a unique group homomorphism $\nu_1: S(k)\rightarrow V_1$ such that 
\begin{equation}\label{eqn:tranlation by torus}
    <\nu_1(s), \chi>= - \rv(\chi(s))
\end{equation}
for all $s\in S(k)$ and $\chi\in X^*(S)$. Let $S_b(k)$ denote the kernel of $\nu_1$. Note that in this case $S_b(k)=S(\mf O_k)$. Let $N=N_G(S)$. $S(k)$ is a normal subgroup of $N(k)$, and we can immediately observe that $S_b(k)$ is also a normal subgroup of $N(k)$ since $\prescript{n}{}{}\chi \in X^*(S)$ for $n \in N(k)$. Thus, we have the exact sequence 
\begin{equation*}
    0 \rightarrow S(k)/S_b(k) \rightarrow N(k)/ S_b(k) \rightarrow N(k)/S(k) \rightarrow 1
\end{equation*}
The group $N(k)/S(k)$ is the Weyl group of the root system $\Phi$ and acts naturally on the vector spaces $V_1$ and $V_1^*$. The first group in the sequence $S(k)/S_b(k)$ is a free abelian group of rank dim $V_1$ (check \cite[ Lemma 1.3 ]{landvogtcompactificationbook}). Let $W_S=N(k)/S(k)$, $\Lambda= S(k)/S_b(k)$ and $\widetilde W_S= N(k)/S_b(k)$. 
Consider the subspace $V_0\subset V_1$ defined by 
\begin{equation*}
    V_0:=\{v \in V_1\:|\: \alpha(v)=0 \:\forall\: \alpha \in \Phi \} \cong X_*(Z(G))\otimes_\Z \R
\end{equation*}
The group $N(k)/S(k)$ acts trivially on $V_0$ and $V_0=0$ if $G$ is semisimple. Let $V= V(G,S,k)$ denote the vector space $V_1/V_0$ and $\nu:S(k)\rightarrow V$ be the composition $S(k) \xrightarrow[]{\nu_1} V_1\twoheadrightarrow V$. Note that $V \cong \Phi^\vee\otimes_\Z \R$ and the canonical group homomorphism $j_1: W_S \rightarrow \GL(V_1)$ induces a group homomorphism $j: W_S\rightarrow \GL(V)$ since the image of $j_1$ acts trivially on $V_0$. \par
Let $A$ be an affine space over $V$ and $\mathrm{Aff}(A)$ denote the group of affine isomorphism $A\rightarrow A$. There is an exact sequence
\begin{equation*}
    0\rightarrow V\rightarrow \mathrm{Aff}(A)\xrightarrow{d}\GL(V)\rightarrow 1 
\end{equation*}
which splits non-canonically (depending upon choice of $x \in A$) and hence $\mathrm{Aff}(A) \cong V \rtimes \GL(V)$ (check \cite[Section 1.2]{KP23}). We have an extension of groups 
\begin{equation*}
    0 \rightarrow \Lambda \rightarrow \widetilde W_S \rightarrow W_S \rightarrow 1
\end{equation*}
which represents a class in $H^2(W_S,\Lambda)$. Using the map $\Lambda \xhookrightarrow{\nu} V$ and functoriality of $H^2(-,W_S)$, we have its image in $H^2(V, W_S)$ given by 
\begin{equation*}
    0 \rightarrow V\rightarrow \widetilde W_S' \rightarrow W_S \rightarrow 1
\end{equation*}
However, $H^2(V, W_S)=0$, and hence the second extension is trivial, i.e., $\widetilde W_S' \cong V \rtimes W_S$. 
Since $\nu(w_1\cdot  l\cdot w_1^{-1})= j(\bar w_1)(\nu(l))$, for $w_1 \in \widetilde W_S$, $l \in \Lambda$ where $\bar w_1$ denotes the image of $w_1$ in $W_S$, we note that the map $j: W_S \rightarrow \GL(V)$ is induced by $W_S \rightarrow GL_{\Z}(\Lambda)$ and hence the map $W_S \rightarrow \GL(V)$ in the semi-direct product is given by $j$. Thus, there an affine space $A$ over $V$ and a map $\tilde f: V \rtimes W_S \rightarrow \mathrm{Aff}(A)$ such that $\tilde f((v,1))$ is translation by $v \in V$ and $d(f(v,w))= j (w)$. This gives a map $f: \widetilde W_S \rightarrow V \rtimes W_S \xrightarrow{f}\mathrm{Aff}(A)$ such that $f(\lambda)$ is translation by $\nu(\lambda)$ for $\lambda \in \Lambda$ and $d(f(w_1))= j(\bar w_1)$ for $w_1 \in \widetilde W_S$, which fits into the commutative diagram 
\begin{equation*}
    \begin{tikzcd}
        0 \arrow{r} &\Lambda\arrow{r}\arrow{d}{\nu} & \widetilde W_S \arrow{r}\arrow{d}{f}& W_S \arrow{r}\arrow{d}{j}&1\\
        0 \arrow{r}  & V\arrow{r} & \mathrm{Aff}(A) \arrow{r} & \GL(V) \arrow{r}& 1
    \end{tikzcd}
\end{equation*}
So, we have an affine space $A$ over $V$ and a map $N(k) \rightarrow \mathrm{Aff}(A)$ extending $\nu:S(k) \rightarrow V$, and by \cite[Proposition 1.8]{landvogtcompactificationbook}
such a pair is unique upto unique isomorphism. Hence we denote the unique map $N(k) \rightarrow \mathrm{Aff}(A)$ by $\nu $ as well. As mentioned in \cite[Section 2.2]{yu2009bruhat} the obstruction to existence of an isomorphism lies in $H^1(W_S,V)$ which is $0$, and the obstruction to its uniqueness lies in $H^0(W_S,V)=V^{W_S}=0$. The affine space $\A(G,S):=\A_S:=A$ along with the group homomorphism $\nu:N(k)\rightarrow \mathrm{Aff}(\A_S)$ is called the (reduced) apartment of $G$ with respect to $S$. \par 
Note that if we were working with $V_1$ instead of $V$, the pair would have been unique, but not upto unique isomorphism. Working with the full vector space $V_1$ gives an affine space often called the extended apartment. The exteneded apartment is canonical in the case of semisimple groups, in which case $V=V_1$. \paragraph{}

We will give a concrete realization of the apartment, and this is the notion we will use most often. The apartment $\A_S$ is an affine space under $V \cong \Phi^\vee\otimes_\Z\R$, and the pair $(\A_S,\nu)$ is unique upto unique isomorphism. Hence, it is enough to give one description of it. \par

For $\alpha \in \Phi$, let $U_\alpha$ denote the root subgroup of $G$ with respect to $S$, and $r_\alpha$ denote the reflection with respect to $\alpha$ in the Weyl group of $\Phi$. A system $(x_\alpha)_{\alpha \in \Phi}$ of $k$-group isomorphisms $\G_a\rightarrow U_\alpha $ is called a $k$-\'epinglage of $G$ with respect to $S$. Two $k$ isomorphisms $x_\alpha:\G_a\rightarrow U_\alpha$
and $x_{-\alpha}: \G_a \rightarrow U_{-\alpha}$ are said to be associated if there is a $k$-group monomorphism $\varepsilon_\alpha : SL_2 \rightarrow G$ such that for $y \in \G_a(k)=k$, the following conditions hold :
\begin{equation*}
    x_\alpha (y)= \varepsilon_\alpha\begin{pmatrix}
        1&y\\
        0&1
    \end{pmatrix} \text{ and } x_{-\alpha}(y)=\varepsilon_\alpha \begin{pmatrix}
        1&0\\
        -y&1
    \end{pmatrix}
\end{equation*}
Note that $\varepsilon_\alpha$ is uniquely determined by this condition and $m_\alpha:= x_\alpha(1)x_{-\alpha}(1)x_\alpha(1)=\varepsilon_\alpha\begin{pmatrix}
        0&1\\
        -1&0
    \end{pmatrix}$ is in the normalizer $N_G(S)(k)$ of the maximal torus $S$. \par
    A system $\{x_\alpha:\G_a \xrightarrow{\simeq} U_\alpha\}_{\alpha\in \Phi}$ of $k$-group isomorphisms is called a $k$-Chevalley system of $G$ (with respect to $S$) if the following proprties are true.
    \begin{itemize}
        \item $x_\alpha$ and $x_{-\alpha}$ are associated for all $\alpha \in \Phi$. 
        \item For $\alpha ,\beta \in \Phi$, there exists $\epsilon_{\alpha ,\beta}\in \{\pm1\}$ such that for all $y \in \G_a(k)$, we have 
        \begin{equation*}
            x_{r_\alpha(\beta)}(y)=m_\alpha \cdot x_\beta(\epsilon_{\alpha,\beta}\cdot y)m_\alpha^{-1}
        \end{equation*}
        As per our convention $\epsilon_{\alpha,\alpha}=1$. 
    \end{itemize}

We fix a Chevalley system $\{x_\alpha:\G_a \xrightarrow{\simeq} U_\alpha\}_{\alpha\in \Phi}$ of $G$ with respect to $S$. We can define a valuation $\varphi_\alpha: U_\alpha(k) \rightarrow \Z \cup \{\infty\}$ of $U_\alpha(k)$ by 
\begin{equation*}
    \varphi_\alpha (u)= \rv(x_\alpha^{-1}(u)) \text{ for } u \in U_\alpha(k)
\end{equation*}
This valuation corresponds to a special point in the apartment $\A(S,k)$ denoted by $x_0$, which is the point fixed by $m_\alpha$ for all $\alpha\in \Phi$. Fixing the point $x_0$ identifies the affine space $\A(S,k)$ with $V$, with the point $x_0+v$ corresponding to the valuation $\tilde \{\varphi_\alpha(u)\}_{\alpha \in \Phi}$ given by $\tilde\varphi_\alpha(u)= \varphi_{\alpha}(u)+\alpha(v)$. The apartment $\A(S,k)$ can be defined as the affine space defined by the set of all such valuations $\tilde \{\varphi_\alpha(u)\}_{\alpha \in \Phi}$. The map $\nu|_{S(k)}$ is already defined as in \eqref{eqn:tranlation by torus}, and it is enough to specify the actions of $W_S$ on $V$ to give $\nu:N(k) \rightarrow \mathrm{Aff}(\A_S)$. Note that $m_\alpha$ maps to the reflections $r_\alpha$ under the isomorphism of $N(k)/S(k)$ with the Weyl group of $\Phi$, and $r_\alpha\in W_S$ has the obvious natural action on $V$, which finishes the definition of the pair $(\A_S,\nu)$.
\begin{remark}
    As mentioned in \cite[Remark 2.1.1]{fintzenmoyprasad}, similar definitions can be made without the associated condition for each root $\alpha$, and we can let $\epsilon_{\alpha,\alpha}\in \{\pm1\}$ such that $m_\alpha= x_\alpha(1)x_{-\alpha}(\epsilon_{\alpha,\alpha})x_\alpha(1)$ is contained in the normalizer of $S$. 
\end{remark}

For $\alpha \in \Phi$, let $\Gamma_\alpha:=\{ \varphi_\alpha(u)\:|\: u \in U_\alpha(k)\setminus\{1\}\}=\Z$. The set of affine roots $\Psi(\A_S)=\Psi(G,S)$ on $\A_S$ consists of affine functions given by 
\begin{equation*}
   \Psi(\A_S)=\{ x \mapsto \alpha(x-x_0)+\gamma\:|\: \alpha \in \Phi,\;\gamma \in  \Gamma_\alpha\}
\end{equation*}
We often denote the affine function $x \mapsto \alpha(x-x_0)+n$ by $\alpha+n$. For affine function on $\A_S$ of the form $\psi = \alpha +l$, $l \in \R$, let $\dot \psi=\nabla \psi:= \alpha $ denote and $H_\psi :=\{ x \in \A_S\;|\: \psi(x)=0\}$. The hyperplanes $\{H_\psi\:|\: \psi \in \Psi(\A_S)\}$ are affine subspaces of codimension $1$, often called a wall and they give $\A_S$ the structure of a poly-simplicial complex. The connected componenets of $\A_S \setminus \cup_{\psi \in \Psi(\A_S)}H_\psi$ are called chambers. Two points $x,y \in \A_S $  are called equivalent if for all affine roots $\psi$, $\psi(x)$ and $\psi(y)$ have the same sign or are both $0$. The equivalence classes are called facets, and two points are in the same facet if they belong to the same (open) poly-simplex. The apartment $\A_S$ can be written as a disjoint union of these open polysimplices. There is a $W_S$-invariant scalar product on $V$, and if we equip $\A_S$ with the metric defined by the scalar product, the $N(k)$ action on $\A_S$ becomes isometrical. \par

Using affine functions of the form $\alpha +l$ for $l \in \R$, we obtain a filtration of the root subgroups $U_\alpha(k) $ for $\alpha \in \Phi$. For $\psi =\alpha +l$ we define 
\begin{equation*}
    U_{\psi}(k)=\{ u \in U_{\dot\psi}(k)\:|\: u=1 \text{ or } \varphi_{\dot\psi}(u)\geq \psi(x_0)\}= \{x_\alpha(y)\:|\: \rv(y)\geq l\} \cup \{1\}
\end{equation*}
often denoted by $U_{\alpha,l}$ as well. We can similarly define $U_{\psi +}$ as 
\begin{equation*}
    U_{\psi +}(k)=\{ u \in U_{\dot\psi}(k)\:|\: u=1 \text{ or } \varphi_{\dot\psi}(u)> \psi(x_0)\}= \{x_\alpha(y)\:|\: \rv(y)> l\} \cup \{1\}
\end{equation*}
For a bounded subset $\Omega \subset \A_S$ let $f_{\Omega}: \Phi\rightarrow \R$ be defined as $f_\Omega(\alpha)=\mathrm{inf}\{ l \in \R\:|\: \alpha(x-x_0+l\geq 0 \:\forall \: x \in \Omega\}$. We define $U_{\alpha,\Omega}(k):= U_{\alpha, f_\Omega(\alpha)}(k)$ for $\alpha \in \Phi$ and $U_\Omega(k):=\langle U_{\alpha,\Omega}(k)\:|\: \alpha \in \Phi \rangle \subset G(k)$. For $\Omega = \{s\}$, we often denote $U_\Omega$ by $U_x$. \paragraph{}
Consider the equivalence relation on $G(k)\times \A_S$ defined by $(g,x)\sim (h,y)$ if there is an element $n \in N(k)$ with $y = \nu(n)(x)$ and $g^{-1}hn\in U_x$. Let $\B(G,k):= (G(k)\times \A_S)/\sim $. The canonical map 
\begin{align*}
    \A_S &\rightarrow \B(G,k)\\
    x &\mapsto [(1,x)]
\end{align*}
where $[(1,x)]$ denotes the equivalence class of $(1,x)$ is injective and hence we can identify $\A_S$ with its image in $\B(G,k)$. We have a $G(k)$-action on $\B(G,k)$ given by 
\begin{align*}
    G(k) \times \B(G,k) & \rightarrow \B(G,k)\\
    (g, [(h,x)]) &\mapsto [(gh,x)]
\end{align*}
The subsets of $\B(G,k)$ of the form $g\A_S$ for $g \in G(k)$ are called apartments. The apartment $g\A_S$ coincides with the apartment $\A(gSg^{-1},k)$. The stabilizer of $\A_S$ in $\B(G,k)$ is $N(k)$, and the map $S \mapsto \A_S$ is a $G(k)$ equivariant bijection between $k$-split maximal tori of $G$ and the set of apartments of $\B(G,k)$. A subset $Y \subset \B(G,k)$ is called a facet (resp. chamber) if there exists $g \in G(k)$ such that $gY \subset \A_S$ is a facet (resp. chamber). This gives the building a polysimplicial structure, and it is a union of these (open) polysimplices. We fix a $W_S$-invariant scalar product on $V$ and consider the induced metric on $\A_S$. There is a unique metric on $\B(G,k)$ which is $G(k)$-invariant and coincides with the one on $\A_S$ (check \cite[Section 13.14]{landvogtcompactificationbook}). The group $G(k)$ acts via isometries on $\B(G,k)$. The $G(k)$-set $\B(G,k)$ along with the polysimplicial structure is called the Bruhat-Tits building of $G$.

\subsection{Moy-Prasad Filtrations}\label{section:MoyPrasadfiltrations}
\paragraph{Filtrations for split Tori and its Lie algebra.}Let $S$ be a split torus, $\mf s = \Lie(S)$ and $r \in \R_{\geq 0}$. We define 
\begin{equation*}
    S(k)_0= \{ s \in S(k)\:|\: \rv(\chi(t))=0\: \forall \: \chi \in X^*_{\bar k}(S)\}
\end{equation*}
In the split case since $\Hom_k(S,\G_m)= \Hom_{\bar k}(S,\G_m)$, we have that $S_b(k)= S(k)_0$. In the general cases, $S(k)_0$ is a finite index subgroup of $S_b(k)$. Henceforth, we will identify them. The torus $S$ has a natural structure over $\mf O_k$, and $S(k)_0= S(\mf O_k)$. For any $r \in \R_{\geq 0}$, we define 
\begin{equation*}
    S(k)_r =\{ s \in S(k)_0\:|\: \rv(\chi(t)-1)\geq r\: \forall\: \chi \in X^*(S)\} 
\end{equation*}
and $S(k)_{r+}= \cup_{r_1>r}S(k)_{r_1}= \{ s \in S(k)_0\:|\: \rv(\chi(t)-1)> r\: \forall\: \chi \in X^*(S)\}$. For the Lie algebra $\mf s$ of $S$, we can similarly define a filtration by $\mf O_k$-modules
\begin{equation*}
    \mf s(k)_r= \{ X \in \mf s(k)\:|\: \rv(d\chi(X))\geq r \:\forall \chi \in X^*(S)\}
\end{equation*}
\paragraph{Filtrations of $G(k)$.} Let $x \in \A_S \subset \B(G,k)$. For $r \in \R_{\geq 0}$, we can define a filtration of the root group $U_\alpha(k)$ depending on $x$ as follows: 
\begin{equation*}
    U_{\alpha}(k)_{x,r}:= x_\alpha (\varpi^{\lceil r-\alpha(x-x_0)\rceil}\mf O_k)
\end{equation*}
We define the Moy-Prasad filtration subgroups $G(k)_{x,r}$ of $G(k)$ as 
\begin{equation*}
    G(k)_{x,r}=\langle S(k)_r, U_{\psi}(k)\:| \psi \in \Psi(\A_S), \; \psi(x)\geq r\rangle = \langle S(k)_r,\; U_\alpha(k)_{x,r}\:|\: \alpha \in \Phi(G,S)\rangle
\end{equation*}
The subgroups $G(k)_{x,0}\subset G(k)$ for $x \in \B(G,k)$ are called parahoric subgroups of $G(k)$ corresponding to $x$. We set $G(k)_{x,r+}= \cup_{r_1>r}G(k)_{x,r_1}$. When the ground field is clear, we will denote $G(k)_{x,r}$ and $G(k)_{x,r+} $ by $G_{x,r}$ and $G_{x,r+}$ respectively.

\paragraph{Filtrations of the Lie algebra and its dual}
Let $\fg = \Lie(G)$, $\mf u_\alpha = \Lie (U_\alpha) $ and $X_\alpha = dx_\alpha(1)$ where $dx_\alpha:\G_a\rightarrow \mf u_\alpha$ is the derivative of $x_\alpha:\G_a\rightarrow U_\alpha $. For $r \in \R_{\geq 0}$, we can define 
filtrations $\mf u_\alpha(k)_{x,r}$ of $\mf u_\alpha(k)$ depending on $x \in \A_S\subset \B(G,k)$ by 
\begin{equation*}
    \mf u_\alpha(k)_{x,r}= \varpi^{\lceil r-\alpha(x-x_0)\rceil} \mf O_k X_\alpha \subset \mf u_\alpha(k)
\end{equation*}
Note that given an affine function $\psi=\alpha +l$, $l \in \R$, we can define a filtration $\mf u_\psi(k)$ of $\mf u_\alpha(k)$ similarly. Then, we can define the Moy-Prasad filtrations of the Lie algebra $\fg(k)$ by 
\begin{equation*}
    \fg(k)_{x,r}= \mf s(k)_r\oplus \left(\bigoplus_{\alpha \in \Phi}\mf u_\alpha(k)_{x,r} \right)= \langle \mf s(k)_r, \mf u_\psi \:|\: \psi \in \Psi(\A_S), \; \psi(x)\geq r\rangle
\end{equation*}
and $\fg(k)_{x,r+}=\cup_{r_1>r}\fg(k)_{x,r+}$. \par
Let $\fg^*=\Hom(\fg,k)$ denote the dual of $\fg$. We also have the Moy-Prasad filtration subspaces of $\fg^*(k)$, defined in the following way :
\begin{equation}
   \fg^*(k)_{x,-r}=\{ X^* \in \fg^*(k)\:|\: X^*(Y) \in \varpi \mf O_k\: \forall \:Y \in \fg_{x,r+}\}
\end{equation}
and $\fg^*(k)_{x,-r+}=\cup_{r_1<r}\fg^*(k)_{x,-r_1} $. When the ground field is clear, we will denote $\fg^*(k)_{x,-r}$ and $\fg^*(k)_{x,-r+}$ by $\fg^*_{x,-r}$ and $\fg^*_{x,-r+}$ respectively. 

 \subsection{The Bernstein center}\label{section: Bernstein center}
We fix a Haar measure $\mu$ of $G(k)$ and let $\HH(G):=(C_c^\infty(G),*)$ be the Hecke algebra of compactly supported smooth functions on $G(k)$ with multiplication given by the convolution product 
\begin{equation}\label{eqn:convolutionwrtmu}
    f*g(x)= \int_{G(k)}f(xy^{-1})g(y)d\mu(y)
\end{equation}
The Bernstein center $\ZZ(G)$ of $G(k)$ is defined as the algebra of endomorphisms of the identity functor $\mathrm{End}(\Id_{R(G)})$ in the category $R(G)$ of smooth representations of a $p$-adic group. There are several equivalent ways of describing $\ZZ(G)$, and we give a brief review of them following \cite{BKV},\cite{haines}. \paragraph{}
A distribution is a $\C$-linear map $D:C_c^\infty(G)\rightarrow \C$. For $f \in C^\infty(G)$, let $\breve f$ denote the function $\breve f(x)=f(x^{-1})$. We define $\breve D(f)= D(\breve f )$ for a distribution $D$ and $f \in C^\infty_c(G)$. The convolution of a distribution with a function $f \in C^\infty_c(G)$ can be defined by 
\begin{equation*}
    (D*f)(g)=\breve D(g\cdot f)
\end{equation*}
where $g\cdot f(x)=f(xg)$. Note that $D*f \in C^\infty(G)$. A distribution $D$ is said to be essentially compact if $D*f \in C^\infty_c(G)$ for all $f \in C^\infty_c(G)$. Let $\prescript{g}{}{}f(x)= f(g^{-1}xg)$. A distribution is said to be $G(k)$-invariant if $D(\prescript{g}{}{}f)=D(f) \: \forall \: f \in C^\infty_c(G),\: g \in G(k)$. The set of essentially compact $G(k)$-invariant distributions is denoted by $\mc D(G)^G_{ec}$ and it is a an associative and commutative $\C$-algebra (check \cite[Corollary 3.1.2]{haines}) with convolution product defined in the following way:
\begin{equation*}
    (D_1*D_2)(f)=\breve D_1(D_2*f)
\end{equation*}
Note that this only works for essentially compact distributions. \paragraph{}

Given $(\pi, V) \in R(G)$, each $z\in \ZZ(G)$ defines an endomorphism $z|_V \in \text{End}_{G( k)}(V)$. In particular, if $(\pi, V) \in \Irr(G)$, each $z\in \ZZ(G)$ defines a function $f_z :Irr(G)  \rightarrow \C$ by Schur's Lemma such that $z|_V=f_z(\pi)\Id_V$. Moreover, the map $z\mapsto f_z$ is an algebra homomorphism $\ZZ(G) \rightarrow Fun(\: \Irr(G), \: \C)$, which is injective. \par
  Any smooth $G( k)$-representation is equivalently a non-degenerate $\HH(G)$-module. Let $(l, \HH(G))$ and $(r,\HH(G))$ denote the smooth $G( k)$-representations induced by left and right translations by $G( k)$ on $\HH(G)$. The action on $G( k)$ by $G( k)^2$, defined by $(g,h)(x)=gxh^{-1}$ gives a $G( k)^2$ action on $\HH(G)$, given by $(g,h)f(x)=l(g)r(h)f(x)= f(g^{-1}xh)$, and hence $\HH(G)^2$-module structure on $\HH(G)$. Note that the actions $l$ and $r$ commute, and the action of $\HH(G)^2$ on $\HH(G)$ is given by $(\alpha, \beta)f= \alpha * f* \breve \beta$, where $\breve\beta(x)=\beta(x^{-1})$. Each $z\in \ZZ(G)$ defines an endomorphism $z_\HH$ of the smooth representation $(l, \HH(G))$, and since the actions $l$ and $r$ commute, the endomorphism $z_\HH$ of the  Hecke algebra $\HH(G)$ commutes with left and right $G( k)$-actions and hence left and right convolutions. For every $(\pi, V) \in R(G)$, $v\in V$ and $h\in \HH(G)$, we have the equality $z_V(h(v))=(z_\HH(h))(v)$. Moreover, the map $z\mapsto z_\HH$ defines an algebra isomorphism $\ZZ(G) \xrightarrow{\sim} \text{End}_{\HH(G)^{2}}(\HH(G))$. \par

An element $z\in \ZZ(G)$ defines an endomorphism $z_{reg}$ of the $G( k)$ representation on $\HH(G)$ given by the conjugation action, and hence gives rise to an $G( k)$-invariant distribution $\nu_z$ such that $\nu_z(f)=z_{reg}(\breve f)(1) $ for all $f\in \HH(G)$. The invariant distribution $\nu_z$ can be characterised by the condition $\nu_z * h = z_\HH(h) \: \forall \: h\in \HH(G)$. Moreover, the map $z \mapsto \nu_z$ gives an isomorphism of $\ZZ(G)\xrightarrow{\simeq}\mc D(G)^G_{ec}$. \par

For a compact open subgroup $K \subset G(k)$, let $\delta_K \in \HH(G)$ be defined as $ \delta_K:= \mu(K)^{-1} \mbm 1_K$, where $\mbm 1_K$ is the characteristic function of $K$. We denote by $\HH(G,K)= \delta_K* \HH(G)*\delta_K$ the convolution algebra of $K$ bi-invariant compactly supported smooth functions on $G(k)$. Let $\ZZ(G,K)= Z(\HH(G,K))$ denote the center of $\HH(G,K)$, and they form a projective system with maps given by 
\begin{align*}
    \ZZ(G,K)& \rightarrow \ZZ(G,K')\\
    z_K &\mapsto z_K* \delta_{K'}
\end{align*}
for $K'\subset K$. An element of $\varprojlim_{K}\ZZ(G,K)$ acts on objects of $R(G)$ in a way that commutes with the $\HH(G)$-action, and we have an isomorphism $\ZZ(G) \xrightarrow{\simeq }\varprojlim_{K}\ZZ(G,K)$, where $K$ runs over all compact open subgroups of $G(k)$. The center $\ZZ(G)$ can also be described as the ring of regular functions $\C[\Omega(G)]$ on the variety of (super)cuspidal supports $\Omega(G)$ (see Section \ref{section:decomposition} and \cite[Section 3.3]{haines}). 

\paragraph{Stable Bernstein center.}  An element $x \in G$ is said to be strongly regular semisimple if the stabilizer $G_x:= \mathrm{Stab}_G(x) \subset G$ is a maximal torus. Let $G^{sr}$ denote the set of strongly regular semisimple elements of $G$. It is an open subvariety of $G$. For $f \in C_c^\infty(G)$ and $x \in G^{sr}(k)$, we can define the normalized orbital integral of $f$ at $x$ by 
\begin{equation*}
    O_x(f)= |D_G(x)|^{1/2} \int_{G(k)/G_x(k)} f(gxg^{-1})d\dot g
\end{equation*}
where $D_G(x)$ is the Weyl discriminant of $x$ in $G(k)$ and $d\dot g$ is the left $G(k)$-invariant Haar measure on $G(k)/G_x(k)$ induced by Haar measures on $G(k)$ and $G_x(k)$. When $x$ is semisimple, $G_x(k)$ is unimodular and hence $G(k)/G_x(k)$ does carry a $G(k)$-invariant measure. Note that $G^{sr}(k)$ is invariant under stable conjugacy, and we can define the stable orbital integral of $f$ at $x$ by 
\begin{equation*}
    O^{st}_x(f):=\sum_{x \sim_{st}x'/\sim}O_{x'}(f)
\end{equation*}
where the sum is over representatives of $G(k)$-conjugacy classes is the stable conjugacy class of $x$. This defines a $G(k)$-invariant distribution $O^{st}_x$, determined uniquely upto a constant. A function $f \in C^\infty_c(G)$ is called unstable if $O_x^{st}(f)=0$ for all $x \in G^{sr}(k)$, and an invariant distribution $D$ is called stable if $D(f)=0$ for every unstable $f \in C^\infty_c(G)$. An element $z \in \ZZ(G)$ in the Bernstein center is called stable if the associated distribution $\nu_z \in \mc D(G)^G_{ec}$ is a stable distribution. We denote by $\ZZ^{st}(G)$ the vector subspace of stable elements in the Bernstein center $\ZZ(G)$. \par The stable center conjecture states that $\ZZ^{st}(G)\subset \ZZ(G)$ is a subalgebra, and there is a decomposition of the set $\Irr(G)$ of irreducible representations of $G(k)$ into packets via characters of $\ZZ^{st}(G)$ , which is slightly coarser than the conjectural decomposition into $L$-packets as per the Local Langlands correspondence (check \cite[Section 3.1.4]{BKV}). \par
An element $z \in \ZZ(G)$ is called very stable if $\nu_z*f$ is unstable for every unstable $f \in C_c^\infty(G)$. We denote the set of very stable elements in $\ZZ(G)$ by $\ZZ^{vst}(G)$, and it is a commutative $\C$-algebra contained in $\ZZ^{st}(G)$. A stronger form of the stable center conjecture suggests that $\ZZ^{vst}(G)$ and $\ZZ^{st}(G)$ coincides (check \cite[Section 1.1]{hansen2025excursion}).

\paragraph{}Let $W_k$ denote the Weil group of the local field $k$, and $WD_k=W_k\ltimes \C$ denote the Weil-Deligne group. Further let $G^\vee $ denote the complex dual group of $G$ and $\prescript{L}{}{}G= G^\vee \ltimes W_k$ denote the $L$-group. The $G^\vee$ conjugacy class of an admissible (in the sense of \cite[Section 5]{haines}) homomorphism $\lambda: W_k \rightarrow \prescript{L}{}{}G$ is called an infinitesimal character.  Let $\Omega(\prescript{L}{}{}G)$ denote the variety of infinitesimal characters as defined in \cite[Section 5.3]{haines}. Under the assumption that Local Langlands correspondence for $G$ and its Levi subgroups is known, and some compatibility of LLC with normalized parabolic induction (check LLC+ as defined in \cite[Definition 5.2.1]{haines}), Haines showed in \cite[Proposition 5.5.1]{haines} that there is a morphism of algebraic varieties $p_1: \Omega(G) \rightarrow \Omega(\prescript{L}{}{}G)$ which is surjective when $G $ is quasi-split. Hence, we get an embedding $\C[\Omega(\prescript{L}{}{}G)] \hookrightarrow\C[\Omega(G)]\cong \ZZ(G)$, and $\C[\Omega(\prescript{L}{}{}G)]$ is defined in \cite[Section 5.3]{haines} as the stable Bernstein center. \par

Under the assumption that $G$ is quasi-split, enough is known about the Local Langlands correspondence of the group and existence of tempered $L$-packets, Varma showed in \cite[Theorem 1.1.5]{varmastable} that $\ZZ^{st}(G)=\ZZ^{vst}(G)$ and hence stable center conjecture is true. Under some additional assumptions (check \cite[Proposition 1.1.7]{varmastable}), it was proved in the same article that all three notions of the stable center are the same, i.e., $\ZZ^{st}(G)=\ZZ^{vst}(G) = p_1^*(\C[\Omega(\prescript{L}{}{}G)])$. 

\paragraph{Relations to LLC:} Assume that the Local Langlands correspondence is known for $G/k$, and let $\Phi(G/k)$ denote the set of Langlands parameters for $G(k)$. If $\phi_\pi: WD_k\rightarrow \prescript{L}{}{}G$ is a Langlands parameter attached to $\pi$ by LLC, then the $G^\vee$ conjugacy class of the restriction $\phi_\pi|_{W_k}=\lambda_\pi:W_k \rightarrow \prescript{L}{}{}G) $ is called the infinitesimal character attached to $\pi$. We define the infinitesimal class $\tilde\Pi(\lambda)$ of $\lambda:W_k \rightarrow \prescript{L}{}{}G$ to be the union of $L$-packets for which the corresponding $L$-parameters restricts to $\lambda$, i.e.,
\begin{equation*}
    \tilde\Pi(\lambda)= \coprod_{\phi|_{W_k=\lambda}}\Pi(\phi)
\end{equation*}
where $ \Pi(\phi)$ denotes the $L$-packet corresponding to $\phi \in \Phi(G/k)$. An element $z \in \ZZ^{st}(G)$ conjecturally acts by the same constant on irreducible representations in the same infinitesimal class, i.e, $f_z(\pi)=f_z(\pi')$ if $\lambda_\pi=\lambda_{\pi'}$. 

 \subsection{The fractional depths and subdividing facets}\label{section: subdiving facets}

 In our earlier work \cite{cb24}, we gave a description of the integral depth center for simply connected groups. In the present article, we extend the result to general reductive $p$-adic groups, and fractional depths. The description of the integral depth center only needed ``data" from the standard parahorics and their integer depth Moy-Prasad filtration subgroups. This is equivalent to fixing an apartment and an alcove(chamber) in it, and using the parahoric subgroups and integer depth filtration subgroups corresponding to the facets (or open polysimplices) in its closure. This is because the parahoric subgroups and their integral depth Moy-Prasad filtration subgroups do not change in the interior of a facet.  However, when we start considering fractional depths, the Moy-Prasad filtration subgroups for fractional depth also changes in the interior of a facet, and how it changes depends upon $m\in \Z_{>0}$ where the depth $r \in  \frac{1}{m}\Z_{>0}$. Hence, in order to deal with fractional depths, we have to subdivide the facts in the reduced Bruhat-Tits building into smaller parts depending on the same $m$.\paragraph{}

 Let $\A=\A_S$ be the apartment corresponding to a $k$-split maximal tori $S$, and for $m\in \Z_{>0}$, let $\Psi_m(\A)$ denote the set of affine functions of the form $\psi + \frac{1}{m}\Z$, $\psi \in \Psi(\A)$. These are affine functions of the form $\Phi + \frac{k}{m}$ for $\Phi \in \Phi(G,S)$ and $\Psi_m(\A)\subset \Psi(\A)$. We can use the co-dimension one affine subspaces $ H_\psi= \{x \in \A\:|\: \psi(x)=0\}$ for $\psi \in \Psi_m(\A)$ to get a refined polysimplicial decomposition of $\A$, and hence of $\B(G,k)$. We denote $\B(G,k)$ (resp. $\A$) with the new polysimplicial decomposition $\B_m$ (resp. $ \A_m$), and we call the facets we get in this case refined facets. Let $[\B]$ denote the set of facets (or open polysimplices) of $\B(G,k)$ and $[\B_m]$ denote the set of refined facets obtained by using $\Psi_m(\A)$. These are obtained by ``subdividing each polysimplex $\sigma \in [\B]$ into $m^{\mathrm{dim}\:\sigma} $ smaller polysimplices". Similarly, we let $[\A]$ (resp. $[\A_m]$) denote the set of facets (resp. refined facets) in $\A$ for an apartment $\A$, and $[\bar \mcC]$ and $[\bar \mcC_m]$ the corresponding set for the closure of a chamber $\mcC \subset \A$. \par
 Let $\sigma \in [\B_m]$, and $x \in \sigma $. For $r \in \frac{1}{m}\Z_{>0}$, we can define $G_{\sigma,r}:=G(k)_{\sigma,r}= G(k)_{x,r}$. Since $r \in \frac{1}{m}\Z_{>0}$ and $\sigma \in [\B_m]$, the definition does not depend on the choice of $x \in \sigma$. The group $G_{\sigma, r+}=\cup_{s>r}G_{\sigma, s}$ is defined in the ususal manner and $G_{\sigma, 0}$ denotes the parahoric subgroup corresponding to $\sigma$. Further, let $G_\sigma= \text{Stab}_{G(k)}(\sigma)$ denote the stabilizer of the (refined) facet. 
\paragraph{Some notations:} We can define a partial order on $[\B_m]$ by defining $\sigma \preceq \tau$ if $\sigma $ is contained in the closure of $\tau$, and we call $\sigma$ a face of $\tau$. A facet of dimension $0$ is called a vertex, and we denote the set of vertices in $\B_m$ by $V(\B_m)$. A subset $\Sigma \subset [\B_m]$ is a subcomplex if $|\Sigma|=\cup_{\sigma \in \Sigma}\sigma \subset \B(G,k)$ is closed. It is convex if $|\Sigma|$ is convex.

\section{Description of fractional depth Bernstein center }\label{section:fracdepthdescription}
\subsection{Stabilization in the fractional depth case}
 
We fix a Haar measure $\mu$ of $G(k)$ and let $\HH(G):=(C_c^\infty(G),*)$ be the Hecke algebra of compactly supported smooth functions on $G(k)$ with multiplication given by the convolution product with respect to $\mu $ as defined in \eqref{eqn:convolutionwrtmu}. Let $T$ be a $k$-split maximal torus which we fix henceforth, and let $\sA:=\A_T$ be the apartment corresponding to $T$. We also fix a fundamental alcove (chamber) $\mcC \subset \mc A_T$. This is equivalent to fixing a Iwahori subgroup. Let $\Bar{\mcC}$ denote it's closure, and $[\Bar{\mcC}_m]$ denote the set of refined facets obtained by subdividing the facets in $\Bar{\mcC}$. Note that $[\Bar{\mcC}_m]$ is finite. For each $\sigma \in [\B_m]$, we define 
\[
\M_\sigma^r:=C_c^\infty\left(\frac{G(k)/G_{\sigma,r+}}{G_{\sigma,0}}\right)
\]
to be the algebra (under convolution with respect to $\mu$) of compactly supported smooth functions on $G(k)$ which are $G_{\sigma,r+}$ bi-invariant and ${G_{\sigma,0}}$ conjugation invariant. \par
We have the partial order on $[\B_m]$ given by $\sigma' \preceq \sigma $ if $\sigma'$ is contained in the closure of $\sigma$, and this gives a partial order on $[\Bar{\mcC}_m]$. For $\sigma',\:\sigma\in [\Bar{\mcC}_m] $ and $\sigma' \preceq \sigma $, we have a map 
\begin{align*}
\phi^r_{\sigma',\sigma}:\M^r_{\sigma'} &\longrightarrow \M^r_{\sigma}\\
f &\longmapsto f*\delta_{G_{\sigma , r+}}  
\end{align*}
Further, for any element $n\in \mc N:=N_G(T)(k)$ such that $n\mcC=\mcC$, if $n\sigma_1=\sigma_1'\preceq \sigma_2$, we add morphisms $\phi^r_{\sigma_1,\sigma_2,n}:\M^r_{\sigma_1} \longrightarrow \M^r_{\sigma_2}$ in the following way
\[
\phi^r_{\sigma_1,\sigma_2,n}:\M^r_{\sigma_1} \xrightarrow[]{\Ad(n)} \M^r_{n\sigma_1}\xrightarrow[]{\phi^r_{n\sigma_1,\sigma_2}}\M^r_{\sigma_2} 
\]

With the above defined maps, we have an inverse system $\{\M^r_{\sigma}\}_{\sigma\in [\Bar{\mcC}_m] }$ and we define $A^r(G)$ to be the inverse limit of the algebras $\M^r_{\sigma}$.
$$ A^r(G) := \lim_{\sigma\in [\Bar{\mcC}_m] } \M^r_{\sigma}$$

Given $h=\{ h_\sigma\}_{\sigma\in [\Bar{\mcC}_m]}\in A^r(G)$, we can define $h_{\sigma'}$ for all $\sigma' \in [\B_m]$. If $\sigma'\in [\Bar{\mcC}_m]$, then it is already defined. Otherwise,there exists a chamber $\mcC' \subset \B$ such that $\sigma' \subset \Bar{\mcC'}$, and $g\in G(k)$ such that $g\mcC=\mcC'$ and $ g\sigma =\sigma'$ for some $\sigma \in [\Bar{\mcC}_m]$. In this case, we define $h_{\sigma'} = \Ad_g(h_\sigma)$. \par
Showing that $h_{\sigma'}$ is well-defined is slightly subtle. If $g_1\mcC=g_2\mcC=\mcC'$ and $g_1 \sigma_1 =g_2\sigma_2 =\sigma'$ for $\sigma_1,\sigma_2 \in [\Bar{\mcC}_m]$ , then $g_2^{-1}g_1\mcC=\mcC$, which means $g_2^{-1}g_1 \in G_\mcC$. Let $\mc N_\Omega = \text{Stab}_{\mc N}(\Omega)$ for any bounded subset $\Omega\subset \A_T$. Using results in Section 7.7 in \cite{KP23}( for example equation $7.7.1,\;7.7.2$ and Proposition $7.7.5$) or Proposition $4.6.28 (\mathrm{ii})$ in \cite{BT84}, we get that $G_\mcC=\mc N_\mcC G_{\mcC,0}$. Therefore, we can write $g_2^{-1}g_1 = n h $ for $\in \mc N_\mcC,\: h\in G_{\mcC,0}$.Our assumptions imply that $n\mcC= \mcC$ and $n\sigma_1=\sigma_2$ since $G_{\mcC,0}$ pointwise stabilizes $\mcC$ and hence $\Bar{\mcC}$.  We are trying to show that $\Ad_{g_2}(h_{\sigma_2})= \Ad_{g_1}(h_{\sigma_1}) \Leftrightarrow \Ad_{g_2^{-1}g_1}(h_{\sigma_1}) = h_{\sigma_2}$. Since $h_{\sigma_i} $ is $G_{\sigma_i,0}$-conjugation invariant and $G_{\sigma_i,0}\supset G_{\mcC,0}$ for $i=1,2$, in order to show that $h_{\sigma'}$ is well-defined, it is enough to show that $\Ad_n(h_{\sigma_1})=h_{\sigma_2}$ for $n \in \mc N_\mcC\subset \mc N$ such that $n\mcC=\mcC$ and $n\sigma_1 =\sigma_2$ for $\sigma_1,\sigma_2  \in [\Bar{\mcC}_m] $. Now, $n\sigma_1 =\sigma_2 \preceq \sigma_2 $, we have a morphism $\phi^r_{\sigma_1,\sigma_2,n}:\M^r_{\sigma_1} \longrightarrow \M^r_{\sigma_2}$ where $\phi^r_{\sigma_1,\sigma_2,n} = \Ad_n$ in this case since $\phi^r_{n\sigma_1,\sigma_2}=\phi^r_{\sigma_2,\sigma_2}= \text{Id}_{\M^r_{\sigma_2}}$. Since $h=\{ h_\sigma\}_{\sigma\in [\Bar{\mcC}_m]}\in A^r(G)$, we have $\phi^r_{\sigma_1,\sigma_2,n}(h_{\sigma_1})=\Ad_n(h_{\sigma_1})= h_{\sigma_2}$, and we are done. Also, defined this way $h_{\sigma'} \in \M_{\sigma'}^r$. \par

For every finite subset $\Sigma \subset [\B_m]$, we can associate an element $[A^\Sigma_h]$ to $h \in A^r(G)$
\begin{equation}\label{averaging defn}
    [A^\Sigma_h] = \sum_{\sigma \in \Sigma }(-1)^{\text{dim} \: \sigma} h_\sigma \in \HH(G)
\end{equation}
Note that $\delta_r= \{\delta_{G_{\sigma,r+}}\}_{\sigma\in [\Bar{\mcC}_m]}\in A^r(G)$, and $[A^\Sigma_{\delta_r}]=E^\Sigma_r$ as defined in \cite{BKV2}.  Let $\Theta_m$ denote the set of non-empty finite convex subcomplexes $\Sigma \subset [\B_m]$. Note that $\Theta_m$ is an inductive system with respect to inclusions. 

\begin{theorem}\label{thm stab frac}
     For every $f\in \mc H(G)$ and $h\in A^r(G)$, the sequence $\{[A^\Sigma_h] * f\}_{\Sigma \in \Theta_m} $ stabilizes, and hence $\lim_{\Sigma \in \Theta_m} \:[A^\Sigma_h] * f $ is well-defined. 
\end{theorem}
In order to prove this theorem we need generalizations of some results in \cite{BKV2}.  We  first recall some notations and definitions.
\paragraph{Some notations:} \begin{itemize}
    \item [(i)]Let $\A\subset \B$ an apartment, and $\sigma \in [\A_m]$ a chamber (facet of maximal dimension). We use $\Delta_\A(\sigma)$ to denote the set of all $\psi \in \Psi(\A)$ such that $\psi(\sigma)>0$, and $\psi(\sigma')=0$ for some face $\sigma'\preceq \sigma $ of co-dimension one. We call $\Delta_\A(\sigma)$ the set of simple affine roots relative to $\sigma$.
    \item[(ii)] Given $x \in V(\B_m),\: s \in \R_{\geq 0}$ and $\sigma' \in [\B_m]$, we denote by $\Upsilon_{x,s}$ the set of all chambers $\sigma \in [\B_m]$ such that for every apartment $\A \subset \B$ containing $\sigma$ and $x$ and for every $\psi \in \Delta_{\A}(\sigma)$ we have $\psi(x)\leq s $. 
    \item[(iii)] We denote by $\Gamma_s(\sigma',x) \subseteq [\B_m]$ the subcomplex consisting of all $\sigma \in [\A_m]$ such that for every $\psi \in \Psi_m(\A) $  satisfying $\psi(\sigma')\leq 0$ and $\psi(x)\leq s $,we have $\psi(\sigma)\leq 0$. By definition, $\Gamma_s(\sigma',x)$ is convex.
    \item[(iv)]  From \cite[Lemma 4.10]{BKV2}, there exists a unique minimal face $\sigma'=m_{x,s}(\sigma)$ of $\sigma$ such that $\sigma\in \Gamma_s(\sigma',x)$. This defines an idempotent map $m_{x,s}:[\B_m]\rightarrow[\B_m] $.
\end{itemize}

\begin{lemma}\label{lemma stab frac}
    Let $\sigma,\: \sigma' \in [\B_m]$, $x\in V(\B_m)$ and $r,s \in \frac{1}{m}\Z_{\geq 0} $ such that $\sigma' \preceq \sigma $ and $\sigma \in \Gamma_s(\sigma',x)$. Further, let $h=\{ h_\sigma\}_{\sigma\in [\Bar{\mcC}_m]}\in A^r(G)$.  Then we have the equality 
    \begin{equation}
        h_{\sigma}* \delta_{G_{x,(r+s)+}}=h_{\sigma'}*\delta_{G_{x,(r+s)+}}
    \end{equation}
\end{lemma}
\begin{proof}
    From Lemma 4.9 in \cite{BKV2},  we have $ \delta_{G_{\sigma, r+}} * \delta_{G_{x,(r+s)+}}=\delta_{G_{\sigma', r+}}*\delta_{G_{x,(r+s)+}}$, and hence our assertion for $h=\delta_r$.  Note that $ \sigma,\: \sigma' \subset \Bar{\mcC'}$ for some chamber $\mcC'\subset \B$. Then, there exists $g \in G(k)$ and $\sigma_1,\: {\sigma'}_1 \in [\bar{\mcC}_m]$, ${\sigma'}_1 \preceq \sigma_1$ such that $g\mcC =\mcC'$, $g{\sigma'}_1 =  \sigma' $ and $g \sigma_1= \sigma $. Since $h=\{ h_\sigma\}_{\sigma\in [\Bar{\mcC}_m]}\in A^r(G)$, we have $h_{\sigma_1}= h_{{\sigma'}_1}* \delta_{G_{\sigma_1, r+}}$. Further, $h_\sigma= \Ad_g(h_{\sigma_1})$, $h_{\sigma'}= \Ad_g(h_{\sigma_1'})$ and $\delta_{G_{\sigma, r+}}= \Ad_g(\delta_{G_{\sigma_1, r+}})$, which gives us that $h_{\sigma}= h_{{\sigma'}}* \delta_{G_{\sigma, r+}}$. Using this fact, we get 
    \begin{align*}
         h_{\sigma}* \delta_{G_{x,(r+s)+}}&= \left(h_{\sigma'}* \delta_{G_{\sigma, r+}}\right)* \delta_{G_{x,(r+s)+}} =h_{\sigma'}* \left(\delta_{G_{\sigma, r+}}* \delta_{G_{x,(r+s)+}}\right)\\
         &=h_{\sigma'}* \left(\delta_{G_{\sigma', r+}}* \delta_{G_{x,(r+s)+}}\right) \:\:\text{ (Using Lemma 4.9, \cite{BKV2})}\\
         &=\left(h_{\sigma'}* \delta_{G_{\sigma', r+}}\right)* \delta_{G_{x,(r+s)+}}=h_{\sigma'}*\delta_{G_{x,(r+s)+}}       
    \end{align*}
    which finishes the proof. 
\end{proof}
Our next proposition is a generalization of Proposition 4.14 (a) in \cite{BKV2}, and the main technical result used in the proof of Theorem \ref{thm stab frac}.
\begin{proposition}\label{prop satb frac}
    Let $x\in V(\B_m)$, $r,s \in \frac{1}{m}\Z_{\geq 0}$ and let $\Sigma,\: \Sigma' \in \Theta_m$ be such that $x\in \Sigma' \subseteq \Sigma$ and $\Upsilon_{x,s} \subseteq \Sigma'$. Let $h=\{ h_\sigma\}_{\sigma\in [\Bar{\mcC}_m]}\in A^r(G)$ and $[A^{\Sigma}_h]$ be as defined in \eqref{averaging defn}. 
    \begin{equation}
        [A^{\Sigma}_h]* \delta_{G_{x,(r+s)+}}=[A^{\Sigma'}_h]*\delta_{G_{x,(r+s)+}}
    \end{equation}
\end{proposition}
\begin{proof}
    Set $\Sigma''=\Sigma \setminus \Sigma'$. Then our assertion reduces to proving $[A^{\Sigma''}_h]*\delta_{G_{x,(r+s)+}}=0$. We can define an equivalence relation on $\Sigma''$ by $\sigma_1 \sim \sigma_2 \Leftrightarrow m_{x,s}(\sigma_1)=m_{x,s}(\sigma_2)$.  Then $\Sigma''$ decomposes as a disjoint union of equivalence classes $\Sigma'' = \bigsqcup {\Sigma''}_\sigma$, where ${\Sigma''}_\sigma\subseteq \Sigma''$ denotes the equivalence class of $\sigma \in \Sigma''$. Since $m_{x,s}(\tau) \preceq \tau$ and $\tau \in \Gamma_s(m_{x,s}(\tau),x)$ by definition, using Lemma \ref{lemma stab frac}, we have $$h_{\tau}* \delta_{G_{x,(r+s)+}}=h_{m_{x,s}(\tau)}*\delta_{G_{x,(r+s)+}}$$
    for every $\tau \in [\B_m]$. Since $m_{x,s}(\tau)=m_{x,s}(\sigma)\: \forall \: \tau \in {\Sigma''}_\sigma$, we have  \[
    [A^{{\Sigma''}_\sigma}_h]*\delta_{G_{x,(r+s)+}}=\left(\sum_{\tau \in {\Sigma''}_\sigma}(-1)^{\text{dim}\: \tau}\right)\left(h_{m_{x,s}(\sigma)}*\delta_{G_{x,(r+s)+}}\right)
    \]
    From the proof of Proposition 4.14 (a), we see that $\sum_{\tau \in {\Sigma''}_\sigma}(-1)^{\text{dim}\: \tau}=0$, which proves our assertion. 
\end{proof}
Using the results stated above, we can complete the proof of Theorem \ref{thm stab frac}. 
\begin{proof}[ Proof of Theorem \ref{thm stab frac}] 
    Let $x\in V(\B_m) \cap [\bar{\mcC}_m]$. Then, for any $f \in \HH(G)$, $\exists s \in \frac{1}{m}\Z_{\geq 0}$ such that $f$ is left $G_{x,(r+s)+}$-invariant, i.e., $f= \delta_{G_{x,(r+s)+}}*f$, and hence we have 
    \[
    [A^\Sigma_h] * f=[A^\Sigma_h]* \delta_{G_{x,(r+s)+}}*f
    \]
    We choose $\Sigma' \in \Theta_m$ such that $x\in \Sigma'$ and $\Upsilon_{x,s} \subseteq \Sigma'$. From proposition \ref{prop satb frac}, we observe that 
    \[
    [A^{\Sigma}_h]* \delta_{G_{x,(r+s)+}}=[A^{\Sigma'}_h]*\delta_{G_{x,(r+s)+}}
    \]
    for large enough $\Sigma \in \Theta_m$ such that $\Sigma' \subseteq \Sigma$, since $\Upsilon_{x,s}$ is finite by \cite[ Lemma 4.4]{BKV2}. Hence for $\Sigma \in \Theta_m$ such that $\Sigma' \subseteq \Sigma$, we have 
   \[
   [A^\Sigma_h] * f=[A^\Sigma_h]* \delta_{G_{x,(r+s)+}}*f=[A^{\Sigma'}_h]*\delta_{G_{x,(r+s)+}}*f=[A^{\Sigma'}_h]*f
   \]
   which proves that $\{[A^\Sigma_h] * f\}_{\Sigma \in \Theta_m} $ stabilizes.
\end{proof}

\subsection{A limit description of the fractional depth center}
Since $\lim_{\Sigma \in \Theta_m} \:[A^\Sigma_h] * f $ is well-defined, we can define 
$[A_h] \in \text{End}_{\HH(G)^{op}}(\HH(G))$ by the formula         
\begin{equation}
    [A_h](f):=\lim_{\Sigma \in \Theta_m} \:[A^\Sigma_h] * f.
\end{equation}
\begin{remark}
    For $h=\delta_r \in A^r(G)$, we have from \cite{BKV2} that $[A_{\delta_r}]$ is the projector to the depth-$r$ part of the Bernstein center.
\end{remark}
\begin{proposition}\label{prop eval frac}
    For every $r\in \frac{1}{m}\Z_{\geq 0}$, $\Sigma \in \Theta_m$, $\sigma \in \Sigma$ and $h \in A^r(G)$, we have 
    \begin{equation}
        [A^\Sigma_h]*\delta_{G_{\sigma,r+}}=h_\sigma. 
    \end{equation}
    \end{proposition}
    \begin{proof}
        Choose $x \in V(\B_m)$ such that $x\preceq \sigma$. Then $G_{x,r+}\subseteq G_{\sigma,r+}$, and we have $\delta_{G_{x,r+}} * \delta_{G_{\sigma,r+}}= \delta_{G_{\sigma,r+}}$ and $h_x* \delta_{G_{\sigma,r+}} = h_\sigma$. Hence, it is enough to show that $[A^\Sigma_h]*\delta_{G_{x,r+}}=h_x$. Since $\Upsilon_{x,0}=\emptyset$ by \cite[Lemma 4.4]{BKV2}, the subcomplex $\Sigma'=\{x\}$ satisfies the assumptions of Proposition \ref{prop satb frac} with $s=0$. Hence, we have 
    \[
    [A^\Sigma_h]*\delta_{G_{x,r+}}= [A^{\{x\}}_h]*\delta_{G_{x,r+}}=h_x*\delta_{G_{x,r+}}=h_x
    \]
    and we are done.
    \end{proof}
\begin{remark}\label{remarkevalfrac}
    Note that the above proposition implies that for $h \in A^r(G)$ and $\forall \:\sigma \in [\B_m]$, we have $[A_h](\delta_{G_{\sigma,r+}})=h_\sigma$. 
\end{remark}
\begin{theorem}\label{thm2 frac}
      For each $h\in A^r(G)$, we have $[A_h]\in \mc Z^r(G)   \subset \mc Z(G) \simeq  \emph{End}_{\HH(G)^{2}}(\HH(G))  $, and the assignment $h\mapsto [A_h]$ defines an algebra map 
        \begin{align*}
            [A^r]\: :\: A^r(G) &\longrightarrow \mc Z^r(G) \\
            h &\longmapsto [A_h] 
        \end{align*} 
\end{theorem}
\begin{proof}
    We first show that $[A_h]\in \text{End}_{\HH(G)^{2}}(\HH(G))$, and then prove that the map defined in an algebra homomorphism. For the first part,  it can be easily observed from the definition of $[A_h]$ that $[A_h](f*g)=[A_h](f)*g$ for $h\in A^r(G)$. Since, $[A_h]$ commutes with right convolutions, it is enough to show that it is $G(k)$-conjugation equivariant.\par
    Let $L$ be a compact open subgroup of $ G(k)$. Given any $ l \in L$ and $f \in \HH(G)$, we can choose $\Sigma \in \Theta_m$ large enough such that $[A_h](f) = A^\Sigma_h*f$ and $[A_h](\Ad_l(f)) = A^\Sigma_h*\Ad_l(f)$, since $A^\Sigma_h*f$ stabilizes $\forall \: f \in \HH(G)$.  Then, we can choose the $\Sigma \subseteq \Sigma^L \in \Theta_m$ such that $\Sigma^L$ is Ad $L$-invariant. In that case,$A^{\Sigma^L}_h$ is Ad $L$-invariant and we have 
    \[
    \Ad_l([A_h](f))=\Ad_l(A^{\Sigma^L}_h*f)=\Ad_l(A^{\Sigma^L}_h)*\Ad_l(f)=A^{\Sigma^L}_h*\Ad_l(f)=[A_h](\Ad_l(f)).
    \]
    So, we have shown that $[A_h]$ is Ad $L$-invariant for any compact open subgroup $L \subseteq G(k)$.
    By functoriality of buildings (\cite[Theorem 2.1.8]{Landvogt}), the natural projection $p:G \rightarrow G^{\text{ad}}$ induces a bijection $\B(G)\rightarrow\B(G^{\text{ad}})$ which is compatible with $G(k)$-action on the left hand side and $G^{\text{ad}}(k)$-action on the right hand side.  In particular, this means $G(k)$ acts on $\B(G)$ via $G^{\text{ad}}(k)$, i.e., $g\cdot x = p_k(g)\cdot x$ for $g\in G(k),\: x \in \B(G)=\B(G^{\text{ad}})$. So, it is enough to prove that $[A_h]$ is $G^{\text{ad}}(k)$-conjugation invariant. Since $G^{\text{ad}}(k)$ is generated by compact open subgroups (\cite{BKV2}, 6.1), we have that $[A_h]$ is $G(k)$-conjugation invariant and hence $[A_h]\in \text{End}_{\HH(G)^{2}}(\HH(G))$.\par
    Given $h,h' \in A^r(G)$ and $f \in \HH(G)$, we will show that $[A_{h*h'}](f)=([A_h]\circ [A_{h'}])(f)$.  Choose $\Sigma \in \Theta_m$ large enough such that $[A_{h*h'}](f)=A^\Sigma_{h*h'}*f$ and $[A_{h'}](f)=A^\Sigma_{h'}*f$. Then, $([A_h]\circ [A_{h'}])(f)=[A_h]([A_{h'}](f))=[A_h](A^\Sigma_{h'}*f)=[A_h](A^\Sigma_{h'})*f$, and it is enough to show that $A^\Sigma_{h*h'}=[A_h](A^\Sigma_{h'})$. Now, using the fact that $[A_h]$ is linear and the definitions $A^\Sigma_{h*h'}=\sum_{\sigma \in \Sigma }(-1)^{\text{dim} \: \sigma} h_\sigma *{h'}_\sigma$, $A^\Sigma_{h'}= \sum_{\sigma \in \Sigma }(-1)^{\text{dim} \: \sigma}{h'}_\sigma$, we have 
\begin{align*}
    [A_h](A^\Sigma_{h'})&=\sum_{\sigma \in \Sigma}(-1)^{\text{dim} \: \sigma}[A_h]({h'}_\sigma)=\sum_{\sigma \in \Sigma}(-1)^{\text{dim} \: \sigma}[A_h](\delta_{G_{\sigma, r+}}*{h'}_\sigma)\\
    &=\sum_{\sigma \in \Sigma}(-1)^{\text{dim} \: \sigma}[A_h](\delta_{G_{\sigma, r+}})*{h'}_\sigma=\sum_{\sigma \in \Sigma}(-1)^{\text{dim} \: \sigma}h_\sigma*{h'}_\sigma
\end{align*}
    which shows that $[A_h]$ is an algebra map. Further, for $h\in A^r(G)$, $$[A_h]=[A^r](h)= [A^r](h*\delta_r)= [A_{h*\delta_r}]= [A_h]\circ [A_{\delta_r}]$$
    and hence $[A_h]\in \mc Z^r(G)$ since $[A_{\delta_r}]$ is the depth-$r$ projector.
    
\end{proof}
\begin{theorem}\label{thm: isominverselimit}
    The map $[A^r]: A^r(G) \rightarrow \ZZ^r(G) $ defined in Theorem \ref{thm2 frac} is an algebra isomorphism onto the depth-$r$ Bernstein centre.
\end{theorem}
\begin{proof}
    We prove the assertion in two steps. We first construct a section of this map and then show that it is actually an inverse algebra map. 
    \begin{claim}
         We have an algebra map $\Psi^r \::\: \mc Z^r(G) \longrightarrow A^r(G)$ such that $\Psi^r \circ [A^r] = Id_{A^r(G)}$
    \end{claim}
    For $\sigma \in [\bar{\mcC}_m]$, we define 
    \begin{align*}
        \Psi^r_\sigma \::\: \mc Z^r(G) &\longrightarrow \M^r_\sigma\\
        z &\longmapsto z_\HH(\delta_{G_{\sigma, r+}})
    \end{align*}
    We easily see that $\Psi^r_\sigma$ is well-defined since $G_{\sigma, r+}$ is normal in $G_{\sigma,0}$. Let $\pi^r_\sigma:A^r(G) \rightarrow \M^r_\sigma $ be the canonical projection map. For $\sigma' \preceq \sigma \in [\bar{\mcC}_m]$ we have $\Psi^r_\sigma=  \phi^r_{\sigma',\sigma}\circ\Psi^r_{\sigma'}$. So, there exists a map $\lim_{\sigma \in [\bar{\mcC}_m]}\Psi^r_\sigma =:\Psi_r : \mc Z^r(G) \longrightarrow A^r(G)$ such that $\pi^r_\sigma \circ \Psi_r=\Psi^r_\sigma$ for $\sigma \in [\bar{\mcC}_m]$. Also note that each $\Psi^r_\sigma$ is an algebra map. Given $z,z' \in \ZZ^r(G)$,
    \begin{align*}
        (z\circ z')_\HH(\delta_{G_{\sigma, r+}}) &= z_\HH \circ {z'}_\HH(\delta_{G_{\sigma, r+}})=z_\HH({z'}_\HH(\delta_{G_{\sigma, r+}}*\delta_{G_{\sigma, r+}}))\\
        &= z_\HH(\delta_{G_{\sigma, r+}}*{z'}_\HH(\delta_{G_{\sigma, r+}})) =z_\HH(\delta_{G_{\sigma, r+}})*{z'}_\HH(\delta_{G_{\sigma, r+}})
    \end{align*}
    Hence, $\Psi^r : \ZZ^r(G) \rightarrow A^r(G)$ is an algebra map.  Finally, for $h=\{ h_\sigma\}_{\sigma\in [\Bar{\mcC}_m]}\in A^r(G)$, we have using Remark \ref{remarkevalfrac} that $\Psi^r_\sigma([A_h])= [A_h](\delta_{G_{\sigma,r+}})= h_\sigma$ $\forall \: \sigma \in [\B_m]$. Hence, $\Psi^r \circ [A^r] (h) = \Psi^r([A_h])=h$, which finishes the proof of this claim. \par
    Note that the above claim implies that $\Psi^r$ is surjective and $[A^r]$ is injective as algebra maps. So, if we can show injectivity of $\Psi^r$, we can conclude that $\Psi^r$ and $[A^r]$ are inverse algebra isomorphisms.
    \begin{claim}
        $\Psi^r$ is injective. 
    \end{claim}
     Assume $\Psi^r (z) = \Psi^r (z')$ for $z,\:z' \in \ZZ^r(G)$ and let $(\pi,\; V) \in \Irr(G)_{\leq r}$ .$\exists \:\sigma\in [\Bar{\mcC}_m] $ such that $V^{G_{\sigma, r+}}\ni v \neq \{0\}$. Then $\delta_{G_{\sigma, r+}}(v)=v$. In order to show $z=z'$, it is enough to show $z_V(v) = {z'}_V(v)$. Note that $z_V(v)=z_V(\delta_{G_{\sigma, r+}}(v))=z_\HH(\delta_{G_{\sigma, r+}})(v)$ and the same is true for $z'$. Since $\Psi^r (z) = \Psi^r (z')$, we have 
    \[
    \pi^r_\sigma \circ \Psi^r (z) = \pi^r_\sigma \circ \Psi^r(z') \Rightarrow \Psi^r_\sigma(z)=\Psi^r_\sigma(z')\Rightarrow z_\HH(\delta_{G_{\sigma, r+}})= {z'}_\HH(\delta_{G_{\sigma, r+}})
    \]
    $\forall \;\sigma\in [\Bar{\mcC}_m]$. Hence,
    \[
    z_V(v) = z_\HH(\delta_{G_{\sigma, r+}})(v)= {z'}_\HH(\delta_{G_{\sigma, r+}})(v)= {z'}_V(v)
    \]
    which proves injectivity of $\Psi^r$ and gives us an isomorphism. 
\end{proof}

We have isomorphisms $[A^r] : A^r(G) \longrightarrow \mc Z^r(G)$ for all $r\in \Q_{\geq 0}$. For any $r,\;s\in \Q_{\geq 0}, \:r>s,\: r\in \frac{1}{m}\Z,\; s\in \frac{1}{n}\Z$, we have a map 
\begin{align*}
            e_{r,s} \:: \:A^{r}(G) &\longrightarrow  A^s(G) \\
            \{ h_\sigma\}_{\sigma\in [\Bar{\mcC}_m]} &\longmapsto \{ h'_\tau\}_{\tau\in [\Bar{\mcC}_n]}
        \end{align*} 
    where $ h'_\tau= h_\sigma *\delta_{G_{\tau,s+}}$ for any $\sigma \in [\Bar{\mcC}_m]$ such that $\tau\cap \sigma \neq \emptyset$. Let $A_\tau=\{\sigma \in [\Bar{\mcC}_m] \:|\: \sigma \cap \tau \neq \emptyset\}$ and $A^{max}_\tau=\{\sigma \in A_\tau\:|\: \sigma \text{ of maximal dimension }  \}$. We claim that $h_\sigma *\delta_{G_{\tau,s+}}=h_{\sigma'} *\delta_{G_{\tau,s+}}$ for all $\sigma,\: \sigma'\in A_\tau$ and hence the map $e_{r,s}$ is well-defined. If $\sigma, \sigma'\in  A^{max}_\tau$ with $x\in \sigma \cap \tau,\: x' \in \sigma' \cap \tau $, then without loss of generality we can assume that there exists $\Tilde{\sigma}$ such that $\Tilde{\sigma} \preceq \sigma$ and $\Tilde{\sigma} \preceq \sigma'$, and hence $\Tilde{\sigma} \in A_\tau$. In that case,
    \[
    h_{\sigma}*\delta_{G_{\tau,s+}}=h_{\Tilde{\sigma}}*\delta_{G_{\sigma,r+}}*\delta_{G_{\tau,s+}}=h_{\Tilde{\sigma}}*\delta_{G_{x,r+}}*\delta_{G_{x,s+}}=h_{\Tilde{\sigma}}*\delta_{G_{x,s+}}=h_{\Tilde{\sigma}}*\delta_{G_{\tau,s+}}
    \]
and similarly $h_{\sigma'}*\delta_{G_{\tau,s+}}=h_{\Tilde{\sigma}}*\delta_{G_{\tau,s+}}=h_{\sigma}*\delta_{G_{\tau,s+}}$. Now, if $\sigma \in A_\tau$, there exists $\Tilde{\sigma} \in A^{max}_\tau$ with $\Tilde{x}\in \Tilde{\sigma}\cap \tau $ such that $\sigma \preceq \Tilde{\sigma}$ which gives us
\[
h_{\sigma}*\delta_{G_{\tau,s+}}=h_{\sigma}*\delta_{G_{\Tilde{x},s+}}=h_{\sigma}*\delta_{G_{\Tilde{x},r+}}*\delta_{G_{\Tilde{x},s+}}=h_{\sigma}*\delta_{G_{\Tilde{\sigma},r+}}*\delta_{G_{\Tilde{x},s+}}=h_{\Tilde{\sigma}}*\delta_{G_{\Tilde{x},s+}}=h_{\Tilde{\sigma}}*\delta_{G_{\tau,s+}}
\]
and hence we are done. We define
\begin{equation}
    A(G) := \lim_{r\in \Q_{\geq 0}}A^r(G) = \lim_{r\in \Z_{\geq 0}}A^r(G) 
\end{equation}
where the second limit is taken with respect to the maps $e_{r+1,r}$.\par
Further, note that we have a natural map 
\begin{align*}
            z_{r,s} \:: \:\ZZ^{r}(G) &\longrightarrow  \ZZ^s(G) \\
            z &\longmapsto z \circ [A_{\delta_s}]
        \end{align*} 
        such that $\ZZ(G) = \lim_{r\in \Q_{\geq 0}}\ZZ^r(G)=\lim_{r\in \Z_{\geq 0}}\ZZ^r(G)$. 

        \begin{theorem}\label{main theorem}
            Let $r,s$ be as in the preceding paragraphs. The algebra isomorphisms $[A^r] : A^r(G) \rightarrow \mc Z^r(G)$ fit into the following commutative diagram 
             \[
             \xymatrix{A^{r}(G) \ar[r]^{[A^{r}]}\ar[d]^{e_{r,s}}&\ZZ^{r}(G)\ar[d]^{z_{r,s}}\\
        A^s(G)\ar[r]^{[A^s]}&\ZZ^s(G)}
        \]
         In particular, we have an algebra isomorphism 
        \begin{equation}
            [A]=\lim_{r\in \Z_{\geq 0}}[A^r] : A(G) \longrightarrow \ZZ(G).
        \end{equation}
        \end{theorem}
        \begin{proof}
            Let $\{ h_\sigma\}_{\sigma\in [\Bar{\mcC}_m]}\in A^r(G)$ and $h'=e_{r,s}(h)=\{h_{\sigma_\tau}*\delta_{G_{\tau,s+}}\}_{\tau\in [\Bar{\mcC}_n]}$ where $\sigma_\tau \in A_\tau$. We want to show that $[A_{h'}]=[A_h]\circ [A_{\delta_s}]$. Let $(\pi, V)$ be a smooth irreducible representation of $G(k)$. If the depth of $\pi$ is $\leq s$, then  $\exists \: 0\neq v\in V^{G_{\tau,s+}}$, $\tau \in [\Bar{\mcC}_n]$. In this case,
            \[
            [A_{h'}](v)= [A_{h'}](\delta_{G_{\tau,s+}}(v))=[A_{h'}](\delta_{G_{\tau,s+}})(v)=(h_{\sigma_\tau}*\delta_{G_{\tau,s+}})v=h_{\sigma_\tau}(v).
            \]
            Further, fix $\sigma_\tau \in A_\tau$ and let $x \in \sigma_\tau \cap \tau$. Then $v\in V^{G_{x,s+}}\supset V^{G_{x,r+}}$ since $r>s$, and 
            \[
            [A_h]\circ [A_{\delta_s}](v)=[A_h](v)=[A_h](\delta_{G_{x,r+}}(v))=[A_h](\delta_{G_{x,r+}})v=[A_h](\delta_{G_{\sigma_\tau,r+}})v= h_{\sigma_\tau}(v).
            \]
            So, we have $[A_{h'}]|_V=([A_h]\circ [A_{\delta_r}])|_V$ when the depth of $\pi$ is less than $s$. \par
            When depth of $\pi$ is $>s$, $[A_{h'}]|_V=0=([A_h]\circ [A_{\delta_s}])|_V$. The rest follows immediately.
        \end{proof}
        
\section{Stable functions on positive depth Moy-Prasad quotients}\label{section:stablefnsonpositivedepth}
 For $\sigma \in [\B_m]$, let $\G_\sigma$ be the connected reductive $\bar{\F}_q$- group defined over $\F_q$ with associated Frobenius $F$ such that $\G_\sigma(\Bar\F_q) =G(K)_{\sigma,0}/G(K)_{\sigma,0+}$. We know that $\G_{\sigma}$ is the reductive quotient of the special fibre of the parahoric group scheme corresponding to $\sigma$. Further, let $\bg_\sigma= \text{Lie}(\G_\sigma)$ and $\bg_{\sigma, r}$ be the $\bar{\F}_q$ vector space such that $\bg_{\sigma, r}(\Bar\F_q)=G(K)_{\sigma,r}/G(K)_{\sigma,r+}$ for $r \in \Q_{>0}$. It has a $\F_q$- structure induced by $G_{\sigma,r}/G_{\sigma,r+}\cong \bg_{\sigma,r}(\F_q)$. Note that $ \bg_{\sigma, r}\cong \bg_\sigma$ for $r \in \Z_{>0}$. The action of $G_{\sigma,0}/G_{\sigma,0+}$ on $G_{\sigma,r}/G_{\sigma,r+}$ for $r \in \Q_{>0}$ is the $\F_q$-points of a linear algebraic action of $\G_\sigma$ on $\bg_{\sigma,r}$. For the special case where $r \in \Z_{>0}$, this is isomorphic to the $\F_q$-points of the adjoint action of $\G_\sigma$ on its Lie algebra, and we studied stable functions on these Moy-Prasad quotients at positive integral depth in \cite{cb24}. The goal of this section is to study stable functions more generally allowing fractional depth quotients as well. 
\subsection{Fractional depth quotients}
Fix $m\in \Z_{>0}$ and let $0<r \in \frac{1}{m}\Z_{>0}$. We will define an analogue of Iwahori decomposition following \cite{moysavin}. Let $\sigma \in [\B_m]$ and $\A$ be an apartment containing $\sigma$. We define a set $\Phi_\sigma$
\begin{equation}\label{defnphi_sigma}
    \Phi^m_\sigma=\Phi_\sigma:=\{\nabla\psi\:|\: \psi \in \Psi_m(\A),\: \psi=r \text{ on }\sigma\}
\end{equation}
This is a root-subsystem of $\Phi$, and is independent of $r$ (depends only on $m$). For $\sigma, \; \tau \in [\B_m]$ such that $\sigma \preceq \tau$, $\Phi_\tau\subset \Phi_\sigma$ is a root subsystem. 
Some specific examples of $\Phi_\sigma$:
\begin{itemize}
    \item If the facet $\sigma$ is a hyperspecial point, $\Phi_\sigma=\Phi$.
    \item If $\sigma$ is a chamber, $\Phi_\sigma=\emptyset$
\end{itemize}
Observe that if $m,n \in \Z_{>0}, \; m|n$ and for $\sigma' \in [\B_n]$, there exists unique $\sigma \in [\B_m]$ such that $\sigma' \subset \sigma$ and $\Phi_{\sigma'}^n\supset \Phi_\sigma^m=:\Phi_{\sigma'}^m$ and $\Phi_{\sigma'}^m $ defined this way is well-defined and a root subsystem of $\Phi_{\sigma'}^n$. In particular, $\Phi^1_{\sigma'} \subset \Phi^n_{\sigma'}$ for all $n \in \Z_{>0}$. Further, for $x \in \A\subset \B$, we can define $\Phi^n_x$ the following way-
\begin{equation}
    \Phi_x^n= \{\nabla\psi\:|\: \psi \in \Psi_n(\A),\: \psi(x)=\frac{1}{n}\}
\end{equation}
Defined this way, we have $\Phi_{\sigma'}^n=\cap_{x \in \sigma'}\Phi_x^n$ for $\sigma'\in [\A_n] \subset [\B_n]$. \par
For $\sigma \in [\B_m]$ contained in some apartment $\A$, let $\text{Aff}(\sigma)$ denote the affine space generated by $\sigma$ in $\A$. If $X$ is a real affine space and $Y$ is an affine subspace of $X$, there is a unique affine orthogonal transformation $R_Y$ which reflects points of $X$ across $Y$. 
\begin{definition}
    Let $\sigma \in [\B_m]$. Two refined facets $\tau, \tau' \in [\B_m]$ such that $\sigma \preceq\tau,\: \sigma\preceq \tau'$ are said to be opposite with respect to $\sigma$ (or $\sigma$-opposite) if there exists an apartment $\A$ containing both $\tau $ and $\tau'$ such that in $\A$:
    \begin{itemize}
        \item [(i)] The affine subspaces Aff$(\tau)$ and Aff$(\tau')$ are equal. 
        \item[(ii)] The reflection $R_{\text{Aff}(\sigma)}(\tau)=\tau'$ 
    \end{itemize}
\end{definition}
Let $\tau$ and $\tau'$ be $\sigma$-opposite and $\tau,\;\tau' \subset \A$. We immediately observe that if $\psi \in \Psi_m(\A)$ such that $\psi =r $ on $\tau$, then $\psi=r$ on $\tau'$ and vice-versa. Hence, $\Phi_\tau=\Phi_{\tau'}$. Further, $G_{\tau,r}/G_{\tau,r+}\cong G_{\tau',r}/G_{\tau',r+}$. For $\psi \in \Psi_m(\A)$ such that $\psi=r$ on $\sigma$, there are two possibilities for $\nabla\psi$:
\begin{itemize}
    \item [(i)] $\psi=r$ on $\tau $ and $\tau'$, and hence $\nabla\psi \in \Phi_\tau=\Phi_{\tau'}$.
    \item[(ii)] $\psi$ is non-constant on $\tau$ (and $\tau'$), which presents us with two possibilities. $\psi >r $ on $\tau$ and $\psi<r$ on $\tau'$ and we denote these by $\Phi_{\sigma,\tau}$, and $\psi>r$ on $\tau'$ and $\psi<r$ on $\tau$, which we denote by $\Phi_{\sigma,\tau'}$. This basically gives us a choice of positive roots in the root system $\Phi_{\sigma}$, and $\Phi_{\sigma,\tau}=-\Phi_{\sigma,\tau'}$. Further, $\Phi_\sigma$ is a disjoint union
    \begin{equation}
        \Phi_\sigma = \Phi_\tau\;\sqcup \Phi_{\sigma,\tau}\;\sqcup \Phi_{\sigma,\tau'}
    \end{equation}
\end{itemize}
For  $\psi \in \Psi_m(\A)$, the quotient $\mathscr L_\psi :=U_\psi(K)/U_{\psi +}(K)$ has a natural structure of a $\bar\F_q$-vector space, with $\F_q$-structure given by $\mathscr L_\psi(\F_q)=U_\psi(k)/U_{\psi +}(k)$. So, we have $\mathscr L_\psi = \mathscr L_\psi (\F_q) \otimes_{\F_q}\bar\F_q$. \paragraph{}
For $\alpha \in \Phi \cup \{0\}$, $\sigma \in [\A_m]$ and $0<r \in \frac{1}{m}\Z_{>0} $, we define 
\begin{equation*}
    \psi_{\alpha,\sigma,r}:= \text{ The smallest } \psi \in \Psi_m(\A) \text{ with } \nabla\psi=\alpha \text{ such that } \psi(x)\geq r \:\forall\: x \in \sigma 
\end{equation*}
\begin{equation*}
    \psi_{\alpha,\sigma,r+}:= \text{ The smallest } \psi \in \Psi_m(\A) \text{ with } \nabla\psi=\alpha \text{ such that } \psi(x)> r \:\forall\: x \in \sigma 
\end{equation*}
Let $S$ be a $k$-split maximal torus such that $\A= \A(S,k)$ (cf. Corollary 7.6.9, \cite{KP23}) and let $\fs= \Lie (S)$. Then, from Proposition 13.2.5 in \cite{KP23}, we have 
\begin{equation*}
    G_{\sigma,r}= S(k)_r \times \prod_{\alpha \in \Phi} U_{\psi_{\alpha,\sigma,r}}(k)  
\end{equation*}
and 
\begin{equation*}
    G_{\sigma,r+}= S(k)_{r+} \times \prod_{\alpha \in \Phi }U _{\psi_{\alpha,\sigma,r+}}(k) 
\end{equation*}
Note that for $\alpha \not\in \Phi_\sigma $, $\psi_{\alpha,\sigma,r}=\psi_{\alpha,\sigma,r+}$, and hence we have 
\begin{equation}\label{quotientstructure}
    G_{\sigma,r}/G_{\sigma,r+}= \left(S(k)_r/S(k)_{r+}\right) \times \prod_{\alpha \in \Phi_\sigma } U _{\psi_{\alpha,\sigma,r}}(k)/U _{\psi_{\alpha,\sigma,r}+}(k)
\end{equation}
Using the above-mentioned ideas, we have the following proposition about the structure of the Moy-Prasad filtration quotient $G_{\sigma,r}/G_{\sigma,r+}$, which is similar to Proposition 2.5.4, \cite{moysavin}.

\begin{proposition}\label{Iwahorifracdepthprop}
    Let $\sigma \in [\B_m]$ be a refined facet, and $\tau,\;\tau'\in [\B_m]$ be two refined facets which are $\sigma$-opposite. Then, we have (Iwahori decomposition):
    \begin{equation}\label{Iwahoridecomp}
      G_{\sigma,r}/G_{\sigma,r+} = G_{\tau,r}/G_{\tau,r+}\oplus G_{\tau,r+}/G_{\sigma,r+}\oplus G_{\tau',r+}/G_{\sigma,r+}  
    \end{equation}
    Further, we have a similar decomposition for the Moy-Prasad quotients of the Lie algebra
     \begin{equation}\label{Iwahoridecomplie}
      \fg_{\sigma,r}/\fg_{\sigma,r+} = \fg_{\tau,r}/\fg_{\tau,r+}\oplus \fg_{\tau,r+}/\fg_{\sigma,r+}\oplus \fg_{\tau',r+}/\fg_{\sigma,r+}  
    \end{equation}
    and the Moy-Prasad isomorphism (Theorem 13.5.1, \cite{KP23} )
    \begin{equation}\label{MPisom}
         G_{\sigma,r}/G_{\sigma,r+} \cong \fg_{\sigma,r}/\fg_{\sigma,r+}
    \end{equation}
    induces isomorphisms 
    \begin{equation}\label{iwahoriunipotentpart}
        G_{\tau,r+}/G_{\sigma,r+} \cong \fg_{\tau,r+}/\fg_{\sigma,r+} \:\:\text{  and   } \: \:G_{\tau',r+}/G_{\sigma,r+} \cong \fg_{\tau',r+}/\fg_{\sigma,r+}
    \end{equation}
\end{proposition}
\begin{proof}
    From \eqref{quotientstructure}, we see that 
    \[
    G_{\sigma,r}/G_{\sigma,r+}= \left(S(k)_r/S(k)_{r+}\right) \times \prod_{\alpha \in \Phi_\sigma } U _{\psi_{\alpha,\sigma,r}}(k)/U _{\psi_{\alpha,\sigma,r}+}(k).
    \]
For $\alpha \in \Phi_\tau$, we have $\psi_{\alpha,\sigma,r}=\psi_{\alpha,\tau,r}$ and $\psi_{\alpha,\sigma,r+}=\psi_{\alpha,\tau,r+}$. Further, for $\alpha \in \Phi_{\sigma,\tau}$, $\psi_{\alpha,\sigma,r}=\psi_{\alpha,\tau,r+}$ and for $\alpha \in \Phi_{\sigma,\tau'}$, $\psi_{\alpha,\sigma,r}=\psi_{\alpha,\tau',r+}$. This gives \eqref{Iwahoridecomp}. Further, from the proof of Theorem 13.5.1, \cite{KP23}, we note that there exist isomorphisms $U_{\psi_{\alpha,\sigma,r}}(k)/U _{\psi_{\alpha,\sigma,r}+}(k) \xrightarrow{\simeq} \fu_{\psi_{\alpha,\sigma,r}}(k)/\fu_{\psi_{\alpha,\sigma,r}+}(k)$ and $S(k)_r/S(k)_{r+}\xrightarrow{\simeq} \fs(k)_{r}/\fs(k)_{r+} $. This gives the triangular decomposition in \eqref{Iwahoridecomplie} and the isomorphisms in \eqref{iwahoriunipotentpart}. 
\end{proof}

Note that each component of the direct sum in \eqref{Iwahoridecomp} is a $\F_q$-vector space. We have a similar decomposition of $\kappa_K=\bar\F_q$-vector spaces 
\begin{equation}\label{Iwahoriunramified}
    G(K)_{\sigma,r}/G(K)_{\sigma,r+} = G(K)_{\tau,r}/G(K)_{\tau,r+}\oplus G(K)_{\tau,r+}/G(K)_{\sigma,r+}\oplus G(K)_{\tau',r+}/G(K)_{\sigma,r+}  
\end{equation}
Further each of the isomorphisms \eqref{MPisom}, \eqref{iwahoriunipotentpart} are compatible with unramified algebraic extensions of $k$, and hence holds over $K$ as isomorphisms of $\Bar\F_q$-spaces. Let $\bu_{\sigma,\tau,r}$ denote the $\Bar\F_q$-vector space $G(K)_{\tau,r+}/G(K)_{\sigma,r+}\cong \fg(K)_{\tau,r+}/\fg(K)_{\sigma,r+}$, and $\bp_{\sigma,\tau,r}$ denote the $\bar \F_q$-vector space $G(K)_{\tau,r}/G(K)_{\sigma,r+} \cong \bu_{\sigma,\tau,r} \oplus \bg_{\tau,r}$. The $\F_q$-vector space $G_{\tau,r+}/G_{\sigma,r+}$ gives $\bu_{\sigma,\tau,r}$ an $\F_q$-structure, and hence we can denote it by $\bu_{\sigma,\tau,r}^F=\bu_{\sigma,\tau,r}(\F_q)$. Similar statements are also true for $\bp_{\sigma,\tau,r}$ and the other components of the decomposition in \eqref{Iwahoriunramified}, and their corresponding counterparts in \eqref{Iwahoridecomp}.

\subsection{Fourier Transform on positive depth Moy-Prasad quotients}\label{subsection: Fouriertransform}
 Let $\fg=\text{Lie}(G)$ and $\fg^*$ be its dual. Given $x \in \B(G)$, we have the Moy-Prasad filtrations for $\fg(k)$ denoted by $\fg_{x,r}:=\fg(k)_{x,r}$, and that of $\fg^*(k)$ denoted by $\fg^*_{x,-r}:=\fg^*(k)_{x,-r}$. We can define $\fg_{\sigma,r}$ and $\fg^*_{\sigma,-r}$ for $\sigma \in [\B_m]$, $r \in \frac{1}{m}\Z$ similarly to $G_{\sigma,r}$. The $\mf O_k$-module $\fg_{\sigma,r}$ is stable under the adjoint action of $G_{\sigma,r}$, and the action of $G_{\sigma,r+}$ on $\fg_{\sigma,r}/\fg_{\sigma,r+}$ is trivial, which induces an action of $\G_{\sigma}^F$ on $\fg_{\sigma,r}/\fg_{\sigma,r+}$. Similarly, we have an action of $\G_{\sigma}^F$ on $\fg^*_{\sigma,-r}/\fg^*_{\sigma,-r+}$ induced by the co-adjoint action. For $r>0$, the Moy-Prasad isomorphism gives a $\G_\sigma^F$-equivariant isomorphism 
\begin{equation}
    G_{\sigma,r}/G_{\sigma,r+} \cong \fg_{\sigma,r}/\fg_{\sigma,r+}. 
\end{equation}
Further, for $r>0$, the $\F_q$-bilinear map 
\begin{align*}
    \fg^*_{\sigma,-r}/\fg^*_{\sigma,-r+} \times \fg_{\sigma,r}/\fg_{\sigma,r+} &\longrightarrow \F_q\\
    (X,Y)&\longmapsto X(Y) \text{ mod } \varpi \mf O_k
\end{align*}
is a non-degenerate $\G_\sigma^F$-invariant pairing, and gives a $\G_\sigma^F$-equivariant isomorphism between $(\bg_{\sigma,r}^F)^* \cong (\fg_{\sigma,r}/\fg_{\sigma,r+})^*$ and $\fg^*_{\sigma,-r}/\fg^*_{\sigma,-r+}$. \paragraph{}
We develop a theory of Fourier transforms on positive depth Moy-Prasad filtration quotients, mostly following the ideas in \cite{Le} which studies the space of adjoint-invariant functions finite Lie algebras. Our theory is a generalisation, and would reduce to the statements in Section 4 of \cite{Le} if we consider integral depths and use a non-degenerate invariant bilinear form to identify the Lie algebra and its dual. The Fourier transorm in the afore-mentioned paper was used in \cite{cb24} to study stable functions on integral depth quotients. \par
Let $C(\bg_{\sigma,r}^F)$ denote the space of $\G_\sigma^F$-invariant functions on $\bg_{\sigma,r}^F$ equipped with the convolution product
\begin{equation}
    f*f'(X)=\left|\bg_{\sigma,r}^F\right|^{-1/2}\sum_{Y\in \bg_{\sigma,r}^F}f(X-Y)f'(Y).
\end{equation}
and $C\left((\bg_{\sigma,r}^F)^*\right)$ denote the space of  $\G_\sigma^F$-invariant functions on $(\bg_{\sigma,r}^F)^*$ equipped with point-wise multiplication. Note that we can also have a similar convolution product on  $C\left((\bg_{\sigma,r}^F)^*\right)$. For $f,g \in C(\bg_{\sigma,r}^F)$, we can define a positive definite non-singular Hermitian form $(\:,\:)$ on $C(\bg_{\sigma,r}^F)$ by 
\begin{equation*}
    (f,g)= \left|\G_{\sigma}^F\right|^{-1}\sum_{Y\in \bg_{\sigma,r}^F}f(Y)\overline{g(Y)}
\end{equation*}
    We can similarly define a positive definite non-singular Hermitian form on $C\left((\bg_{\sigma,r}^F)^*\right)$. \par
We fix a non-trivial additive character $\Tilde{\psi}:\mathbb {F}_q\to\mathbb C^\times$. For a finite dimensional vector space $V$ over $\F_q$, the choice of a non-trivial character $\Tilde{\psi}$ determines an identification  of the dual space $V^*$ and the Pontryagin dual $\widehat{V}$. 
Denote by $\mathbb C[V]$ (resp. $\mathbb C[V^*]$) the space of complex valued functions on $V$ (resp. $V^*$). We have the Fourier-transform
$\mcF_V:\mathbb C[V]\to \mathbb C[V^*]$
defined by the formula
\begin{equation}\label{FT_V}
\mcF_V(f)(X^*)=|V|^{-1/2}\sum_{Y\in V}\psi(( X^*(Y))f(Y).
\end{equation}
We can define a Fourier Transform $\mcF_{\bg_{\sigma,r}^F}:C(\bg_{\sigma,r}^F) \longrightarrow C\left((\bg_{\sigma,r}^F)^*\right)$ following the idea in \eqref{FT_V}, given by the formula :
\begin{equation}\label{fourierdefnlie}
    \mcF_{\bg_{\sigma,r}^F}(f)(X^*)=\left|\bg_{\sigma,r}^F\right|^{-1/2}\sum_{Y\in \bg_{\sigma,r}^F}\Tilde\psi( X^*(Y))f(Y).
\end{equation}
and a Fourier Transform $ \mcF_{(\bg_{\sigma,r}^F)^*}:C\left((\bg_{\sigma,r}^F)^*\right) \longrightarrow C(\bg_{\sigma,r}^F) $ given by 
\begin{equation}\label{fourierdefndual}
  \mcF_{(\bg_{\sigma,r}^F)^*}(f)(Y)=\left|(\bg_{\sigma,r}^F)^*\right|^{-1/2}\sum_{X^*\in (\bg_{\sigma,r}^F)^*}\Tilde\psi( X^*(Y))f(X^*).
\end{equation}
Let $\sigma\preceq \tau $, $\sigma \preceq \tau'$ such that $\tau$ and $\tau'$ are $\sigma$-opposite. Then from $\eqref{Iwahoridecomp}$ we know 
\begin{equation}\label{Iwahorinotation}
    \bg_{\sigma,r}^F \cong \bu_{\sigma,\tau,r}^F \oplus \bg_{\tau,r}^F \oplus \bu_{\sigma,\tau',r}^F
\end{equation}
as $\F_q$-vector spaces. The projection maps to the subspaces of the Iwahori decomposition give a natural identification of $(\bg_{\tau,r}^F)^*$, $(\bu_{\sigma,\tau,r}^F)^*$ and $(\bu_{\sigma,\tau',r}^F)^*$ as subspaces of $(\bg_{\sigma,r}^F)^*$ and we have 
\begin{equation*}
    (\bg_{\sigma,r}^F)^* \cong (\bu_{\sigma,\tau,r}^F)^* \oplus (\bg_{\tau,r}^F)^* \oplus (\bu_{\sigma,\tau',r}^F)^*
\end{equation*}
For $\sigma\preceq\tau$, we have the normalized parabolic restriction map $\text{Res}^{\bg_{\sigma,r}}_{\bg_{\tau,r}}:C(\bg_{\sigma,r}^F) \rightarrow C(\bg_{\tau,r}^F)$ defined by 
\begin{equation}
    \text{Res}^{\bg_{\sigma,r}}_{\bg_{\tau,r}}(f)(X)=\left|\bu_{\sigma,\tau,r}^F\right|^{-1}\sum_{N \in \bu_{\sigma,\tau,r}^F}f(X+N)
\end{equation}
 and the restriction map for the dual case $\text{Res}^{(\bg_{\sigma,r})^*}_{(\bg_{\tau,r})^*}:C\left((\bg_{\sigma,r}^F)^*\right) \rightarrow C\left((\bg_{\tau,r}^F)^*\right)$ defined by
\begin{equation}
    \Res^{(\bg_{\sigma,r})^*}_{(\bg_{\tau,r})^*}(\tilde f)(X^*)=\left|\bu_{\sigma,\tau',r}^F\right|^{-1}\sum_{N^* \in (\bu_{\sigma,\tau',r}^F)^*}\tilde f(X^*+N^*).
\end{equation} 
 For $\sigma \preceq \tau_1 \preceq \tau_2$, we have transitivity of restriction maps 
 \begin{equation*}
     \text{Res}^{\bg_{\sigma,r}}_{\bg_{\tau_2,r}}= \text{Res}^{\bg_{\tau_1,r}}_{\bg_{\tau_2,r}}\circ \text{Res}^{\bg_{\sigma,r}}_{\bg_{\tau_1,r}} \:\:\text{  and }\:\:\text{Res}^{(\bg_{\sigma,r})^*}_{(\bg_{\tau_2,r})^*}= \text{Res}^{(\bg_{\tau_1,r})^*}_{(\bg_{\tau_2,r})^*}\circ \text{Res}^{(\bg_{\sigma,r})^*}_{(\bg_{\tau_1,r})^*}
 \end{equation*}

 We give some details of the dual case. Let $\tau_1, \tau_1'$ and $\tau_2, \tau_2'$ be $\sigma$-opposite pairs, and $\tau_2, \tau_2''$ be $\tau_1$-opposite. This implies $\tau_1'\preceq \tau_2'$ and $\sigma \preceq \tau_2''$. Using the Iwahori decomposition as in \eqref{Iwahorinotation} and the fact that $\bu_{\sigma,\tau_2',r}^F \cong \bu_{\sigma,\tau_1',r}^F \oplus \bu_{\tau_1',\tau_2',r}^F$, the transitivity of restriction in the dual case would follow if $\bu_{\tau_1',\tau_2',r}^F \cong \bu_{\tau_1,\tau_2'',r}^F$ and hence their duals are isomorphic. We have
\begin{equation*}
 \bu_{\tau_1',\tau_2',r}^F=G_{\tau_2',r+}/{G_{\tau_1',r+}}\cong \prod_{\alpha \in \Phi_{\tau_1',\tau_2'}}  U _{\psi_{\alpha,\tau_2',r+}}(k)/{U _{\psi_{\alpha,\tau_1',r+}}}(k)\cong \prod_{\mathclap{\substack{\alpha \in \Phi_{\tau_1',\tau_2'}\\r-\alpha(\tau_1')\in \Z}}}\mathscr L_{\psi_{\alpha,\tau_2',r+}}(\F_q)
\end{equation*}
and 
\begin{equation*}
\bu_{\tau_1,\tau_2'',r}^F=G_{\tau_2'',r+}/{G_{\tau_1,r+}}\cong \prod_{\alpha \in \Phi_{\tau_1,\tau_2''}}  U _{\psi_{\alpha,\tau_2'',r+}}(k)/{U _{\psi_{\alpha,\tau_1,r+}}}(k)\cong \prod_{\mathclap{\substack{\alpha \in \Phi_{\tau_1,\tau_2''}\\r-\alpha(\tau_1)\in \Z}}}\mathscr L_{\psi_{\alpha,\tau_2'',r+}}(\F_q)
\end{equation*}
We certainly have $\Phi_{\tau_1,\tau_2''}=\Phi_{\tau_1',\tau_2'}$ just from their definitions, and since $\tau_1$ and $\tau_1'$ are $\sigma$-opposite, $r-\alpha(\tau_1')\in \Z\Longleftrightarrow r-\alpha(\tau_1)\in \Z$. Further, for $\alpha \in \Phi_{\tau_1,\tau_2''}$, we have $\psi_{\alpha,\tau_2'',r+}=\psi_{\alpha,\tau_2',r+}$ and hence $\bu_{\tau_1',\tau_2',r}^F \cong \bu_{\tau_1,\tau_2'',r}^F$.\par
Then, for $\tilde f \in C\left( (\bg_{\sigma,r}^F)^*\right) $ and $X^* \in (\bg_{\tau_2,r}^F)^*$
\begin{align*}
    \text{Res}^{(\bg_{\sigma,r})^*}_{(\bg_{\tau_2,r})^*}(\tilde f) (X^*) &= \left|\bu_{\sigma,\tau_2',r}^F\right|^{-1}\tilde f(X^*+N^*)\\
    &= \left|\bu_{\sigma,\tau_1',r}^F \oplus \bu_{\tau_1',\tau_2',r}^F\right|^{-1}\sum_{N^* \in (\bu_{\sigma,\tau_1',r}^F)^* \oplus (\bu_{\tau_1',\tau_2',r}^F)^*}\tilde f(X^*+N^*)\\
    &=\left|\bu_{\sigma,\tau_1',r}^F\right|^{-1}\sum_{N_1^* \in (\bu_{\sigma,\tau_1',r}^F)^*}\left|\bu_{\tau_1,\tau_2'',r}^F\right|^{-1}\sum_{N_2^* \in (\bu_{\tau_1,\tau_2'',r}^F)^*}\tilde f(X^*+N_1^*+N_2^*)\\
    &= \text{Res}^{(\bg_{\sigma,r})^*}_{(\bg_{\tau_1,r})^*}\circ \text{Res}^{(\bg_{\tau_1,r})^*}_{(\bg_{\tau_2,r})^*}(\tilde f)(X^*)
\end{align*}
\par
We can also define an induction map, similar to the Harish-Chandra induction in the Lie algebra case. 
\begin{equation*}
    \Ind^{\bg_{\sigma,r}}_{\bg_{\tau,r}}f(X)= \left|\G_\tau^F\right|^{-1}\left|\bu_{\sigma,\tau,r}^F\right|^{-1}\sum_{{\mathclap{\substack{g \in \G_\sigma^F\\ \prescript{g}{}{} X\in \bp_{\sigma,\tau,r}^F}}}}f^\vee(\prescript{g}{}{} X)
\end{equation*}
where $f^\vee (X+N)= f(X)$ where $X \in \bg^F_{\tau,r}$ and $N\in \bu_{\sigma,\tau,r}^F$. 

\begin{lemma}\label{indandres}
    The maps $ \Res^{\bg_{\sigma,r}}_{\bg_{\tau,r}}$ and $\Ind^{\bg_{\sigma,r}}_{\bg_{\tau,r}}$ are adjoint with respect to the inner products $(\:,\:)_{_{\bg_{\sigma,r}^F}}$ and $(\:,\:)_{_{\bg_{\tau,r}^F}} $ on $C(\bg_{\sigma,r}^F)$ and $C(\bg_{\tau,r}^F)$ respectively.
\end{lemma}
\begin{proof}
    Let $f \in C(\bg_{\sigma,r}^F)$ and $g \in C(\bg_{\tau,r}^F)$. Then, 
    \begin{align*}
        (f,\Ind^{\bg_{\sigma,r}}_{\bg_{\tau,r}}(g))&=\left|\G_{\sigma}^F\right|^{-1}\sum_{Y\in \bg_{\sigma,r}^F}f(Y)\overline{\Ind^{\bg_{\sigma,r}}_{\bg_{\tau,r}}(g)}\\
        &= \left|\G_{\sigma}^F\right|^{-1}\left|\G_\tau^F\right|^{-1}\left|\bu_{\sigma,\tau,r}^F\right|^{-1}\sum_{Y\in \bg_{\sigma,r}^F}f(Y)\sum_{{\mathclap{\substack{h \in \G_\sigma^F\\ \prescript{h}{}{} Y\in \bp_{\sigma,\tau,r}^F}}}}\overline{g^\vee(\prescript{h}{}{} Y)}\\
        &=\left|\G_{\sigma}^F\right|^{-1}\left|\G_\tau^F\right|^{-1}\left|\bu_{\sigma,\tau,r}^F\right|^{-1}\sum_{h \in \G_\sigma^F} \sum_{Y\in \Ad(h^{-1})\bp_{\sigma,\tau,r}^F}f(\prescript{h}{}{} Y)\overline{g^\vee(\prescript{h}{}{} Y)}\\
        &=\left|\G_{\sigma}^F\right|^{-1}\left|\G_\tau^F\right|^{-1}\sum_{h \in \G_\sigma^F} \sum_{Y \in \bg_{\tau,r}^F} \overline{g(Y)}\left(\left|\bu_{\sigma,\tau,r}^F\right|^{-1} \sum_{N\in \bu_{\sigma,\tau,r}^F}f(Y+N)\right)\\
        &= \left(\Res^{\bg_{\sigma,r}}_{\bg_{\tau,r}}(f), g\right)
    \end{align*}
\end{proof}
Similar definition for $\Ind^{(\bg_{\sigma,r})^*}_{(\bg_{\tau,r})^*}$ can be made in the dual case as well, with similar properties. 
\begin{proposition}[Properties of Fourier transform]\label{fourierprop}
    Let  $f,g \in C\left(\bg_{\sigma,r}^F\right)$ and $\tilde f, \tilde g \in C\left((\bg_{\sigma,r}^F)^*\right)$. The Fourier transforms defined in \eqref{fourierdefnlie} and \eqref{fourierdefndual} have the following properties:
    \begin{itemize}
        \item[(i)]  $\left(\mcF_{\bg_{\sigma,r}^F}(f),\mcF_{\bg_{\sigma,r}^F}(g)\right)_{(\bg_{\sigma,r}^F)^*}= \left(f,g\right)_{\bg_{\sigma,r}^F}$. 
        \item [(ii)]  $\mcF_{\bg_{\sigma,r}^F}\circ \mcF_{(\bg_{\sigma,r}^F)^*}(\tilde f) = (\tilde f)^-$ and $ \mcF_{(\bg_{\sigma,r}^F)^*}\circ\mcF_{\bg_{\sigma,r}^F}(f)=f^-$, where $f^-(X)=f(-X)$. 
        \item[(iii)]  $\mcF_{\bg_{\sigma,r}^F}(f*g)= \mcF_{\bg_{\sigma,r}^F}(f)\cdot \mcF_{\bg_{\sigma,r}^F}(g)$ and $\mcF_{(\bg_{\sigma,r}^F)^*} (\tilde f*\tilde g )=\mcF_{(\bg_{\sigma,r}^F)^*} (\tilde f) \cdot \mcF_{(\bg_{\sigma,r}^F)^*}(\tilde g) $.
        \item[(iv)] $\mcF_{\bg_{\sigma,r}^F}(f\cdot g)= \mcF_{\bg_{\sigma,r}^F}(f)* \mcF_{\bg_{\sigma,r}^F}(g)$ and $\mcF_{(\bg_{\sigma,r}^F)^*} (\tilde f\cdot \tilde g )=\mcF_{(\bg_{\sigma,r}^F)^*} (\tilde f) * \mcF_{(\bg_{\sigma,r}^F)^*}(\tilde g)$.
    \end{itemize}
    As a consequence, we see that the Fourier transform $\mcF_{\bg_{\sigma,r}^F}:C(\bg_{\sigma,r}^F) \longrightarrow C\left((\bg_{\sigma,r}^F)^*\right)$ is an algebra isomorphism with inverse given by $\mcF_{(\bg_{\sigma,r}^F)^*}\circ \mcF_{\bg_{\sigma,r}^F}\circ \mcF_{(\bg_{\sigma,r}^F)^*}$, where multiplication on $C\left((\bg_{\sigma,r}^F)^*\right)$ is given by the usual pointwise one. The analogous result holds true for  $\mcF_{(\bg_{\sigma,r}^F)^*}:C\left((\bg_{\sigma,r}^F)^*\right) \longrightarrow C(\bg_{\sigma,r}^F) $.
\end{proposition}
\begin{proof}
    Most of these proofs follow from simple calculations using the definitions.
    \begin{itemize}
       \item[(i)]  Follows immediately from calculations.
        \item [(ii)] Let $Y \in \bg_{\sigma,r}^F$. Then, \begin{align*}
            \mcF_{(\bg_{\sigma,r}^F)^*}\circ\mcF_{\bg_{\sigma,r}^F}(f)(Y)& =\left|(\bg_{\sigma,r}^F)^*\right|^{-1/2}\sum_{X^*\in (\bg_{\sigma,r}^F)^*}\Tilde\psi( X^*(Y))\mcF_{\bg_{\sigma,r}^F}(f)(X^*)\\
            &= \left|(\bg_{\sigma,r}^F)^*\right|^{-1/2}\sum_{X^*\in (\bg_{\sigma,r}^F)^*}\Tilde\psi( X^*(Y)) \left(\left|\bg_{\sigma,r}^F\right|^{-1/2}\sum_{Y_1\in \bg_{\sigma,r}^F}\Tilde\psi( X^*(Y_1))f(Y_1)\right)\\
            &= \left|\bg_{\sigma,r}^F\right|^{-1}\sum_{X^*\in (\bg_{\sigma,r}^F)^*}\sum_{Y_1\in \bg_{\sigma,r}^F}\Tilde\psi( X^*(Y+Y_1)) f(Y_1)\\
            &= \sum_{Y_1\in \bg_{\sigma,r}^F}f(Y_1)\left( \left|\bg_{\sigma,r}^F\right|^{-1} \sum_{X^*\in (\bg_{\sigma,r}^F)^*} \Tilde\psi( X^*(Y+Y_1))\right)
        \end{align*}
        We know that the inside sum $ \frac{1}{\left|\bg_{\sigma,r}^F\right|}\sum_{X^*\in (\bg_{\sigma,r}^F)^*} \Tilde\psi( X^*(Y+Y_1))= \begin{cases} 0, & \text{ if } Y \neq -Y_1\\
        1, &\text{ if } Y = -Y_1\\
        \end{cases}$ which gives us $\mcF_{(\bg_{\sigma,r}^F)^*}\circ\mcF_{\bg_{\sigma,r}^F}(f)(Y)= f (-Y)$ and we are done. 
        \item[(iii)] Let $X^* \in (\bg_{\sigma,r}^F)^*$. Then, 
        \begin{align*}
            \mcF_{\bg_{\sigma,r}^F}(f*g)(X^*)&= \left|\bg_{\sigma,r}^F\right|^{-1/2}\sum_{Y\in \bg_{\sigma,r}^F}\Tilde\psi( X^*(Y))f*g(Y)\\
            &= \left|\bg_{\sigma,r}^F\right|^{-1/2}\sum_{Y\in \bg_{\sigma,r}^F}\Tilde\psi( X^*(Y))\left( \left|\bg_{\sigma,r}^F\right|^{-1/2}\sum_{Z\in \bg_{\sigma,r}^F}f(Y-Z)g(Z)\right)\\
            &= \left|\bg_{\sigma,r}^F\right|^{-1} \sum_{Y\in \bg_{\sigma,r}^F}\Tilde\psi( X^*(Y-Z))f(Y-Z)\sum_{Z\in \bg_{\sigma,r}^F}\Tilde\psi( X^*(Z))g(Z)\\
            &=\left(\left|\bg_{\sigma,r}^F\right|^{-1/2}\sum_{Y\in \bg_{\sigma,r}^F}\Tilde\psi( X^*(Y-Z))f(Y-Z)\right)\left(\left|\bg_{\sigma,r}^F\right|^{-1/2}\sum_{Z\in \bg_{\sigma,r}^F}\Tilde\psi( X^*(Z))g(Z)\right)\\
            &=  \mcF_{\bg_{\sigma,r}^F}(f)(X^*)\cdot  \mcF_{\bg_{\sigma,r}^F}(g)(X^*)\\
        \end{align*} 
        \item[(iv)] For any $f,g \in  C\left(\bg_{\sigma,r}^F\right)$, we see using $(i)$ that 
        \begin{equation}\label{fourierpf1}
            \mcF_{(\bg_{\sigma,r}^F)^*}\circ\mcF_{\bg_{\sigma,r}^F}(f*g)= \left(\mcF_{(\bg_{\sigma,r}^F)^*}\circ\mcF_{\bg_{\sigma,r}^F}(f)\right)*\left(\mcF_{(\bg_{\sigma,r}^F)^*}\circ\mcF_{\bg_{\sigma,r}^F}(g)\right).
        \end{equation}
        Further, applying $\mcF_{(\bg_{\sigma,r}^F)^*}\circ \mcF_{\bg_{\sigma,r}^F}\circ \mcF_{(\bg_{\sigma,r}^F)^*}$ to both sides of $(ii)$, we have 
       \[
       \mcF_{(\bg_{\sigma,r}^F)^*}\circ \mcF_{\bg_{\sigma,r}^F}\circ \mcF_{(\bg_{\sigma,r}^F)^*} \circ \mcF_{\bg_{\sigma,r}^F}(f*g)= f*g=\mcF_{(\bg_{\sigma,r}^F)^*}\circ \mcF_{\bg_{\sigma,r}^F}\circ \mcF_{(\bg_{\sigma,r}^F)^*}\left(\mcF_{\bg_{\sigma,r}^F}(f)\cdot \mcF_{\bg_{\sigma,r}^F}(g)\right)
        \]
         Then, applying $\mcF_{(\bg_{\sigma,r}^F)^*} \circ \mcF_{\bg_{\sigma,r}^F}$ to both sides, we have by using \eqref{fourierpf1}
        \begin{align*}
            \mcF_{(\bg_{\sigma,r}^F)^*} \circ \mcF_{\bg_{\sigma,r}^F}(f*g)&=\mcF_{(\bg_{\sigma,r}^F)^*}\left(\mcF_{\bg_{\sigma,r}^F}(f)\cdot \mcF_{\bg_{\sigma,r}^F}(g)\right)\\
            &= \left(\mcF_{(\bg_{\sigma,r}^F)^*}\circ\mcF_{\bg_{\sigma,r}^F}(f)\right)*\left(\mcF_{(\bg_{\sigma,r}^F)^*}\circ\mcF_{\bg_{\sigma,r}^F}(g)\right).
        \end{align*}
         For any $\tilde f, \tilde g \in C\left((\bg_{\sigma,r}^F)^*\right)$, we can choose $f, g \in  C\left(\bg_{\sigma,r}^F\right)$ such that $\mcF_{\bg_{\sigma,r}^F}(f)=\tilde f$ and $\mcF_{\bg_{\sigma,r}^F}(g)=\tilde g$ (follows from $(i)$). Then, using the previous equation, we have $\mcF_{(\bg_{\sigma,r}^F)^*}(\tilde f \cdot \tilde g) = \mcF_{(\bg_{\sigma,r}^F)^*}(\tilde f) * \mcF_{(\bg_{\sigma,r}^F)^*}(\tilde g)$ and we are done. 
    \end{itemize}
    We have given the proofs for one of the results in each of $(ii)$, $(iii)$ and $(iv)$. The others follow from exactly similar calculations. Once we have these, the last statement follows immediately.  
\end{proof}
\begin{proposition}[Compatibility with restriction maps]\label{fourierres}
    Let $\sigma, \tau \in [\B_m]$ such that $\sigma \preceq \tau$ and $\tau',\;\tau $ be $\sigma$-opposite. We have 
    \begin{itemize}
        \item [(i)] $ \left|\bu_{\sigma,\tau',r}^F\right|^{1/2}\mathrm{Res}^{(\bg_{\sigma,r})^*}_{(\bg_{\tau,r})^*} \circ \mcF_{\bg_{\sigma,r}^F} = \left|\bu_{\sigma,\tau,r}^F\right|^{1/2}\mcF_{\bg_{\tau,r}^F} \circ \mathrm{Res}^{\bg_{\sigma,r}}_{\bg_{\tau,r}} $.
        \item[(ii)] $\left|\bu_{\sigma,\tau',r}^F\right|^{1/2}\mcF_{(\bg_{\tau,r}^F)^*}\circ \mathrm{Res}^{(\bg_{\sigma,r})^*}_{(\bg_{\tau,r})^*}= \left|\bu_{\sigma,\tau,r}^F\right|^{1/2}\mathrm{Res}^{\bg_{\sigma,r}}_{\bg_{\tau,r}} \circ \mcF_{(\bg_{\sigma,r}^F)^*}$. 
    \end{itemize}
\end{proposition}
\begin{proof}
    Let $f \in  C\left(\bg_{\sigma,r}^F\right)$ and $X^* \in (\bg_{\tau,r}^F)^*$. Then, 
    \begin{align*}
        \mathrm{Res}^{(\bg_{\sigma,r})^*}_{(\bg_{\tau,r})^*} \circ \mcF_{\bg_{\sigma,r}^F} (f)(X^*) &= \left|\bu_{\sigma,\tau',r}^F\right|^{-1}\sum_{N^* \in (\bu_{\sigma,\tau',r}^F)^*} \mcF_{\bg_{\sigma,r}^F} (f)(X^*+N^*)\\
        &=\left|\bu_{\sigma,\tau',r}^F\right|^{-1}\sum_{N^* \in (\bu_{\sigma,\tau',r}^F)^*}\left|\bg_{\sigma,r}^F\right|^{-1/2}\sum_{Y\in \bg_{\sigma,r}^F}\Tilde\psi(( X^*+N^*)(Y))f(Y)\\
        &=\left|\bg_{\sigma,r}^F\right|^{-1/2}\left|\bu_{\sigma,\tau',r}^F\right|^{-1}\sum_{Y\in \bg_{\sigma,r}^F}\Tilde\psi(X^*(Y))f(Y)\sum_{N^* \in (\bu_{\sigma,\tau',r}^F)^*}\Tilde\psi(N^*(Y)). 
    \end{align*}
    Now, $\sum_{N^* \in (\bu_{\sigma,\tau',r}^F)^*}\Tilde\psi(N^*(Y))=0$ if $N^*(Y)\neq 0$ since $N^* \mapsto \Tilde\psi(N^*(Y))$ is a character of  $(\bu_{\sigma,\tau',r}^F)^*$. Since $\bg_{\sigma,r}^F \cong \bu_{\sigma,\tau,r}^F \oplus \bg_{\tau,r}^F \oplus \bu_{\sigma,\tau',r}^F$, the sum is only non-zero when $N^*(Y)=0$, i.e., $Y \in \bg_{\tau,r}^F \oplus \bu_{\sigma,\tau,r}^F$. So, it follows that 
    \[
    \sum_{N^* \in (\bu_{\sigma,\tau',r}^F)^*}\Tilde\psi(N^*(Y))= \begin{cases}
        0, & \text{ when  } Y \in \bu_{\sigma,\tau',r}^F\\
        \left|\bu_{\sigma,\tau',r}^F\right|, &\text{ when } Y \not \in \bu_{\sigma,\tau',r}^F
    \end{cases}
    \]
    Using the above, we have 
    \begin{align*}
        \mathrm{Res}^{(\bg_{\sigma,r})^*}_{(\bg_{\tau,r})^*} \circ \mcF_{\bg_{\sigma,r}^F} (f)(X^*) &= \left|\bg_{\sigma,r}^F\right|^{-1/2}\sum_{Y \in \bg_{\tau,r}^F \oplus \bu_{\sigma,\tau,r}^F}\Tilde\psi(X^*(Y))f(Y)\\
        &=\left|\bg_{\sigma,r}^F\right|^{-1/2}\sum_{Y_1 \in \bg_{\tau,r}^F}\Tilde\psi(X^*(Y_1)) \sum_{N\in\bu_{\sigma,\tau,r}^F}f(Y_1+N)\\
        &=\left|\bg_{\sigma,r}^F\right|^{-1/2}\left|\bu_{\sigma,\tau,r}^F\right|\sum_{Y_1 \in \bg_{\tau,r}^F}\Tilde\psi(X^*(Y_1))\mathrm{Res}^{\bg_{\sigma,r}}_{\bg_{\tau,r}}(f)(Y_1)\\
        &=\left|\bu_{\sigma,\tau,r}^F\right|^{1/2}\left|\bu_{\sigma,\tau',r}^F\right|^{-1/2}\mcF_{\bg_{\tau,r}^F} \circ \mathrm{Res}^{\bg_{\sigma,r}}_{\bg_{\tau,r}}(f)(X^*)
    \end{align*}
\end{proof}
\begin{remark}\label{csigmataur}
    Let us define $c_{\sigma,\tau,r} = \left|\bu_{\sigma,\tau,r}^F\right|^{1/2}\left|\bu_{\sigma,\tau',r}^F\right|^{-1/2}$ for $\sigma \preceq\tau,\tau'$ and $\tau$, $\tau'$ being $\sigma$-opposite. Then we can restate the results in Proposition \ref{fourierres} as follows:
     \begin{itemize}
        \item [(i)] $ \mathrm{Res}^{(\bg_{\sigma,r})^*}_{(\bg_{\tau,r})^*} \circ \mcF_{\bg_{\sigma,r}^F} = c_{\sigma,\tau,r} \mcF_{\bg_{\tau,r}^F} \circ \mathrm{Res}^{\bg_{\sigma,r}}_{\bg_{\tau,r}} $.
        \item[(ii)] $\mcF_{(\bg_{\tau,r}^F)^*}\circ \mathrm{Res}^{(\bg_{\sigma,r})^*}_{(\bg_{\tau,r})^*}= c_{\sigma,\tau,r}\mathrm{Res}^{\bg_{\sigma,r}}_{\bg_{\tau,r}} \circ \mcF_{(\bg_{\sigma,r}^F)^*}$. 
    \end{itemize}
    Note that, for $r \in \Z_{>0}$, $\bg_{\sigma,r} \cong \bg_\sigma$ and the Iwahori decomposition of $\bg_{\sigma,r}$ for agrees with the usual triangular decomposition of the Lie algebra $\bg_\sigma$. In this case, the theory of Fourier transform developed here reduces to the case studied in \cite{Le}. Further, for integral depths, $c_{\sigma,\tau,r}=1$ and the restriction and Fourier transform maps commute, as expected. 
\end{remark}
\begin{proposition}[Compatibility with induction]
     Let $\sigma, \tau \in [\B_m]$ such that $\sigma \preceq \tau$ and $\tau',\;\tau $ be $\sigma$-opposite. We have 
     \begin{itemize}
         \item [(i)]$c_{\sigma,\tau,r}\mcF_{\bg_{\sigma,r}^F} \circ \Ind^{\bg_{\sigma,r}}_{\bg_{\tau,r}} = \Ind^{(\bg_{\sigma,r})^*}_{(\bg_{\tau,r})^*} \circ \mcF_{\bg_{\tau,r}^F}$.
         \item[(ii)] $ c_{\sigma,\tau,r} \Ind^{\bg_{\sigma,r}}_{\bg_{\tau,r}}\circ \mcF_{(\bg_{\tau,r}^F)^*}= \Ind^{(\bg_{\sigma,r})^*}_{(\bg_{\tau,r})^*} \circ \mcF_{\bg_{\tau,r}^F}$
     \end{itemize}
\end{proposition}
\begin{proof}
    Let $f \in C(\bg_{\tau,r}^F)$ and $g \in C((\bg_{\sigma,r}^F)^*)$. Then,  
    \begin{align*}
        (\mcF_{\bg_{\sigma,r}^F} \circ \Ind^{\bg_{\sigma,r}}_{\bg_{\tau,r}}(f),g)&=(\Ind^{\bg_{\sigma,r}}_{\bg_{\tau,r}}(f), \mcF_{\bg_{\sigma,r}^F}^{-1}(g))\\
        &=(f, \mathrm{Res}^{\bg_{\sigma,r}}_{\bg_{\tau,r}}\circ \mcF_{\bg_{\sigma,r}^F}^{-1}(g))\\
        &=(f, c_{\sigma,\tau,r}^{-1}\mcF_{\bg_{\tau,r}^F}^{-1}\circ \mathrm{Res}^{(\bg_{\sigma,r})^*}_{(\bg_{\tau,r})^*}(g))\\
        &= (\mcF_{\bg_{\tau,r}^F}(f), c_{\sigma,\tau,r}^{-1}\mathrm{Res}^{(\bg_{\sigma,r})^*}_{(\bg_{\tau,r})^*}(g))\\
        &=(c_{\sigma,\tau,r}^{-1}\Ind^{(\bg_{\sigma,r})^*}_{(\bg_{\tau,r})^*}\circ\mcF_{\bg_{\tau,r}^F}(f), g)
    \end{align*}
    Since, the inner product is non-degenerate and $f$ and $g$ were arbitrary this proves the first statement, and the second one follows similarly. 
\end{proof}

\subsection{Stable functions}\label{subsectionstablefunction}
We define and study the properties of stable functions on Moy-Prasad quotients, extending ideas developed in \cite{cb24} using the theory of Fourier transforms developed in the previous section. Throughout, $r \in \R_{>0}$, and we will consider only $m$ such that $p \nmid m$. Let $\Gamma_k$ denote $\Gal(\bar k/k)$, with Krull topology. We fix an uniformizer $\varpi \in \mf O_{k}$ and an arbitrary $m \in \Z_{>0}$ be such that $p \nmid m$. Let $E$ be the unique tamely totally ramified extension of $K$ of degree $m$. We further fix $\gamma \in \mf O_{E}$ such that $\gamma^m=\varpi$. Then $k(\gamma)=E'$ is a tame totally ramified extension of $k$, and $E=(E')^{u}= K(\gamma)$ . Let $\mf F_\gamma$ denote the unique Frobenius element in $\Gal(E/k)$ such that $\mf F_\gamma(\gamma)=\gamma$. The subgroup $\langle \mf F_\gamma\rangle \subset \Gal(E/k)$ lies inside  $\Gal(E/E')$, and $\mf F_\gamma$ is the topological generator of $\Gal(E/E')$. $\Gal(E/K)$ is a cyclic group of order $m$, generated by $\vartheta$ (say). Further, consider the natural projection map $\Gal(E/k) \longrightarrow \Gal(K/k)$ and we denote the image of $\mf F_\gamma$ by $\mf F$, which is the topological generator of $\Gamma^{ur}:=\Gal(K/k)$.\paragraph{} 

We go back to the setting of the fixed $k$-split maximal torus $T$, and the apartment $\sA=\A_T$ corresponding to it. Let $k^t$ be the maximal tamely ramified extension of $k$, and $\T$ be the $\Bar\F_q$- group defined by $\T(\Bar\F_q)= T(k^t)_0/T(k^t)_{0+}$. $\T$ is the reductive quotient of the special fiber of the connected Neron model of $T$. Since $T$ is $k$-split we have $\T(\Bar\F_q)=T(K)_0/T(K)_{0+} = T(E)_0/T(E)_{0+}$. Further, $\T$  is $\F_q$-split and has a $\F_q$-structure with the group of $\F_q$-rational points given by $\T(\F_q)= T(k)_0/T(k)_{0+}=T(E')_0/T(E')_{0+}$. Let $\bt=\Lie(\T) $ and $\bt^*$ be its dual.Then, $\bt(\Bar\F_q)= \ft(K)_0/\ft(K)_{0+}= \ft(E)_0/\ft(E)_{0+}$ and has a $\F_q$-structure given by $\bt^F=\bt(\F_q)= \ft(k)_0/\ft(k)_{0+}=\ft(E')_0/\ft(E')_{0+}$. There is a canonical isomorphism between $ (\bt^F)^*$ and $\ft^*(k)_0/\ft^*(k)_{0+}= \bt^*(\F_q)$, which gives $\bt^*$ an $\F_q$-structure. \paragraph{}

 We denote again by $\mathrm{v}:E\rightarrow \rv(E):=\frac{1}{m}\Z \cup \{\infty\}$ the extension of the fixed valuation  $\mathrm{v}: K\rightarrow \rv(K):=\Z \cup \{\infty\}$ to $E$. We know that $\B(G,k)= \B(G_{K},K)^{\Gamma}$. Let $x \in \A(T,k)$. There is a $G(K)$-equivariant injection $\B(G_{K},K) \hookrightarrow \B(G_{E},E)$, and we denote the image of $x$ under this map to be $x$ as well. We can define Moy-Prasad filtrations of $G(E)$, $\fg(E)$ and $\fg^*(E)$, using the valuation $\mathrm{v}$ instead of the normalized one ($ m \cdot \rv$). We denote by $(\G_E)_x$ the reductive quotient at $x$, i.e., the $\Bar\F_q$-group such that $(\G_E)_x(\Bar\F_q)= G(E)_{x,0}/G(E)_{x,0+}$. Let $(\bg_E)_x=\Lie ((\G_E)_x)$ and $(\bg_E)_x^*$ be it's dual. Let $(\bg_E)_{x,r}=G(E)_{x,r}/G(E)_{x,r+}\cong \fg(E)_{x,r}/\fg(E)_{x,r+}$ and $(\bg^*_E)_{x,-r}= \fg^*(E)_{x,-r}/\fg^*(E)_{x,(-r)+}$.  Since $x \in \A(T,k)$, $\mf F_\gamma(x)=\mf F(x)=x$, and $\mf F $ preserves $G(K)_{x,0}$ inducing the $\F_q$ structure on $\G_x$ with $\G_x^F=G(k)_{x,0}/G(k)_{x,0+}$. Similarly, $\mf F_\gamma$ induces an $\F_q$- structure on $(\G_E)_x$ and $(\G_E)_x^F=(\G_E)_x(\F_q)=G(E')_{x,0}/G(E')_{x,0+}$. $\mf F_\gamma$ also induces an $\F_q$-structure on $(\bg_E)_{x,r}$ and $(\bg^*_E)_{x,-r}$, with $(\bg_E)_{x,r}(\F_q)= G(E')_{x,r}/G(E')_{x,r+}\cong \fg(E')_{x,r}/\fg(E')_{x,r+}$ and $(\bg^*_E)_{x,-r}(\F_q) =\fg^*(E')_{x,-r}/\fg^*(E')_{x,(-r)+}$. There is a $(\G_E)_x$-equivariant isomorphism
 \begin{equation}
     ((\bg_E)_{x,r})^*\cong (\fg(E)_{x,r}/\fg(E)_{x,r+})^*\cong \fg^*(E)_{x,-r}/\fg^*(E)_{x,(-r)+}
 \end{equation}
 defined over $\F_q$, and we henceforth use $(\bg_E)_{x,r}^*$ to denote it. Similarly, there is a $(\G_E)_x^F$-equivariant isomorphism 
\begin{equation}
    ((\bg_E)^F_{x,r})^*\cong \left(\fg(E')_{x,r}/\fg(E')_{x,r+}\right)^*\cong \fg^*(E')_{x,-r}/\fg^*(E')_{x,(-r)+}
\end{equation}
 and hence $((\bg_E)^F_{x,r})^*$, $((\bg_E)_{x,r}^*)^F$ and $\fg^*(E')_{x,-r}/\fg^*(E')_{x,(-r)+}$ are canonically isomorphic. So, we will use $((\bg_E)^F_{x,r})^*$ to denote them without any ambiguity. \par
 The inclusion $G(K) \xhookrightarrow{\iota}G(E)$ maps $G(K)_{x,r}$ into $G(E)_{x,r}$, and this induces an injection at the level of $\Bar \F_q$ points 
\begin{equation}
    \iota_{E}(\Bar\F_q): \G_x(\Bar\F_q)=G(K)_{x,0}/G(K)_{x,0+} \longrightarrow G(E)_{x,0}/G(E)_{x,0+}=(\G_E)_x(\Bar\F_q)
\end{equation}

which gives a map of algebraic groups $\iota_E: \G_x \longrightarrow(\G_E)_x$. As noted in \cite{fintzenmoyprasad} Section 2.6, this is a closed immersion if $p\neq 2$.
\begin{lemma}\label{lemmaGIT}
    For $p\neq 2$ and every $r \in \R_{>0}$, there is an injection 
    \begin{equation}
        \iota_{E,r}:(\bg_{x,r})^*\hookrightarrow  (\bg_E)_{x,r}^*
    \end{equation}
    such that $\iota_E(\G_x)$ preserves $\iota_{E,r}((\bg_{x,r})^*)$ and we have a commutative diagram 
    \begin{equation}
        \begin{tikzcd}
            \G_x \times (\bg_{x,r})^*\arrow{r} \arrow{d}{\iota_E\times \iota_{E,r}} &(\bg_{x,r})^*\arrow{d}{\iota_{E,r}} \\
             (\G_E)_x\times (\bg_E)_{x,r}^* \arrow{r} &(\bg_E)_{x,r}^*
        \end{tikzcd}
    \end{equation}
    Moreover, the maps in the diagram are compatible with the respective $\F_q$-structures and hence descends to morphisms over $\F_q$. 
\end{lemma}
\begin{proof}
$\iota_{E,r}$  is induced by the inclusion $\fg^*(K) \hookrightarrow \fg^*(K) \otimes_K E \cong \fg^*(E)$. Note that equivalently we can consider it to be induced by the natural projection map of $\Bar\F_q$-vector spaces $(\bg_E)_{x,r} \rightarrow\bg_{x,r}$. So, $\iota_{E,r}$ is well-defined. Further, if $r \in \rv(K) \subset \rv(E)$, $\ft^*(K) \cap\fg^*(K)_{x,-r}/ \fg^*(K)_{x,(-r)+} $ maps injectively into $\ft^*(E) \cap\fg^*(E)_{x,-r}/ \fg^*(E)_{x,(-r)+}$, and for $r \not \in \rv(K)$, $\ft^*(K) \cap\fg^*(K)_{x,-r}/ \fg^*(K)_{x,(-r)+}=\{0\} $. Let $\fg^*_\alpha $ be defined as in \cite{MP94}. Observing that $\fg^*_\alpha $ can be identified with the dual of $\fg_{-\alpha}$, we see that the map $\fg^*_\alpha (K) \cap\fg^*(K)_{x,-r}/ \fg^*(K)_{x,(-r)+} \longrightarrow \fg^*_\alpha (E) \cap\fg^*(E)_{x,-r}/ \fg^*(E)_{x,(-r)+}$ is injective for $r-\alpha(x) \in \rv(K) \subset \rv(E)$, and $\fg^*_\alpha (K) \cap\fg^*(K)_{x,-r}/ \fg^*(K)_{x,(-r)+}=\{0\}$ for $r -\alpha(x) \not \in \rv(K)$. This gives that the map $\iota_{E,r}$ is injective. Once we have this, the diagram commutes essentially because both the top and the bottom horizontal action maps are induced by the co-adjoint action of $G$ on $\fg^*$. \par
    Now, since $\mf F_\gamma(x)=\mf F(x)=x$, the top and the bottom horizontal maps descend to $\F_q$ with the given $\F_q$-structure by unramified descent. So, if we can show that the Frobenius structure is compatible with $\iota_E$ and $\iota_{E,r}$, we are done. It is enough to check these properties at the level of $\Bar\F_q$-points. Let $\mf F'$ and $(\mf F_\gamma)'$ denote the maps $G(K) \rightarrow G(K)$ and $G(E)\rightarrow G(E)$ induced by $\mf F$ and $\mf F_\gamma$ respectively, and $i:K \hookrightarrow E$ denote the inclusion. In order to show that the Frobenius $F$ commutes with $\iota_E(\Bar\F_q)$, it is enough to show that $\iota \circ \mf F'=(\mf F_\gamma)'\circ \iota$. This equivalent to showing $i\circ \mf F=\mf F_\gamma\circ i$, which is true since $\mf F_\gamma|_K=\mf F$. Similarly, the case of $\iota_{E,r}$ also boils down to the exact same statement using the inclusion $\fg^*(K) \hookrightarrow  \fg^*(E)$ instead of $\iota$, and we are done.  
\end{proof}
Consider the map of $\Bar\F_q$-varieties
\begin{equation}\label{inclusiondualtame}
    (\bg_{x,r})^*\hookrightarrow (\bg_E)_{x,r}^* \rightarrow  (\bg_E)_{x,r}^*//(\G_E)_x
\end{equation}
Using Lemma \ref{lemmaGIT}, we see that the map factors through 
\begin{equation}\label{GITdiag}
    \begin{tikzpicture}[baseline=(current  bounding  box.center)]
        \node (A) at (0,0) {$ (\bg_{x,r})^*$};
        \node (B) at (2,0) {$(\bg_E)_{x,r}^* $};
        \node (C) at (4.5,0) {$(\bg_E)_{x,r}^*//(\G_E)_x$};
        \node (E) at (2,-2) {$(\bg_{x,r})^*//\G_x$};
        \draw [right hook->] (A)--(B);
        \draw[->] (B)--(C);
         \draw[->] (A)--(E);
         \draw[->] (E)--(C);
    \end{tikzpicture}
\end{equation}
Further, each of the maps in the above two equations is defined over $\F_q$, and hence we also get a map of their $\F_q$-rational (or $F$-fixed) points
\begin{equation}\label{GITdiagFrob}
    \begin{tikzpicture}[baseline=(current  bounding  box.center)]
        \node (A) at (0,0) {$ (\bg_{x,r}^F)^*$};
        \node (B) at (2.5,0) {$((\bg_E)^F_{x,r})^* $};
        \node (C) at (6,0) {$\left((\bg_E)_{x,r}^*//(\G_E)_x\right)^F$};
        \node (E) at (2.5,-2) {$\left((\bg_{x,r})^*//\G_x\right)^F$};
        \draw [right hook->] (A)--(B);
        \draw[->] (B)--(C);
         \draw[->] (A)--(E);
         \draw[->] (E)--(C);
    \end{tikzpicture}
\end{equation}
\paragraph{}

Let $r \in \frac{1}{m}\Z_{>0} \subset \Z_{(p)}$. Note that for $\sigma \in [\sA_m]$, $G(E)_{x,r}$, $\fg(E)_{x,r}$, $\fg^*(E)_{x,r}$ does not vary for $x \in \sigma $, and hence $(\G_E)_\sigma$, $ (\bg_E)_{\sigma}= \Lie((\G_E)_\sigma)$, $(\bg_E)_{\sigma,r}$, $(\bg_E)^*_{\sigma,r}$ and the other similar objects are well-defined. Also, since $\sigma \subset \sA$, each of these objects have an $\F_q$-structure exactly as described before. Further, for any $\frac{i}{m} \in \frac{1}{m}\Z_{>0}$, there are $(\G_E)_{\sigma}$-equivariant isomorphisms of Moy-Prasad filtration quotients
\begin{equation}\label{isomLiem}
\fg(E)_{\sigma,\frac{i}{m}}/\fg(E)_{\sigma,\frac{i}{m}+} \xlongrightarrow{\times \gamma^{-i}} \fg(E)_{\sigma,0}/\fg(E)_{\sigma,0+} 
\end{equation}
and 
\begin{equation}\label{isomdualLiem}
\fg^*(E)_{\sigma,-\frac{i}{m}}/\fg^*(E)_{\sigma,-\frac{i}{m}+} \xlongrightarrow{\times \gamma^{i}} \fg^*(E)_{\sigma,0}/\fg^*(E)_{\sigma,0+}. 
\end{equation}
The action of $G(E)_{\sigma,0}/G(E)_{\sigma, 0+}$ on $\fg(E)_{\sigma,\frac{i}{m}}/\fg(E)_{\sigma,\frac{i}{m}+}$ (and $\fg^*(E)_{\sigma,-\frac{i}{m}}/\fg^*(E)_{\sigma,-\frac{i}{m}+} $) induced by the adjoint (resp. co-adjoint) action of $G$ on $\fg$ (resp. $\fg^*$) is isomorphic to the adjoint (resp. co-adjoint) action of $(\G_E)_{\sigma}$ on $(\bg_E)_\sigma$ (resp. $(\bg_E)_\sigma^*$). This action is defined over $\F_q$, and the action of $G(E')_{\sigma,0}/G(E')_{\sigma, 0+}$ on $\fg^*(E')_{\sigma,-\frac{i}{m}}/\fg^*(E')_{\sigma,-\frac{i}{m}+} $ is isomorphic to the $\F_q$-points of the co-adjoint action $(\G_E)_{\sigma}$ on $(\bg_E)^*_\sigma$. \par
Since we have fixed $m$ and hence $E$, for the next part of this section we define $\mbL_\sigma:=(\G_E)_{\sigma}$ and $\bl_\sigma = \fg(E)_{\sigma,0}/\fg(E)_{\sigma,0+}\cong (\bg_E)_\sigma=\Lie(\mbL_\sigma)$ in order to simplify the notation. Further, we also fix $\gamma \in \mf O_E$ and hence $E'$. Note that if $m=1$, i.e., $r \in \Z_{>0}$, then $E=K$, and $\mbL_\sigma \cong \G_\sigma $ and $\bl_\sigma \cong \bg_\sigma$. $\mbL_\sigma$ is a connected reductive $\bar\F_q$-group defined over $\F_q$ with root system $\Phi_\sigma$, $W_\sigma= N_{\mbL_\sigma}(\T)/\T$ and $\G_\sigma$ embeds into $\mbL_\sigma$ as a generalized Levi subgroup. Also, for $\sigma \preceq \tau$, $\mbL_\tau \subset \mbL_\sigma$ is an $F$-stable Levi subgroup.  \par

Let $\sigma, \tau, \tau' \in [\sA_m]$ be such that $\sigma \preceq \tau$, $\sigma \preceq \tau' $ and $ \tau$, $\tau'$ are $\sigma$-opposite. We have  \begin{equation*}
    \U_{\sigma,\tau}(\bar\F_q) = G(E)_{\tau,0+}/G(E)_{\sigma,0+} \subset G(E)_{\tau,0}/G(E)_{\sigma,0+} =\mbP_{\sigma, \tau}(\bar\F_q)
\end{equation*}
where $\mbP_{\sigma, \tau}$ is an $F$-stable parabolic subgroup with unipotent radical $ R_u(\mbP_{\sigma,\tau} )=\U_{\sigma,\tau}$ and Levi quotient $\mbL_\tau \cong \mbP_{\sigma,\tau}/\U_{\sigma,\tau}$, and hence a Levi decomposition $\mbP_{\sigma,\tau} = \U_{\sigma,\tau} \rtimes \mbL_\tau$.  $\mbP_{\sigma, \tau}$ and $\U_{\sigma,\tau}$ have $\F_q$-structures given by  $\U_{\sigma,\tau}(\F_q) = G(E')_{\tau,0+}/G(E')_{\sigma,0+} $ and $\mbP_{\sigma, \tau}(\F_q) =G(E')_{\tau,0}/G(E')_{\sigma,0+}$. Note that for $\sigma \preceq \tau$, $\mbL_\tau$ is basically the unique Levi subgroup in $\mbP_{\sigma,\tau}$ containing $\T$, and $\mbP_{\sigma,\tau'} $ is the opposite parabolic subgroup in $\mbL_\sigma$ with unipotent radical $\U_{\sigma,\tau'}$. Since $\tau,\tau'$ are $\sigma$-opposite, we have $\bl_{\tau}\cong \bl_{\tau'}$.  Let $\bu_{\sigma,\tau}=\Lie(\U_{\sigma,\tau})$ and $\bp_{\sigma,\tau}=\Lie(\mbP_{\sigma,\tau})=\bl_\tau \oplus \bu_{\sigma,\tau}$. We have $\bl_\sigma= \bl_\tau \oplus \bu_{\sigma,\tau}\oplus \bu_{\sigma,\tau'}=\bp_{\sigma,\tau}\oplus \bu_{\sigma,\tau'}$.\paragraph{}

For $r \in \frac{1}{m}\Z_{>0}$, we see from \eqref{isomLiem} that there is an isomorphism $(\bg_E)_{\sigma,r} \xlongrightarrow{\eta^\sigma_\gamma}\bl_\sigma$. From the statement and the proof of Proposition \ref{Iwahorifracdepthprop} and remarks following it, we note that there is a decomposition 
\begin{equation*}
    (\bg_E)_{\sigma,r} = (\bu_E)_{\sigma,\tau,r} \oplus (\bg_E)_{\tau,r} \oplus (\bu_E)_{\sigma,\tau',r}
\end{equation*}
where $(\bu_E)_{\sigma,\tau,r}= G(E)_{\tau,r+}/G(E)_{\sigma,r+} \cong \fg(E)_{\tau,r+}/\fg(E)_{\sigma,r+}$. The isomorphism $\eta^\sigma_\gamma$ sends $(\bu_E)_{\sigma,\tau,r} \rightarrow \bu_{\sigma,\tau} $ and $(\bg_E)_{\tau,r}\rightarrow \bl_\tau$ and is compatible with the given $\F_q$-structures. So, $\eta^\sigma_\gamma|_{(\bg_E)_{\tau,r}}=\eta^\tau_\gamma$ and from \eqref{isomdualLiem} we also have a similar isomorphism $(\bg_E)^*_{\sigma,r} \xlongrightarrow{\theta^\sigma_\gamma}\bl^*_\sigma$. 
\begin{remark}
    The $\F_q$-structures and the isomorphisms depend upon the choice of the uniformizer $\gamma \in \mf O_E$, and hence we have used a subscript of $\gamma$ in the notation of the same.
\end{remark}

We have $(\bl_\sigma^F)^*=(\bl_\sigma^*)^F$, and using the co-adjoint action of $\mbL_\sigma $ on $\bl_\sigma^*$ and Theorem 4 in \cite{KW1} ( which works for general connected reductive groups instead of just almost simple ones, check \cite[Section,4.1]{chendebackertsai}), we have a map of $\Bar\F_q$-varieties
\begin{equation*}
    \Bar\chi_\sigma : \bl_\sigma^* \longrightarrow \bl_\sigma^*//\mbL_\sigma\cong \bt^*//W_\sigma
\end{equation*}
and hence a map between their $F$-fixed (or $\F_q$-rational points)
\begin{equation}
    \chi_\sigma: (\bl_\sigma^F)^* \cong (\bl_\sigma^*)^F  \longrightarrow (\bl_\sigma^*//\mbL_\sigma)^F \xlongrightarrow{\simeq}(\bt^*//W_\sigma)^F.
\end{equation}
\begin{remark}
     The isomorphism $\mc O(\bt^*)^{W_{\sigma}} \cong \mc O(\bl_\sigma^*)^{\mbL_\sigma}$ is induced by the inclusion $\bt^* \hookrightarrow\bl_\sigma^*$ which is defined over $\F_q$, and hence $\bt^*//W_\sigma \xrightarrow{\sim}\bl_\sigma^*//\mbL_\sigma$ is defined over $\F_q$ (cf for example \cite{Milnealggps}, Corollary 4.34). Since $\bl^*_\sigma$ and $\mbL_\sigma$ are defined over $\F_q$, $\bl_\sigma^* \longrightarrow \bl_\sigma^*//\mbL_\sigma$ induced by $\mc O(\bl_\sigma^*)^{\mbL_\sigma} \hookrightarrow \mc O(\bl_\sigma^*)$ is also defined over $\F_q$. So, the map $\Bar \chi_\sigma$ is defined over $\F_q$ and we have a map of $\F_q$-rational points $\chi_\sigma$.
\end{remark}

From Lemma \ref{lemmaGIT}, we see that there is a map of $\bar\F_q$ vector spaces $(\bg_{\sigma,r})^* \hookrightarrow (\bg_E)_{\sigma,r}^* \xlongrightarrow{\theta_\gamma^\sigma} \bl_\sigma$ which is compatible with the $\F_q$ structures, and hence a map of $F$-fixed points 
\begin{equation*}
    p^*_{\sigma,r}:(\bg^F_{\sigma,r})^* \hookrightarrow ((\bg_E)^F_{\sigma,r})^* \longrightarrow (\bl_\sigma^F)^*
\end{equation*}
The map $ p^*_{\sigma,r}$ can equivalently be described as being induced by the projection map of $\F_q$-vector spaces $\bl_\sigma^F\xrightarrow{\simeq} (\bg_E)^F_{\sigma,r} \rightarrow \bg^F_{\sigma,r} $. Let $\chi_{\sigma,r}:=\chi_\sigma \circ p_{\sigma,r}^*$. Since the isomorphism $(\bg_E)^*_{\sigma,r} \xlongrightarrow{\theta^\sigma_\gamma}\bl^*_\sigma$ is $(\G_E)_\sigma=\mbL_\sigma$-equivariant and defined over $\F_q$, we have $(\bg_E)_{\sigma,r}^*//\mbL_\sigma \cong \bl^*_\sigma//\mbL_\sigma$ defined over $\F_q$ and using \eqref{GITdiag}, the following diagram commutes.
\begin{equation*}
     \begin{tikzpicture}[baseline=(current  bounding  box.center)]
        \node (A) at (0,0) {$ (\bg_{\sigma,r})^*$};
        \node (B) at (2,0) {$(\bg_E)_{\sigma,r}^* $};
        \node (D) at (4,0) {$\bl_\sigma^*$};
        \node (C) at (6,0) {$\bl_\sigma^*//\mbL_\sigma$};
        \node (E) at (3,-2) {$(\bg_{\sigma,r})^*//\G_\sigma$};
        \draw [right hook->] (A)--(B);
         \draw[->] (B)edge node[above] {$\simeq$}(D);
        \draw[->] (D)--(C);
         \draw[->] (A)--(E);
         \draw[->] (E)--(C);
    \end{tikzpicture}
\end{equation*}

Since all the above maps are defined over $\F_q$, we see using \eqref{GITdiagFrob} that the map $\chi_{\sigma,r}:(\bg^F_{\sigma,r})^* \longrightarrow(\bt^*//W_\sigma)^F$ factors as 
\begin{equation}\label{stableGIT}
    \begin{tikzpicture}[baseline=(current  bounding  box.center)]
        \node (A) at (0,0) {$ (\bg^F_{\sigma,r})^*$};
        \node (B) at (2,0) {$((\bg_E)^F_{\sigma,r})^* $};
        \node (D) at (4,0) {$(\bl_\sigma^F)^*$};
        \node (C) at (6,0) {$(\bl_\sigma^*//\mbL_\sigma)^F $};
        \node (E) at (3,-2) {$((\bg_{\sigma,r})^*//\G_\sigma)^F$};
        \node (F) at (8.8,0) {$(\bt^*//W_\sigma)^F$};
        \draw[->] (A)--(B);
         \draw[->] (B)edge node[above] {$\simeq$}(D);
        \draw[->] (D)--(C);
         \draw[->] (A)--(E);
         \draw[->] (E)--(C);
          \draw[->] (C)edge node[above] {$\simeq$}(F);
    \end{tikzpicture}
\end{equation}

Let $\C[(\bt^*//W_\sigma)^F]$ be the space of complex valued functions on  $(\bt^*//W_\sigma)^F$
with multiplication given by pointwise multiplication of functions.  The function $\mathbbm{1}$ which takes a value of $1$ at every point is the identity element. For any $\sigma \preceq \tau$, $W_\tau \subset W_\sigma$ and we have a canonical map $ \bt^*//W_\tau \longrightarrow\bt^*//W_\sigma$ compatible with 
the Frobenious endomorphism. We denote by 
\[\mathrm{res}^{\sigma}_{\tau}:\mathbb C[(\bt^*//W_\sigma)^F]\to \mathbb C[(\bt^*//W_\tau)^F]\]
the map given by pull back along the natural 
map $t_{\sigma,\tau}: (\bt^*//W_\tau)^F \longrightarrow (\bt^*//W_\sigma)^F$. 
\begin{definition}\label{defnstablefns}
    We define $f \in C(\bg_{\sigma,r}^F)$ to be stable if $\mcF_{\bg_{\sigma,r}^F}(f):(\bg_{\sigma,r}^F)^* \rightarrow \C$ factors through 
    \[
    (\bg_{\sigma,r}^F)^* \xlongrightarrow{\chi_{\sigma,r}} (\bt^*//W_\sigma)^F \longrightarrow \C.
    \]
   We use $C^{st}(\bg_{\sigma,r}^F)$ to denote stable functions on $\bg_{\sigma,r}^F$. Using properties of Fourier transforms, we immediately see that $C^{st}(\bg_{\sigma,r}^F) \subset C(\bg_{\sigma,r}^F)$  forms a subalgebra with unit $|\bg_{\sigma,r}^F|^{1/2}\mathbbm{1}_0$. 
\end{definition}

The following lemma is essential in proving certain properties of stable functions, which we will state after this. 
\begin{lemma}\label{lemmaresstable}
    Let $\rG= \rL \rU$ be a Levi decomposition of a parabolic subgroup $\rP$ of a connected reductive algebraic group $\rG$. Let $\mf P = \mf L\oplus \mf U$ be the corresponding decomposition of the Lie algebra $\mf P= \Lie(\rP)$. For, $X\in \mf L^*$ and $N\in \mf U^*$, we have $(X+N)_s=\Ad^*(v)(X_s)$ for some $v \in \rU$ and $X_s$ defined as in \cite{KW1} Section 3, where $\Ad^*$ denotes the co-adjoint action. 
\end{lemma}
\begin{proof}
    Let $\mf G= \Lie (\rG)$. We are working with the Jordan decomposition for the dual of the Lie algebra of a connected algebraic group as defined in Section 3 of \cite{KW1}. There exists a Borel $\rB_{\rL}=\rT\rV$ of $\rL$ such that $X_s\in \Lie(\rT)^*=\mf T^*$ and $X_n\in \Lie(\rV)^*$. We consider $\mf L^*$, $\mf P^*$, $\mf T^*$ and other duals embedded into $\mf G^*$ as in \cite{KW1}. Then, $\rB=\rB_\rL\rU=\rT\rV\rU$ is a Borel subgroup of $\rG$, and the Lie algebra $\Lie(\rB)=\Lie(\rT)\oplus\Lie(\rV)\oplus\Lie(\rU)$ and hence $\Lie(\rB)^*=\mf T^*\oplus \Lie(\rV)^*\oplus\mf U^*$. \par
    Now, $X+N\in \Lie(\rB)^*$ and using Jordan decomposition for duals, we see that $(X+N)_s=\Ad^*(v)(Y)$ for $Y \in \mf T^*$ and $v \in \rB$. Also, without loss of generality, we can assume that $b \in \rV\rU$ since $\rT$ acts via identity on $\mf T^*$. Consider the map 
    \begin{equation*}
        \varphi_v:= \Ad^*(v)-Id : \Lie(\rB)^* \longrightarrow \Lie(\rB)^*
    \end{equation*}
    $\mf T^*$ embedds into $\Lie(\rB)^*$ as $\mf T^*=\{a \in \Lie(\rB)^* \:|\: a(x)=0\: \forall \: x \in \Lie(\rV) \oplus \mf U\}$. We claim that $\varphi_v(\mf T^*) \subset  \Lie(\rV)^* \oplus \mf U^*$. Let $a \in \mf T^*$ and $x \in \mf T$. $\varphi_v(a)(x)=a ((\Ad(v^{-1})-Id)(x))$. Using Proposition 3.7 of \cite{Le}, we see that $(\Ad(v^{-1})-Id)|_{\mf T}$ for $v \in \rV\rU$ is a derivative of the map $\mu_{v^{-1}}:\rT \rightarrow \rV \rU$ defined by $\mu_{v^{-1}}(t)= t^{-1}v^{-1}tv$, and hence a map $(\Ad(v^{-1})-Id):\mf T\rightarrow \Lie(\rV)\oplus \mf U$. Thus, $\varphi_v(a)(x)= 0 \: \forall\: x \in \mf T \Rightarrow \varphi_v(a) \in \Lie(\rV)^* \oplus \mf U^*$, which proves our claim. \par
    Using the above claim, we see that $\Ad^*(v)(Y)-Y \in \Lie(\rV)^* \oplus \mf U^*$, and $(X+N)= (X+N)_s + (X+N)_n \in Y + \Lie(\rV)^* \oplus \mf U^*$ since $\rV\rU $ is normal in $\rB$. Further, we also have $X +N = X_s+X_n+ N\in X_s + \Lie(\rV)^* \oplus \mf U^*$. So, $X_s=Y$ by uniqueness of Jordan decomposition. In the above argument, $\rL$ could have been replaced by $Z_\rL(X_s)$, which shows that $v$ can be taken in $\rU$ and finishes our proof. 
\end{proof}

\begin{proposition}[Properties of stable functions]\label{stablefuncproperties}
    \begin{itemize}
        \item [(1)] For any $\varepsilon \in \C[(\bt^*//W_\sigma)^F]$, the function $ f_\varepsilon := \mcF_{(\bg_{\sigma.r}^F)^*}\left(\chi_{\sigma,r}^*(\varepsilon)^-\right)$
        is a stable function on $\bg_{\sigma,r}^F$, and the assignment $\varepsilon\mapsto f_\varepsilon$ gives a surjective algebra morphism
        \begin{equation*}
             \rho_{\sigma,r}:\C[ (\bt^*//W_\sigma)^F]\longrightarrow C^{st}(\bg_{\sigma,r}^F)
        \end{equation*}
   which is an isomorphism when $r \in \Z_{>0}$.

        \item[(2)] For $\sigma \preceq \tau  \in [\sA_m]$ and $f \in C^{st}(\bg_{\sigma,r}^F) $, we have $\Res^{\bg_{\sigma,r}}_{\bg_{\tau,r}}(f) \in C^{st}(\bg_{\tau,r}^F)$, and the following diagram commutes
        \begin{equation}\label{stableresdiag}
            \begin{tikzcd}
            \C[(\bt^*//W_\sigma)^F] \arrow{r}{\rho_{\sigma,r}} \arrow{d}[swap]{(c_{\sigma,\tau,r})^{-1}\res^{\sigma}_{\tau}} & C^{st}(\bg_{\sigma,r}^F) \arrow{d}{\Res^{\bg_{\sigma,r}}_{\bg_{\tau,r}}} \\
             \C[(\bt^*//W_\tau)^F] \arrow{r}{\rho_{\tau,r}} & C^{st}(\bg_{\tau,r}^F)
        \end{tikzcd}
        \end{equation}
        where $c_{\sigma,\tau,r}$ is as defined in Remark \ref{csigmataur}.
        \item[(3)] For any  $f \in C^{st}(\bg_{\sigma,r}^F) $ and $\sigma \preceq \tau  \in [\sA_m]$, we have 
        \begin{equation}\label{vanishingequation}
            \sum_{N\in \bu_{\sigma,\tau,r}^F}f(X+N)=0 \text{   for   } X \not \in \bg_{\tau,r}^F\oplus \bu_{\sigma,\tau,r}^F
        \end{equation}
    \end{itemize}
\end{proposition}
\begin{proof}
    \begin{itemize}
        \item [(1)]  Let $\varepsilon \in \C[(\bt^*//W_\sigma)^F]$. From equation \eqref{stableGIT}, we see that $\chi_{\sigma,r}^*(\varepsilon)^- \in C((\bg_{\sigma,r}^F)^*)$ and hence $f_\varepsilon=\mcF_{(\bg_{\sigma,r}^F)^*}(\chi_{\sigma,r}^*(\varepsilon)^-) \in  C(\bg_{\sigma,r}^F)$. From properties of Fourier transforms in Proposition \ref{fourierprop} $(ii)$, we have that $\mcF_{(\bg_{\sigma,r}^F)^*}(\chi_{\sigma,r}^*(\varepsilon)^-) \in C^{st}(\bg_{\sigma,r}^F)$. Proposition \ref{fourierprop} $(iv)$ shows that the map $\rho_{\sigma,r}$ is an algebra morphism, with a simple calculation showing $\rho_{\sigma,r}(\mathbbm{1})= |\bg_{\sigma,r}^F|^{1/2}\mathbbm{1}_0$. Let $ f \in C^{st}(\bg_{\sigma,r}^F) $. Then, by  the definition of stable functions, $\mcF_{\bg_{\sigma,r}^F}(f) = \chi_{\sigma,r}^*(\varepsilon) \in C((\bg_{\sigma,r}^F)^*) $ for some $\varepsilon \in \C[(\bt^*//W_\sigma)^F]$. Using \ref{fourierprop} $(ii)$, we see that $$\rho_{\sigma,r}(\varepsilon)=f_{\varepsilon}=\mcF_{(\bg_{\sigma,r}^F)^*}(\chi_{\sigma,r}^*(\varepsilon)^-) = \mcF_{(\bg_{\sigma,r}^F)^*}\circ\mcF_{\bg_{\sigma,r}^F}(f^-)=f$$
        giving surjectivity of $\rho_{\sigma,r}$. \par
        Observe that for integral depths, $m=1$, and hence we have $\mbL_\sigma \cong \G_\sigma$ and $\bg^*_{\sigma,r}= (\bg_E)^*_{\sigma,r} $. In order to show that $\rho_{\sigma,r}$ is an isomorphism in this case, it is enough to show that $\chi_{\sigma,r}: (\bg_{\sigma,r}^F)^* \rightarrow (\bt^*//W_\sigma)^F$ is surjective. With the identifications made above, this is equivalent to showing that the map of $\bar\F_q$-varieties $\bg_{\sigma,r}^* \xrightarrow{\simeq} \bg_{\sigma}^*\rightarrow \bg^*_{\sigma} // \G_\sigma$ is surjective at the level of $\F_q$-points. Note that each map is defined over $\F_q$ and hence we have a map at the level $\bg_{\sigma}^*(\F_q)\rightarrow \bg^*_{\sigma} // \G_\sigma(\F_q)$ of $\F_q$ points.  Let $\phi$ denote the map $\bg_{\sigma}^*(\bar\F_q)\rightarrow \bg^*_{\sigma} // \G_\sigma(\bar\F_q)$, with the $\bar\F_q$ varieties identified with the set of $\bar\F_q$-points and let $y \in \bg^*_{\sigma} // \G_\sigma(\F_q) \subset \bg^*_{\sigma} // \G_\sigma(\bar\F_q)$. The map $\phi$ is surjective, and $\phi^{-1}(y) \subset \bg^*_{\sigma}$ is a union of $\G_\sigma(\bar\F_q)$-orbits. There is a unique closed orbit $O \subset \phi^{-1}(y)$, and by \cite[Theorem 4]{KW1}, it corresponds to a semisimple element. Thus, we have a closed $\bar\F_q$-subvariety $O \subset \bg_\sigma^*$, and $\G_\sigma(\bar\F_q)$ acts transitively on $O(\bar\F_q)$. By Lang's theorem, we know that if $O$ is $F$-stable (i.e., defined over $\F_q$), then $O$ has an $\F_q$-point. Since $y \in  \bg^*_{\sigma} // \G_\sigma(\F_q)$, $\phi^{-1}(y)$ is $F$-stable and hence $O$ is $F$-stable. Thus, $O(\F_q) \neq \emptyset$ which shows surjectivity of the map $\bg_{\sigma}^*(\F_q)\rightarrow \bg^*_{\sigma} // \G_\sigma(\F_q)$, thereby proving our claim.
        
        \item[(2)] Let $\sigma, \tau, \tau' \in [\sA_m]$ such that $\sigma\preceq \tau$, $\sigma\preceq \tau'$ and $\tau,\: \tau'$ are $\sigma$-opposite. We know that $\bl^*_\sigma\cong \bu^*_{\sigma,\tau} \oplus \bl_\tau^* \oplus \bu^*_{\sigma,\tau'}$ Consider the following diagram 
    \begin{equation*}
        \begin{tikzcd}
        \bl_\tau^* \arrow{d}{\bar\chi_\tau} & \bl_\tau^* \oplus \bu^*_{\sigma,\tau'} \arrow{l}[swap]{\bar{pr}_{\sigma,\tau}}\arrow[hookrightarrow]{r}{\bar i_{\sigma,\tau}} & \bl^*_\sigma \arrow{d}{\bar\chi_\sigma} \\
        \bt^* // W_\tau \arrow[rr] && \bt^* // W_\sigma
    \end{tikzcd}
    \end{equation*}
    where ${\bar{pr}_{\sigma,\tau}}$ is the natural projection and $\bar i_{\sigma,\tau}$ is the inclusion. By Lemma \ref{lemmaresstable}, the diagram commutes. Now, $\bg_{\sigma,r}^* \cong \bu_{\sigma,\tau,r}^* \oplus \bg_{\tau,r}^*\oplus \bu_{\sigma,\tau',r}^*$, where we use projection maps onto subspaces to give an identification of $\bu_{\sigma,\tau,r}^*$, $\bg_{\tau,r}^*$ and  $\bu_{\sigma,\tau',r}^*$  as subspaces of $\bg_{\sigma,r}^*$. We also observe that the map $\bg^*_{\sigma,r} \xrightarrow{\iota_{E,r}} (\bg_E)^*_{\sigma,r} \xrightarrow{\theta^\sigma_{\gamma}} \bl_\sigma^*$ maps $\bu_{\sigma,\tau,r}^* \rightarrow \bu_{\sigma,\tau}^*$, $\bu_{\sigma,\tau',r}^* \rightarrow \bu_{\sigma,\tau'}^*$ and $\bg_{\tau,r}^* \rightarrow \bl_\tau^*$ using the fact that $\bg_{\sigma,r} \hookrightarrow (\bg_E)_{\sigma,r}$ maps $\bu_{\sigma,\tau,r} \rightarrow (\bu_E)_{\sigma,\tau,r}$, $\bg_{\tau,r} \rightarrow (\bg_E)_{\tau,r}$ and similar identification of duals. Thus we have the following commutative diagram, 
    \begin{equation*}
      \begin{tikzcd}
        \bg_{\tau,r}^* \arrow{d} & \bg_{\tau,r}^* \oplus \bu^*_{\sigma,\tau',r} \arrow[twoheadrightarrow]{l}\arrow[hookrightarrow]{r} & \bg^*_{\sigma,r} \arrow{d}\\
        \bl_\tau^* &\bl_\tau^* \oplus \bu^*_{\sigma,\tau'} \arrow{l}[swap]{\bar{pr}_{\sigma,\tau}}\arrow[hookrightarrow]{r}{\bar i_{\sigma,\tau}}& \bl^*_\sigma
    \end{tikzcd}  
    \end{equation*}
    with natural projection and inclusion maps. Combining the above two diagrams, we have 
    \begin{equation*}
        \begin{tikzcd}
          \bg_{\tau,r}^* \arrow{d} & \bg_{\tau,r}^* \oplus \bu^*_{\sigma,\tau',r} \arrow[twoheadrightarrow]{l}\arrow[hookrightarrow]{r} & \bg^*_{\sigma,r} \arrow{d}\\
          \bl_\tau^* \arrow{d}{\bar\chi_\tau} && \bl^*_\sigma \arrow{d}{\bar\chi_\sigma}\\
           \bt^* // W_\tau \arrow[rr] && \bt^* // W_\sigma
        \end{tikzcd}
    \end{equation*}
    Further, each of the arrows is compatible with Frobenius endomoprhism, and hence we have a similar commutative diagram for $F$-fixed points. 
 \begin{equation}\label{commdiagres}
        \begin{tikzcd}
          (\bg_{\tau,r}^F)^* \arrow[dd, bend right=40, "\chi_{\tau,r}"']\arrow{d} & (\bg_{\tau,r}^F)^* \oplus (\bu^F_{\sigma,\tau',r})^* \arrow[twoheadrightarrow]{l}[swap]{pr_{\sigma,\tau,r}}\arrow[hookrightarrow]{r}{i_{\sigma,\tau,r}} & (\bg^F_{\sigma,r})^* \arrow{d}\arrow[dd, bend left=40, "\chi_{\sigma,r}"]\\
          (\bl_\tau^F)^*\arrow{d}{\chi_\tau} && (\bl^F_\sigma)^*\arrow{d}[swap]{\chi_\sigma}\\
           (\bt^* // W_\tau)^F\arrow[rr] && (\bt^* // W_\sigma)^F
        \end{tikzcd}
    \end{equation}
   For $\varepsilon \in \C[(\bt^*//W_\sigma)^F] $ and $X^* \in (\bg_{\tau,r}^F)^*$, it follows using \eqref{commdiagres} that 
    \begin{align*}
        \Res^{(\bg_{\sigma,r})^*}_{(\bg_{\tau,r})^*}(\chi_{\sigma,r}^*(\varepsilon)^-)(X^*)&=\left|\bu_{\sigma,\tau',r}^F\right|^{-1}\sum_{N^* \in (\bu_{\sigma,\tau',r}^F)^*} \chi_{\sigma,r}^*(\varepsilon)(-X^*-N^*)\\
        &=\left|\bu_{\sigma,\tau',r}^F\right|^{-1}\sum_{N^* \in (\bu_{\sigma,\tau',r}^F)^*} \chi_{\sigma,r}^*(\varepsilon)(i_{\sigma,\tau,r}(-X^*-N^*))\\
        &=\left|\bu_{\sigma,\tau',r}^F\right|^{-1}\sum_{N^* \in (\bu_{\sigma,\tau',r}^F)^*} (\chi_{\tau,r}\circ pr_{\sigma,\tau,r})^*(\res^\sigma_\tau(\varepsilon))(-X^*-N^*)\\
        &=\left|\bu_{\sigma,\tau',r}^F\right|^{-1}\sum_{N^* \in (\bu_{\sigma,\tau',r}^F)^*}\chi_{\tau,r}^*(\res^\sigma_\tau(\varepsilon))(-X^*)\\
        &=\chi_{\tau,r}^*(\res^\sigma_\tau(\varepsilon))(-X^*)= \chi_{\tau,r}^*(\res^\sigma_\tau(\varepsilon))^-(X^*)
    \end{align*}
    Then, using Proposition \ref{fourierres} and Remark \ref{csigmataur}, we have 
    \begin{align*}
         \Res^{\bg_{\sigma,r}}_{\bg_{\tau,r}}\circ\rho_{\sigma,r}(\varepsilon)&=\Res^{\bg_{\sigma,r}}_{\bg_{\tau,r}}\circ\mcF_{(\bg_{\sigma,r}^F)^*}(\chi_{\sigma,r}^*(\varepsilon)^-)= ( c_{\sigma,\tau,r})^{-1} \mcF_{(\bg_{\tau,r}^F)^*}\circ \mathrm{Res}^{(\bg_{\sigma,r})^*}_{(\bg_{\tau,r})^*}(\chi_{\sigma,r}^*(\varepsilon)^-)\\
         &= ( c_{\sigma,\tau,r})^{-1} \mcF_{(\bg_{\tau,r}^F)^*}(\chi_{\tau,r}^*(\res^\sigma_\tau(\varepsilon))^-)=\rho_{\tau,r}(( c_{\sigma,\tau,r})^{-1}\res^\sigma_\tau(\varepsilon))
    \end{align*}
    which finishes the proof of $(2)$.

    \item[(3)] Consider the following commutative diagram 
    \begin{equation*}
        \begin{tikzcd}
            \bg_{\tau,r}^F\oplus \bu_{\sigma,\tau,r}^F \arrow[hookrightarrow]{r}\arrow[twoheadrightarrow]{d}& \bg_{\sigma,r}^F\arrow[twoheadrightarrow]{d}{P}\\
            \bg_{\tau,r}^F\arrow[hookrightarrow]{r}&\bg_{\sigma,r}^F/\bu_{\sigma,\tau,r}^F\cong \bg_{\tau,r}^F\oplus \bu_{\sigma,\tau',r}^F
        \end{tikzcd}
    \end{equation*}
    with natural inclusions and projections. Given $f \in C^{st}(\bg_{\sigma,r}^F)$, we define $h:\bg_{\sigma,r}^F/\bu_{\sigma,\tau,r}^F\cong \bg_{\tau,r}^F\oplus \bu_{\sigma,\tau',r}^F \longrightarrow \C$ as $h(X)= \sum_{Y \in P^{-1}(X)}f(Y)$. For $X \in \bg_{\tau,r}^F\oplus \bu_{\sigma,\tau',r}^F$, $h(X) = \sum_{N \in\bu_{\sigma,\tau,r}^F} f(X+N)$. If we can show that $h$ is supported on $\bg_{\tau,r}^F$, we are done. Let $V=\bg_{\tau,r}^F\oplus \bu_{\sigma,\tau',r}^F$ and consider the Fourier transform of $h$ defined as in \eqref{FT_V}. Then, for $X^*\in (\bg_{\tau,r}^F)^*\oplus (\bu_{\sigma,\tau',r}^F)^* $ 
    \begin{equation*}
                \mcF_{\bg_{\sigma,r}^F}(f)(X^*)=\left|\bg_{\sigma,r}^F\right|^{-1/2}\sum_{Y\in \bg_{\tau,r}^F\oplus \bu_{\sigma,\tau',r}^F}\Tilde\psi( X^*(Y))h(Y)= \left|\bu_{\sigma,\tau,r}^F\right|^{1/2}\mcF_V(h)(X^*)
    \end{equation*}
     For $f \in C^{st}(\bg_{\sigma,r}^F) $, we have  from \eqref{commdiagres} that $\mcF_{\bg_{\sigma,r}^F}(f)|_{(\bg_{\tau,r}^F)^*\oplus (\bu_{\sigma,\tau',r}^F)^*}$ factors through 
    \begin{equation*}
        (\bg_{\tau,r}^F)^*\oplus (\bu_{\sigma,\tau',r}^F)^*\xlongrightarrow{p_{\sigma,\tau,r}} (\bg_{\tau,r}^F)^* \xlongrightarrow{\chi_{\tau,r}}(\bt^*//W_\tau)^F \longrightarrow (\bt^*//W_\sigma)^F \longrightarrow\C 
    \end{equation*}
    and thus is constant on the fibres of the projection $ (\bg_{\tau,r}^F)^*\oplus (\bu_{\sigma,\tau',r}^F)^*\xlongrightarrow{p_{\sigma,\tau,r}} (\bg_{\tau,r}^F)^*$, and $\tilde h=\mcF_V(h)$ has the same property. In order to show that $h$ is supported on $\bg_{\tau,r}^F$, it is enough to show the same for $\mcF_{V^*}(\tilde h)=h^-$. For $X_1+X_2 \in \bg_{\tau,r}^F\oplus \bu_{\sigma,\tau',r}^F$
    \begin{align*}
        \mcF_{V^*}(\tilde h)(X_1+X_2)&= \sum_{\mathclap{\substack{Y_1^* \in (\bg_{\tau,r}^F)^*\\Y_2^*\in (\bu_{\sigma,\tau',r}^F)^*}}} \Tilde{\psi}(Y_1^*(X_1)+Y_2^*(X_2))\tilde h(Y_1^*+Y_2^*)\\
        &=\sum_{\mathclap{\substack{Y_1^* \in (\bg_{\tau,r}^F)^*\\Y_2^*\in (\bu_{\sigma,\tau',r}^F)^*}}} \Tilde{\psi}(Y_1^*(X_1))\Tilde{\psi}(Y_2^*(X_2))\tilde h(Y_1^*)\\
        &=\sum_{Y_1^* \in (\bg_{\tau,r}^F)^*} \Tilde{\psi}(Y_1^*(X_1))\tilde h(Y_1^*) \sum_{Y_2^*\in (\bu_{\sigma,\tau',r}^F)^*}\Tilde{\psi}(Y_2^*(X_2))
    \end{align*}
Thus, $\mcF_{V^*}(\tilde h)=h^-=0$ if $X_2\neq 0$, and proves that $h$ is supported on $\bg_{\tau,r}^F$. 
    \end{itemize}
\end{proof}

\begin{proposition}
    Let $\sigma \in [\sA_m]$ and $n \in N_G(T)(k)$. Then the adjoint action of $n$ induces an isomorphism 
    \begin{equation*}
         C^{st}(\bg_{\sigma,r}^F) \xlongrightarrow{\simeq}  C^{st}(\bg_{n\sigma,r}^F)  
    \end{equation*}
\end{proposition}
\begin{proof}
   The co-adjoint action of $n$ induces an isomorphism $\bt^*\xlongrightarrow{\Ad^*(n)}\bt^*$ equivariant with respect to $W_\sigma$-action (where $W_\sigma$ acts on the r.h.s via the isomorphism  $W_\sigma \xlongrightarrow{\Ad(n)} W_{n\sigma}$), and thus we have an isomorphism $\bt^*//W_\sigma \xlongrightarrow[\Ad^*(n)]{\simeq}\bt^*//W_{n\sigma}$. Since $nG_{\sigma,r}n^{-1}=G_{n\sigma}$ for all $r \in \R_{\geq 0}$, we immediately see that the following diagram commutes. 
    \begin{equation*}
        \begin{tikzcd}
           (\bg_{\sigma,r}^F)^* \arrow{r}{\chi_{\sigma,r}} \arrow{d}{\Ad^*(n)}&(\bt^*//W_\sigma)^F \arrow{d}{\Ad^*(n)}\\
          (\bg_{n\sigma,r}^F)^* \arrow{r}{\chi_{n\sigma,r}} &(\bt^*//W_{n\sigma})^F
        \end{tikzcd}
    \end{equation*}
    Let  $\varepsilon \in \C[(\bt^*//W_\sigma)^F]$ and $\prescript{n}{}\varepsilon \in \C[(\bt^*//W_{n\sigma})^F]$ such that $\prescript{n}{}\varepsilon(x)=\varepsilon(\Ad^*(n^{-1})x)$. So, $\prescript{n}{}\varepsilon=(\Ad^*(n^{-1}))^*(\varepsilon)$.  Let us similarly use $ \prescript{n^{-1}}{}{}X$ to denote $\Ad(n^{-1})X \in \bg_{\sigma,r}^F$ for $X \in \bg_{n\sigma,r}^F$ etc. Then, 
    \begin{align*}
        \mcF_{(\bg_{n\sigma,r}^F)^*}(\chi_{n\sigma,r}^*(\prescript{n}{}\varepsilon)^-)(X)
        &= \left|\bg_{n\sigma,r}^F\right|^{-1/2}\sum_{Y^*\in (\bg_{n\sigma,r}^F)^*}\Tilde{\psi}(Y^*(X))\prescript{n}{}\varepsilon(\chi_{n\sigma,r}(-Y^*))\\
        &=\left|\bg_{\sigma,r}^F\right|^{-1/2}\sum_{Y^*\in (\bg_{n\sigma,r}^F)^*}\Tilde{\psi}(Y^*(X))\varepsilon(\chi_{\sigma,r}(-\Ad^*(n^{-1})Y^*))\\
        &=\left|\bg_{\sigma,r}^F\right|^{-1/2}\sum_{Y^*\in (\bg_{n\sigma,r}^F)^*}\Tilde{\psi}(\prescript{n^{-1}}{}{}Y^*(\prescript{n^{-1}}{}X))\varepsilon(\chi_{\sigma,r}(-\prescript{n^{-1}}{}Y^*))\\
        &=\left|\bg_{\sigma,r}^F\right|^{-1/2}\sum_{Y_1^* \in (\bg_{\sigma,r}^F)^*}\Tilde{\psi}(Y_1^*(\prescript{n^{-1}}{}X))\varepsilon(\chi_{\sigma,r}(-Y_1^*))\\
        &=\left|\bg_{\sigma,r}^F\right|^{-1/2}\sum_{Y_1^* \in (\bg_{\sigma,r}^F)^*}\Tilde{\psi}(Y_1^*(\prescript{n^{-1}}{}X))\chi_{\sigma,r}^*(\varepsilon)^-(Y_1^*))\\
        &=\mcF_{(\bg_{\sigma,r}^F)^*}(\chi_{\sigma,r}^*(\varepsilon)^-)(\prescript{n^{-1}}{}X)=\Ad(n^{-1})^*(\rho_{\sigma,r}(\varepsilon))(X)
    \end{align*}
     which shows that the following diagram 
    \begin{equation}\label{adjnstable}
        \begin{tikzcd}
            \C[(\bt^*//W_\sigma)^F]\arrow[d,"\wr", "(\Ad^*(n^{-1}))^*"'] \arrow{r}{\rho_{\sigma,r}} & C^{st}(\bg_{\sigma,r}^F)\arrow{d}{\Ad(n^{-1})^*} \\
            \C[(\bt^*//W_{n\sigma})^F]\arrow{r}{\rho_{n\sigma,r}}& C^{st}(\bg_{n\sigma,r}^F)
        \end{tikzcd}
    \end{equation}
    commutes and finishes the proof.
\end{proof}

\section{From Stable functions to positive depth Bernstein center}\label{section:stablefnstocenter}
We will use the theory of stable functions developed in the previous section to construct elements in the depth-$r$ center for $r \in \frac{1}{m}\Z_{>0} \subset \Z_{(p)}\cap \Q_{>0}$, and  attach parameters to smooth irreducible representations, which only depends upon the Moy-Prasad type of the representation. \paragraph{}

Let $W = N_G(T)/T$ be the Weyl group of $G$, and  $\C[(\bt^*//W)^F]$ be the algebra of complex-valued functions on $(\bt^*//W)^F$. We again fix $\varpi \in \mf O_k$, $E$ and $\gamma \in \mf O_E$ with $\gamma^m= \varpi$ as in Section \ref{subsectionstablefunction}. We will construct a map from $\C[(\bt^*//W)^F]$ to the depth-$r$ Bernstein center $\ZZ^r(G)$, using our description of $\ZZ^r(G)$ in Section \ref{section:fracdepthdescription}. Let $\sigma \in [\Bar\mcC_m]$. Note that $W_\sigma = N_{\mbL_\sigma}(\T)/\T$ agrees with the image of $N_{G(E)_{\sigma,0}}(T)$ in $W$ (see \cite{debtotallyram}, Lemma 7.2.1), and hence there is an embedding $W_\sigma \hookrightarrow W$. Thus, we have a map 
$$t_{\sigma} : (\bt^*//W_\sigma)^F \longrightarrow (\bt^*//W)^F$$
We have a natural inclusion map 
\begin{equation*}
    \iota_{\sigma,r}:C(\bg_{\sigma,r}^F) \rightarrow C_c^\infty\left(\frac{G_{\sigma,r}/G_{\sigma,r+}}{G_{\sigma,0}}\right)\rightarrow C_c^\infty\left(\frac{G(k)/G_{\sigma,r+}}{G_{\sigma,0}}\right)=\M^r_\sigma
\end{equation*}
sending $ C(\bg_{\sigma,r}^F) \ni f: \bg_{\sigma,r}^F \rightarrow\C$ to $\iota_{\sigma,r}( f) \in \M^r_\sigma$ supported in $G_{\sigma,r} \subset G(k)$ given by 
\begin{equation*}
    \iota_{\sigma,r}( f) :G_{\sigma,r}\rightarrow G_{\sigma,r}/G_{\sigma,r+} \xrightarrow{\simeq} \bg_{\sigma,r}^F \xrightarrow{f}\C
\end{equation*}
We define $c^\mu_{\sigma,r} $ to be the constant $\mu(G_{\sigma,r+})\left|\bg_{\sigma,r}^F\right|^{1/2}$. Consider the composed map 
\begin{equation*}
    j_{\sigma,r}: \C[(\bt^*//W)^F] \xlongrightarrow{(c^\mu_{\sigma,r})^{-1}t^*_{\sigma}} \C[(\bt^*//W_\sigma)^F]\xlongrightarrow{\rho_{\sigma,r}} C^{st}(\bg_{\sigma,r}^F) \xlongrightarrow{\iota_{\sigma,r}}\M^r_\sigma
\end{equation*}
Observe that the map $j_{\sigma,r}$ is an algebra homomorphism sending the unit $\mathbbm{1} \in \C[(\bt^*//W)^F]$ to $\delta_{G_{\sigma,r+}}\in \M^r_\sigma$. 
\begin{remark}
    Note that some of the the individual maps in the definition of $j_{\sigma,r}$ are just vector space maps, although the composition is an algebra morphism. 
\end{remark}
Our main idea is to show that the following diagrams commute and use that to construct a map into $\ZZ^r(G)$, using the limit description. Let $\sigma,\tau \in [\bar\mcC_m]$, $\sigma \preceq \tau$. Then, we have  
\begin{equation*}
    \begin{tikzcd}
        \C[(\bt^*//W)^F]\arrow{r}{j_{\sigma,r}} \arrow{rd}{j_{\tau,r}} & \M^r_\sigma\arrow{d}{\phi^r_{\sigma,\tau}}\\
        & \M^r_\tau
    \end{tikzcd}
\end{equation*}
Further, let $n \in N_G(T)(k)$ such that $n\mcC=\mcC$ and $\sigma_1, \sigma_2 \in [\Bar{\mcC}_m]$ such that $n\sigma_1\preceq \sigma_2$. Then
\begin{equation*}
    \begin{tikzcd}
        \C[(\bt^*//W)^F]\arrow{r}{j_{\sigma_1,r}} \arrow{rd}{j_{\sigma_2,r}} & \M^r_{\sigma_1}\arrow{d}{\phi^r_{\sigma_1,\sigma_2,n}}\\
        & \M^r_{\sigma_2}
    \end{tikzcd}
\end{equation*}
We do that in several steps using a series of lemmas. 
\begin{lemma}
 Let $\sigma, \tau \in [\Bar \mcC_m]  $ such that $\sigma\preceq \tau $. Then $c^\mu_{\sigma,r} \cdot c_{\sigma,\tau,r}=c^\mu_{ \tau,r}$  
\end{lemma}
\begin{proof}
    \[
    \frac{c^\mu_{ \tau,r}}{c^\mu_{\sigma,r}}=\frac{\mu(G_{\tau,r+})}{\mu(G_{\sigma,r+})}\cdot \frac{\left|\bg_{\tau,r}^F\right|^{1/2}}{\left|\bg_{\sigma,r}^F\right|^{1/2}}=\frac{\left|\bu_{\sigma,\tau,r}^F\right|\cdot \left|\bg_{\tau,r}^F\right|^{1/2}}{\left|\bu_{\sigma,\tau,r}^F\right|^{1/2}\cdot \left|\bg_{\tau,r}^F\right|^{1/2}\cdot \left|\bu_{\sigma,\tau',r}^F\right|^{1/2}}=c_{\sigma,\tau,r}
    \]
\end{proof}

\begin{lemma}\label{lemmatmodWtostable}
    For $\sigma \in [\bar\mcC_m]$, we have a map 
    \[
    i_{\sigma,r}: \C[(\bt^*//W)^F] \longrightarrow C^{st}(\bg_{\sigma,r}^F)
    \]
    such that for $\sigma \preceq\tau \in [\bar\mcC_m]$, we have the following commutative diagram. 
    \begin{equation}\label{tmodWtostable}
        \begin{tikzcd}
             \C[(\bt^*//W)^F] \arrow{r}{i_{\sigma,r}}\arrow{rd}{i_{\tau,r}}& C^{st}(\bg_{\sigma,r}^F)\arrow{d}{\Res^{\bg_{\sigma,r}}_{\bg_{\tau,r}}}\\
             &C^{st}(\bg_{\tau,r}^F)
        \end{tikzcd}
    \end{equation}
\end{lemma}
\begin{proof}
For $\sigma \in [\bar\mcC_m]$, define $i_{\sigma,r} :\C[(\bt^*//W)^F] \longrightarrow C^{st}(\bg_{\sigma,r}^F)$ as the composition 
\begin{equation*}
     \C[(\bt^*//W)^F] \xlongrightarrow{(c^\mu_{\sigma,r})^{-1}t^*_{\sigma}} \C[(\bt^*//W_\sigma)^F]\xlongrightarrow{\rho_{\sigma,r}} C^{st}(\bg_{\sigma,r}^F)
\end{equation*}
For $\sigma \preceq \tau \in [\bar \mcC_m]$, it immediately follows from the previous lemma that the following diagram commutes.
\begin{equation*}
    \begin{tikzcd}
      \C[(\bt^*//W)^F] \arrow{r}{(c^\mu_{\sigma,r})^{-1}t^*_{\sigma}} \arrow{rd}[swap]{(c^\mu_{\tau,r})^{-1}t^*_{\tau}}  & \C[(\bt^*//W_\sigma)^F]\arrow{d}{(c_{\sigma,\tau,r})^{-1}\res^\sigma_\tau}\\
     & \C[(\bt^*//W_\tau)^F]
    \end{tikzcd}
\end{equation*}
Combining the above diagram and \eqref{stableresdiag}, we get 
\begin{equation*}
    \begin{tikzcd}
        \C[(\bt^*//W)^F] \arrow{r}{(c^\mu_{\sigma,r})^{-1}t^*_{\sigma}} \arrow{rd}[swap]{(c^\mu_{\tau,r})^{-1}t^*_{\tau}}  & \C[(\bt^*//W_\sigma)^F]\arrow{d}{(c_{\sigma,\tau,r})^{-1}\res^\sigma_\tau} \arrow{r}{\rho_{\sigma,r}} & C^{st}(\bg_{\sigma,r}^F) \arrow{d}{\Res^{\bg_{\sigma,r}}_{\bg_{\tau,r}}} \\
        &\C[(\bt^*//W_\tau)^F] \arrow{r}{\rho_{\tau,r}} & C^{st}(\bg_{\tau,r}^F)
    \end{tikzcd}
\end{equation*}
which proves the lemma since $i_{\sigma,r}= \rho_{\sigma,r}\circ(c^\mu_{\sigma,r})^{-1}t^*_{\sigma}$. 
\end{proof}

\begin{lemma}
    For $\sigma \preceq \tau \in [\bar\mcC_m]$, the map $\iota_{\sigma,r}:C^{st}(\bg_{\sigma,r}^F) \rightarrow \M^r_\sigma$ fits into the following commutative diagram. 
    \begin{equation}\label{stabletoMrsigma}
        \begin{tikzcd}
          C^{st}(\bg_{\sigma,r}^F) \arrow{r}{\iota_{\sigma,r}}  \arrow{d}[swap]{\Res^{\bg_{\sigma,r}}_{\bg_{\tau,r}}}&\M^r_\sigma\arrow{d}{\phi^r_{\sigma,\tau}}\\
          C^{st}(\bg_{\tau,r}^F) \arrow{r}{\iota_{\tau,r}}&\M^r_\tau
        \end{tikzcd}
    \end{equation}
\end{lemma}
\begin{proof}
For $f\in  C^{st}(\bg_{\sigma,r}^F)$, we immediately see that $\phi^r_{\sigma,\tau}\circ\iota_{\sigma,r}(f)= \iota_{\sigma,r}(f)*\delta_{G_{\tau,r+}}$ is supported inside $G_{\sigma,r+}$. Let $x \in G_{\sigma,r+}$ with image $\bar x \in \bg_{\sigma,r}^F $. Then,
\begin{align*}
     \iota_{\sigma,r}(f)*\delta_{G_{\tau,r+}}(x)&= \frac{1}{\mu(G_{\tau,r+})}\int_{G(k)}\iota_{\sigma,r}(f)(xy^{-1})\delta_{G_{\tau,r+}}(y)d\mu(y)=\frac{1}{\mu(G_{\tau,r+})}\int_{G_{\tau,r+}}\iota_{\sigma,r}(f)(xy^{-1})d\mu(y)\\
     &=\frac{\mu(\delta_{G_{\sigma,r+}})}{\mu(G_{\tau,r+})}\sum_{\bar u \in G_{\tau,r+}/G_{\sigma,r+}} \int_{G_{\sigma,r+}}\iota_{\sigma,r}(f)(xuz)d\mu(z)  = \left|\bu_{\sigma,\tau,r}^F\right|^{-1}\sum_{N\in \bu_{\sigma,\tau,r}^F}f(\bar x+N)
\end{align*}
Since $f\in  C^{st}(\bg_{\sigma,r}^F)$, we know from \ref{vanishingequation} that 
\begin{equation*}
    \left|\bu_{\sigma,\tau,r}^F\right|^{-1}\sum_{N\in \bu_{\sigma,\tau,r}^F}f(\bar x+N)=\begin{cases}
        0, & \text{if } \bar x \not\in \bg_{\tau,r}^F\oplus \bu_{\sigma,\tau,r}^F\\
        \Res^{\bg_{\sigma,r}}_{\bg_{\tau,r}}(f)(\bar x), & \text{if } \bar x \in \bg_{\tau,r}^F\oplus \bu_{\sigma,\tau,r}^F
    \end{cases}
\end{equation*}
Thus, we see that $\iota_{\sigma,r}(f)*\delta_{G_{\tau,r+}}(x)$ is supported on $G_{\tau,r}\subset G_{\sigma,r}$, and for $x \in G_{\tau,r}$ with image $\bar x \in \bg_{\tau,r}^F$
\begin{equation*}
    \phi^r_{\sigma,\tau}\circ\iota_{\sigma,r}(f)(x)= \Res^{\bg_{\sigma,r}}_{\bg_{\tau,r}}(f)(\bar x)= \iota_{\tau,r}\circ \Res^{\bg_{\sigma,r}}_{\bg_{\tau,r}}(f)(x)
\end{equation*}
which proves the commutativity of the diagram. 
\end{proof}
\begin{remark}
    Note that it was essential to consider $C^{st}(\bg_{\sigma,r}^F)$ instead of $C(\bg_{\sigma,r}^F)$, since the vanishing property \eqref{vanishingequation} was instrumental to the proof.
\end{remark}

\begin{lemma}\label{lemmaadjointtmodWtoMrsigma}
    Let $\sigma_1,\sigma_2\in [\Bar{\mcC}_m]$ and $n \in N_G(T)(k)$ such that $n\mcC=\mcC$ and $ n \sigma_1 \preceq \sigma_2$. Then, we have a commutative diagram
    \begin{equation*}
    \begin{tikzcd}
        \C[(\bt^*//W)^F]\arrow{r}{j_{\sigma_1,r}} \arrow{rd}{j_{\sigma_2,r}} & \M^r_{\sigma_1}\arrow{d}{\phi^r_{\sigma_1,\sigma_2,n}}\\
        & \M^r_{\sigma_2}
    \end{tikzcd}
\end{equation*}
\end{lemma}
\begin{proof}
    Since $c^\mu_{\sigma_1,r}=c^\mu_{n\sigma_1,r}$, using \eqref{adjnstable} , we have 
    \begin{equation*}
         \begin{tikzcd}
        \C[(\bt^*//W)^F] \arrow{r}{(c^\mu_{\sigma_1,r})^{-1}t^*_{\sigma_1}} \arrow[dr,"(c^\mu_{n\sigma_1,r})^{-1}t^*_{n\sigma_1}"'] & \C[(\bt^*//W_{\sigma_1})^F]\arrow{d}{\wr} \arrow{r}{\rho_{\sigma_1,r}} & C^{st}(\bg_{\sigma_1,r}^F) \arrow{d}{\wr} \\
        &\C[(\bt^*//W_{n\sigma_1})^F] \arrow{r}{\rho_{n\sigma_1,r}} & C^{st}(\bg_{n\sigma_1,r}^F)
    \end{tikzcd}
    \end{equation*}
Then, applying  Lemma \ref{lemmatmodWtostable} to $n\sigma_1\preceq \sigma_2$ and combining with the above diagram, we have the following.
\begin{equation*}
    \begin{tikzcd}
    &  C^{st}(\bg_{\sigma_1,r}^F) \arrow{d}{\wr} \\
     \C[(\bt^*//W)^F] \arrow{ur}{i_{\sigma_1,r}} \arrow[r,"i_{n\sigma_1,r}" ]  \arrow[rd, "i_{\sigma_2,r}" ] &  C^{st}(\bg_{n\sigma_1,r}^F) \arrow{d}{\Res^{\bg_{n\sigma_1,r}}_{\bg_{\sigma_2,r}}} \\
     & C^{st}(\bg_{\sigma_2,r}^F)
     \end{tikzcd}
\end{equation*}
Further, the adjoint action of $n$ gives us an isomorphism $\Ad(n^{-1})^*:\M^r_{\sigma_1} \xrightarrow{\simeq}\M^r_{n\sigma_1}$ and hence a commutative diagram 
   \begin{equation*}
        \begin{tikzcd}
         C^{st}(\bg_{\sigma_1,r}^F) \arrow{r}{\iota_{\sigma_1,r}} \arrow{d}{\wr} &\M^r_{\sigma_1}\arrow{d}{\wr} \\
          C^{st}(\bg_{n\sigma_1,r}^F) \arrow{r}{\iota_{n\sigma_1,r}}  \arrow{d}[swap]{\Res^{\bg_{n\sigma_1,r}}_{\bg_{\sigma_2,r}}}&\M^r_{n\sigma_1}\arrow{d}{\phi^r_{n\sigma_1,\sigma_2}}\\
          C^{st}(\bg_{\sigma_2,r}^F) \arrow{r}{\iota_{\sigma_2,r}}&\M^r_{\sigma_2}
        \end{tikzcd}
    \end{equation*}
    Since $j_{\sigma,r}= \iota_{\sigma,r}\circ i_{\sigma,r} $, combining the above two diagrams proves the lemma.  
\end{proof}
If $(\pi,V)$ is a smooth representation of $G(k)$ depth $r$, we know from \cite{MP94} that $\exists \: x \in \B(G,k)$ such that $V^{G_{x,r+}}\neq 0$. Since the action of $G(k)$ on chambers is transitive and $r \in \frac{1}{m}\Z_{>0}$, without loss of generality we can assume that $x \in \Bar\mcC$ and $\exists \: \sigma \in [\bar\mcC_m]$ such that $V^{G_{\sigma,r+}}\neq 0$.
\begin{theorem}\label{thmtmodWtocenter}
    There is an algebra homomorphism 
    \begin{equation*}
        \xi^r: \C[(\bt^*//W)^F] \longrightarrow \ZZ^r(G)
    \end{equation*}
    such that for any representation $(\pi,V)$ of depth $r$ and $0\neq v \in V^{G_{\sigma,r+}} $, $\sigma \in [\bar\mcC_m]$, we have 
    \begin{equation}
         \xi^r(\lambda)(v) = \mu(G_{\sigma,r+})\sum_{X\in \bg_{\sigma,r}^F}i_{\sigma,r}(\lambda)(X)\pi^{G_{\sigma,r+}}(X)v
    \end{equation}
    where $i_{\sigma,r}: \C[(\bt^*//W)^F] \rightarrow C^{st}(\bg_{\sigma,r}^F)$ is the map defined in Lemma \ref{lemmatmodWtostable} and $\pi^{G_{\sigma,r+}}$ denotes the natural representation of $\bg_{\sigma,r}^F$ on $V^{G_{\sigma,r+}}$. If $(\pi, V)$ is a smooth irreducible representation of depth $<r$ and $0 \neq v\in   V^{G_{\sigma,r}}$, then $\xi^r(\lambda)(v)= \lambda(\bar0)v$. 
\end{theorem}
\begin{proof}
     Let $\sigma,\tau \in [\bar\mcC_m]$ such that $\sigma \preceq \tau$.  Combining \eqref{tmodWtostable} and \eqref{stabletoMrsigma}, we see that  the following diagram commutes. 
    \begin{equation*}
    \begin{tikzcd}
        \C[(\bt^*//W)^F]\arrow{r}{j_{\sigma,r}} \arrow{rd}{j_{\tau,r}} & \M^r_\sigma\arrow{d}{\phi^r_{\sigma,\tau}}\\
        & \M^r_\tau
    \end{tikzcd}
\end{equation*}
   Let $\sigma_1,\sigma_2\in [\Bar{\mcC}_m]$ and $n \in N_G(T)(k)$ such that $n\mcC=\mcC$ and $ n \sigma_1 \preceq \sigma_2$. Then, Lemma \ref{lemmaadjointtmodWtoMrsigma} gives us
    \begin{equation*}
    \begin{tikzcd}
        \C[(\bt^*//W)^F]\arrow{r}{j_{\sigma_1,r}} \arrow{rd}{j_{\sigma_2,r}} & \M^r_{\sigma_1}\arrow{d}{\phi^r_{\sigma_1,\sigma_2,n}}\\
        & \M^r_{\sigma_2}
    \end{tikzcd}
\end{equation*}
Thus, we have a map from $\C[(\bt^*//W)^F]$ to the inverse system $\{\M^r_\sigma\}_{\sigma \in [\bar\mcC_m]}$, and hence a map $\C[(\bt^*//W)^F] \rightarrow   \lim_{\sigma\in [\Bar{\mcC}_m] } \M^r_{\sigma}=A^r(G)\cong \ZZ^r(G)$. We define this map to be $\xi^r$. Then, 
\begin{align*}
    \xi^r(\lambda)(v)&=\xi^r(\lambda)(\delta_{G_{\sigma,r+}}(v))= \xi^r(\lambda)(\delta_{G_{\sigma,r+}})(v)=j_{\sigma,r}(\lambda)(v)=\int_{G(k)}j_{\sigma,r}(\lambda)(x)\pi(x)v\;d\mu(x)\\
    &= \int_{G_{\sigma,r}}(\iota_{\sigma,r}\circ i_{\sigma,r}(\lambda))(x)\pi(x)v\;d\mu(x)= \mu(G_{\sigma,r+})\sum_{\bar x\in \bg_{\sigma,r}^F}i_{\sigma,r}(\lambda)(\bar x)\pi^{G_{\sigma,r+}}(\bar x )v
\end{align*}
where $\bar x$ is the image of $x \in G_{\sigma,r}$ under the projection $G_{\sigma,r}\rightarrow G_{\sigma,r}/G_{\sigma,r+}\cong \bg_{\sigma,r}^F$. Now, if $(\pi,V)$ has depth $<r$, following the same steps, we have 
\begin{equation*}
    \xi^r(\lambda)(v)=\mu(G_{\sigma,r+})\sum_{\bar x\in \bg_{\sigma,r}^F}i_{\sigma,r}(\lambda)(\bar x)v
\end{equation*}
Let $\mathbbm 1 \in C(\bg_{\sigma,r}^F)$ denote the function which takes the value $1$ at all points. Then,
\begin{align*}
    \mu(G_{\sigma,r+})\sum_{\bar x\in \bg_{\sigma,r}^F}i_{\sigma,r}(\lambda)(\bar x)&= \mu(G_{\sigma,r+})|\G_\sigma^F|\left(i_{\sigma,r}(\lambda), \mathbbm{1}\right)=\mu(G_{\sigma,r+})|\G_\sigma^F|\left(\mcF_{\bg_{\sigma,r}}(i_{\sigma,r}(\lambda)), \mcF_{\bg_{\sigma,r}}(\mathbbm{1})\right)\\
    &=\mu(G_{\sigma,r+})|\G_\sigma^F|\left( (c^{\mu}_{\sigma,r})^{-1}\chi_{\sigma,r}^*(t^*_\sigma(\lambda)), |\bg_{\sigma,r}^F|^{1/2}\mathbbm 1_0\right)=|\G_\sigma^F|\left(\chi_{\sigma,r}^*(t^*_\sigma(\lambda)),\mathbbm 1_0\right)\\
    &=\chi_{\sigma,r}^*(t^*_\sigma(\lambda))(0)=\lambda(t_\sigma\circ\chi_{\sigma,r}(0))=\lambda(\bar 0).
\end{align*}
\end{proof}
\begin{remark}\label{remarkFrobstructure}
    Observe that if $r =\frac{i}{m}$, the map $\xi^r$ depends on the choice of $\gamma^i \in \mf O_E$, since $\rho_{\sigma,r}$ depends on $\gamma^i$. Recall the diagram \ref{stableGIT} which forms the basis of our construction.
    \begin{equation*}
    \begin{tikzpicture}[baseline=(current  bounding  box.center)]
        \node (A) at (0,0) {$ (\bg^F_{\sigma,r})^*$};
        \node (B) at (2,0) {$((\bg_E)^F_{\sigma,r})^* $};
        \node (D) at (4,0) {$(\bl_\sigma^F)^*$};
        \node (C) at (6,0) {$(\bl_\sigma^*//\mbL_\sigma)^F $};
        \node (E) at (3,-2) {$((\bg_{\sigma,r})^*//\G_\sigma)^F$};
        \node (F) at (8.8,0) {$(\bt^*//W_\sigma)^F$};
        \draw[->] (A)--(B);
         \draw[->] (B)edge node[above] {$\simeq$}(D);
        \draw[->] (D)--(C);
         \draw[->] (A)--(E);
         \draw[->] (E)--(C);
          \draw[->] (C)edge node[above] {$\simeq$}(F);
    \end{tikzpicture}
\end{equation*}
Note that $\bg_{\sigma,r}^*$, $\G_\sigma$ and $\bt^*//W_\sigma$ have natural $\F_q$-structure since $T$ splits over $k$. The $\F_q$-structure on $(\bg_E)_{\sigma,r}^*$, $\bl_{\sigma}^*$, $\mbL_\sigma$ and $\bl_\sigma^*//\mbL_\sigma$ depends on the choice of $\gamma \in \mf O_E$, while the isomorphism $(\bg_E)_{\sigma,r}^*\xrightarrow{\simeq} \bl_\sigma^*$ depends only on the choice of $\gamma^i$. However, if we have a different $\gamma'\in \mf O_E$ with $(\gamma')^i=\gamma^i$, the diagram still holds true for the different $\F_q$-structure on $(\bg_E)_{\sigma,r}^*$, $\bl_{\sigma}^*$, $\mbL_\sigma$ and $\bl_\sigma^*//\mbL_\sigma$ and the same $\F_q$-structure $\bg_{\sigma,r}^*$, $\G_\sigma$ and $\bt^*//W_\sigma$. Also, the isomorphism $(\bg_E)_{\sigma,r}^*\xrightarrow{\simeq} \bl_\sigma^*$ still remains the same and is defined over $\F_q$, and we have a map of Frobenius fixed points, with the Frobenius now attached to the new $\F_q$-structure. Thus, the map $\xi^r$ depends only on $\gamma^i$, and we don't need to ``remember" the $\F_q$-structure induced by $\gamma$. More generally, the isomorphism and hence $\xi^r$ only depends on a choice of  $\nu \in k^t$ such that $\rv(\nu)=r$. 
\end{remark}

\section{Parameters attached to smooth irreducible representations of positive depth}\label{section: Parameters attached to repns of positive depth}
We use the maps $\xi^r$ constructed in the previous section to attach parameters to smooth irreducible representations of depth $r$. We briefly describe restricted Langlands parameters attached to Moy-Prasad types, following \cite{chendebackertsai}, and show that the parameters we attached to smooth irreducible representations are same as the ones described in the afore-mentioned work.

\subsection{Depth-r Deligne Lusztig parameters}\label{section:positiveDLparameters}

\begin{definition}
    For $c \in \bar\F_q$, we denote by $j_c: \bt^*//W \xlongrightarrow{\sim}\bt^*//W$ the isomorphism induced by $\bt^*\xrightarrow{\times c} \bt^*$.  The depth-$r$ Deligne-Lusztig parameters of $G(k)$  are defined to be the set 
\begin{equation*}
    \text{DL}_r:=\{(\nu,\theta_{\nu})\:|\: \nu \in k^t, \:\rv( \nu)=r , \: \theta_{\nu}\in (\bt^*//W)(\bar\F_q)\}/\sim 
\end{equation*}
where $(\nu_1,\theta_{\nu_1}) \sim (\nu_2,\theta_{\nu_2})$ if $\theta_{\nu_1} = j_c(\theta_{\nu_2})$ for $c = \nu_1/\nu_2 + \mf m_{k^t} \in \kappa_{k^t}=\bar\F_q$.
\end{definition}
Let $(\pi, V)$ be a smooth irreducible representation of depth $r>0$, where $r =\frac{i}{m}\in  \Z_{(p)}\cap \Q_{>0}$. Let $\nu =\gamma^i$, where $\gamma \in \mf O_E$ is as described earlier. Then, $z \in \ZZ^r(G)$ acts on $\pi$ via a constant, and composing $\xi^r$ and the evaluation map $\ZZ^r(G) \rightarrow \text{End}(\pi)=\C$, we obtain a map 
\begin{equation*}
     \C[(\bt^*//W)^F] \xrightarrow{\xi^r}\ZZ^r(G) \rightarrow \text{End}(\pi)=\C
\end{equation*}
and hence an element of $(\bt^*//W)^F$, say $\tilde\theta_\nu(\pi)$ given by $\xi^r(\lambda)v= \lambda(\tilde\theta_\nu(\pi))v$ for each $\lambda \in \C[(\bt^*//W)^F]$. Note that $\tilde\theta_\nu$ gives a map $\Irr(G)_r \rightarrow (\bt^*//W)^F$ which depends on the choice of $\nu$. \par
Thus, to each smooth irreducible representation of depth $r \in  \Z_{(p)}\cap \Q_{>0}$, we attach can attach a depth-$r$ Deligne-Lusztig parameter $\Theta_r(\pi) \in \mathrm{DL}_r$ given by the equivalence class of $(\nu,\tilde\theta_\nu(\pi))$,with $\nu$ chosen as described earlier and we have a map $\Theta_r: \Irr (G)_r\rightarrow \DL_r$ defined by this assignment. Further, we observe from our construction that $\tilde\theta_\nu(\pi) \in (\bt^*//W)(\F_q) \subset (\bt^*//W)(\bar\F_q)$, which happens in this case because our group is $k$-split.

\begin{remark}
    When we attach a depth-$r$ Deligne-Lusztig parameter to $\pi\in \Irr(G)_r$, the element $\nu \in k^t$ is not an arbitrary element in $k^t$ with $\rv(\nu)=r$. We have chosen it such that it lies in a finite tamely ramified extension $E'$ with $e(E'/k)=m$. Such an extension is not unique, and more generally it lies in a finite tamely ramified Galois extension $M$ of $k$ such that $e(M/k)\cdot r \in \Z$. 
\end{remark}

\subsection{Deligne-Lusztig parameters attached to Moy-Prasad types}
Recall the definition of a Moy-Prasad type of depth $r$ for $r \in \Q_{>0}$.
\begin{definition}
    For $r \in \Q_{>0}$, a Moy-Prasad type of depth $r$  for $G(k)$ is a pair $(x, \X)$ where $x \in \B(G,k)$ and $\X \in \fg^*(k)_{x,-r}/\fg^*(k)_{x,-r+} \cong (\bg_{x,r}^F)^*$. Let $\mathrm{MP}(r)$ denote the set of Moy-Prasad types of depth $r$.
\end{definition}
A pair  $(x, \X) \in \mathrm{MP}(r)$ is called non-degenerate if the coset $\X = X^* + \fg^*(k)_{x,-r+}$ representing $\X$ does not contain any nilpotent elements (Check \cite{chendebackertsai} Section 4.2.1 or \cite{MP94} Section 3.5 for the definition of nilpotent). Two Moy-Prasad types of positive depth $(x,\X)$ and $(y,\mf Y)$ are said to be associates if they have the same depth $r$ and 
\begin{equation*}
    \Ad^*(G(k))(X^* + \fg^*(k)_{x,-r+}) \cap \Ad^*(G(k))(Y^* + \fg^*(k)_{x,-r+}) \neq \emptyset
\end{equation*}
where $X^* + \fg^*(k)_{x,-r+}$ (resp. $Y^* + \fg^*(k)_{x,-r+}$) is the coset realizing $\X$ (resp. $\mf Y$).  
\begin{remark}
    A Moy-Prasad type of depth $r$ is essentially the same as an unrefined Minimal $K$-type of depth $r$, as defined in \cite{MP94}, Section 5. For $r>0$, the character $\chi$ in the definition of a minimal $K$-type can be identified with an element of $\X\in \fg^*(k)_{x,-r}/\fg^*(k)_{x,-r+} \cong (\bg_{x,r}^F)^*$ using the fixed additive character $\tilde \psi :\F_q\rightarrow \C^\times$. The notions of non-degenerate and associates are also exactly the same.  
\end{remark}
 A smooth irreducible representation $(\pi,V)$ of depth $r$ is said to contain $(x,\X)\in \mathrm{MP}(r)$ as a Moy-Prasad type if the natural representation $\pi^{G_{x,r+}}$ of $G_{x,r}/G_{x,r+} \cong \fg(k)_{x,r}/\fg(k)_{x,r+}$ on $V^{G_{x,r+}}$ contains $\tilde\psi\circ \X$ as a one-dimensional sub-representation. By \cite[Theorem 5.2]{MP94}, we know that any smooth irreducible representation contains a non-degenerate Moy-Prasad type of the same depth. Further, any two Moy-Prasad types contained in $\pi$ are associates of each other. \par
We describe how to attach Deligne-Lusztig parameters to Moy-Prasad types, restricting to the case where $G$ is split over tamely ramified extension. Just for the next few paragraphs, assume that $G$ is not necessarily $k$-split, but split over a tamely ramified extension and let $T$ be a maximal $k$ torus such that the maximal $k$-split subtorus in $T$ is a maximal $k$-split torus, and the maximal $K$-split subtorus in $T$ is a maximal $K$-split torus. We keep the same condition on the depth $r$, with $r \in \frac{1}{m}\Z_{>0} \subset \Z_{(p)}$. Let $(x,\X)\in \mathrm{MP}(r)$ be a Moy-Prasad type of depth $r$, and $M$ be a finite Galois extension of $k$ such that $r\cdot e (M/k)\in \Z$ and $G$ splits over $M$. Let $\nu \in M$ be such that $\rv(\nu) =r$. The pair $(\nu, M)$ is $(x, \X)$-adapted as per Definition 22 in \cite[Section 4.1]{chendebackertsai}. For $x \in \B(G,k)$, choose $h\in G(k)$ such that $hx\in \sA=\A_T$ and let $W^M_{hx}:= N_{(\G_{M^{u}})_{hx}}(\T)/\T$. Consider the following commutative diagram : 

\begin{equation}\label{DLMoyPrasad}
    \begin{tikzpicture}[baseline=(current bounding box.center)]
        \node (A) at (0,0) {$ \fg^*(k)_{x,-r}/\fg^*(k)_{x,-r+}$};
        \node (B) at (0,-2) {$\fg^*(k)_{hx,-r}/\fg^*(k)_{hx,-r+}$};
        \node (B1) at (0,-4) {$\fg^*(M)_{hx,-r}/\fg^*(M)_{hx,-r+}$};
        \node (C) at (0,-6) {$\fg^*(M^{u})_{hx,-r}/\fg^*(M^u)_{hx,-r+}$};
        \node (D) at (0,-8) {$\ft^*(M^u)_{-r}/\ft^*(M^u)_{-r+}$};
        \node (E) at (6,-6) {$\fg^*(M^{u})_{hx,0}/\fg^*(M^u)_{hx,0+}$};
        \node (F) at (6,-8) {$\ft^*(M^u)_{0}/\ft^*(M^u)_{0+}=\bt^*$};
        \node (G) at (12,-6) {$(\bg_{M^u})^*_{hx,0}//(\G_{M^u})_{hx,0}$};
        \node (H) at (12,-8) {$\bt^*//W^{M}_{hx}$};
        \node (I) at (12, -10) {$\bt^*//W$};
        \draw[->] (A) edge node[above,rotate=270] {$\Ad^*(h)$}(B);
        \draw [right hook->] (B)--(B1);
        \draw [right hook->] (B1)--(C);
        \draw [right hook->] (D)--(C);
        \draw [right hook->] (F)--(E);
        \draw[->] (C) edge node[above] {$\times \nu$} (E);
        \draw[->] (D) edge node[above] {$\times \nu$} (F);
        \draw[->] (E)--(G);
        \draw[->] (F)--(H);
        \draw[->] (H) edge node[above,rotate=90] {$\sim$} (G);
        \draw[->] (H)--(I);
    \end{tikzpicture}
\end{equation}
where the $\bar\F_q$-varieties are identified with the set of $\bar\F_q$-points, and the isomorphism is also an isomorphism of $\bar\F_q$-varieties, as encountered in earlier sections. Let $i_{M,h,\nu,x} : \fg^*(k)_{x,-r}/\fg^*(k)_{x,-r+} \rightarrow \bt^*//W (\bar\F_q)$ be the composition (through the inverse of the isomorphism). Lemmas 23 and 24 in \cite{chendebackertsai} show that the map is independent of the choice of $h$ and and $M$, and hence we can denote it by $i_{\nu,x}$. The depth-$r$ Deligne-Lusztig parameter attached to $(x,\X)$ is defined to be the equivalence class of $(\nu,i_{\nu,x}(\X))$, and denoted by $\iota_x(\X)$. 
\par

Let $\Ad^*(h) \X$ be denoted by $\prescript{h}{}{}\X$, and note that if $hx\in \sA$, $i_{\nu,x}(\X)= i_{\nu,hx}(\prescript{h}{}{}\X)$ and hence $\iota_x(\X) = \iota_{hx}(\prescript{h}{}{}\X)$, since the first step of the construction is not needed and all the other steps are exactly same. Further, since  $i_{\nu,x}$ does not depend on choice of $h$, we can assume without loss of generality that $hx \in \bar\mcC$.\paragraph{}

Let $\mathscr K$ be the splitting field of $G$. It is a tamely ramified Galois extension of $k$ and let $e(\ms K/k)=n$ be its ramification degree.  Let $m_1=gcd(n,m)$, $n=n'm_1$ and $m=m_1m_2$. Let $\ms E_{m_2}$ be a totally tamely ramified extension of $\ms K$ and $\ms K_f$ be the unique unramified extension of $\ms K$ of degree $f$, such that $\ms K_f= \ms K(\zeta_{m_2})$. The number $f$ is the smallest number such that $|\kappa_{\ms K}|^f\equiv 1(\mathrm{ mod }\:m_2)$, since $gcd(m_2,p)=1$ (Check for example \cite[Chapter II, Prop. 7.12]{neukirchant}). Since the construction of $i_{\nu,x}(\X)$ does not depend on the choice of the Galois extension $M$, we can let $M= \ms E_{m_2}\ms K_f$. Note that it is a Galois extension of $\ms K$ and hence $k$, and satisfies the required properties since $e(M/k)= nm_2=n'm$. Also observe that $M^u= \ms E_{m_2}^u=:\ms M$ (say). We have a tower of local fields $k\subset \ms K \subset \ms E_{m_2} \subset M$, each tamely ramified over $k$. Let $\varpi_1\in \mf O_{\ms K}$ be an uniformizer and choose $\gamma_1 \in \mf O_{\ms E_{m_2}}$ such that $\gamma_1^{m_2}=\varpi_1$ and $\ms E_{m_2}=\ms K(\gamma_1)$. Note that $\rv(\gamma_1)=1/nm_2$ and $\gamma_1$ is also an uniformizer of $\mf O_M$.  \par

Let's follow the maps in the construction of $\iota_x(\X)$, using ideas in Section \ref{subsectionstablefunction}. We can assume $x \in \sA$. Since $T$ splits over $\ms K$, we can identify $\T$ with the reductive quotient of $T(\ms M)$. Further, observe that $\T$ is $\kappa_{\ms K}$-split and has a natural $\kappa_{\ms K}$-structure given by $\T(\kappa_{\ms K})= T(\ms K)_0/T(\ms K)_{0+}$. Similar statements holds true for $\bt$ and $\bt^*$. Let $\mf F$ denote the Frobenius element in $\Gal(\ms K^u/\ms K)$, and $\mf F_{\gamma_1}$ be the unique Frobenius element in $\Gal(\ms M/\ms K)$ such that $\mf F_{\gamma_1}(\gamma_1)=\gamma_1$. Then $\mf F_{\gamma_1}$ is the topological generator of $\Gal (\ms M/\ms E_{m_2})$, and similar to the paragraphs just before Lemma \ref{stableGIT} and its proof, we can give a $\kappa_{\ms K}$-structure to $(\bg_{\ms M})^*_{x,r}$ using $\mf F_{\gamma_1}$, with $(\bg_{\ms M})^*_{x,r}(\kappa_{\ms K})= \fg^*(\ms E_{m_2})_{x,-r}/\fg^*(\ms E_{m_2})_{x,-r+}$, since $\kappa_{\ms K}=\kappa_{\ms E_{m_2}}=\kappa_{M}$.\par
Choose $\nu= \gamma_1^{in'}$. If we look at the morphisms in \eqref{DLMoyPrasad}, note that $(\bg_{\ms M})^*_{x,r} \xlongrightarrow{\gamma_1^{in'}} (\bg_{\ms M})^*_{x,0}$ is defined over $\kappa_{\ms K}$ since $\gamma_1 \in \ms E_{m_2}$ and we have a map of $\kappa_{\ms K}$ points 
$$\fg^*(\ms E_{m_2})_{x,-r}/\fg^*(\ms E_{m_2})_{x,-r+}\xlongrightarrow{\gamma_1^{in'}}\fg^*(\ms E_{m_2})_{x,0}/\fg^*(\ms E_{m_2})_{x,0+}=(\bg_{\ms M})^*_{x,0}(\kappa_{\ms K}).$$
Similarly, the map $\bt_r^* \rightarrow\bt^*$ is also defined over $\kappa_{\ms K}$, and we have a map of $\kappa_{\ms K}$-points with $\ft^*(\ms E_{m_2})_{0}/\ft^*(\ms E_{m_2})_{0+}=\bt^*(\kappa_{\ms K})$. The isomorphism $\bt^*//W^{\ms M}_{x} \rightarrow(\bg_{\ms M})^*_{x,0}//(\G_{\ms M})_{x,0}$ is induced by $\bt^* \rightarrow (\bg_{\ms M})^*_{x,0}$, and is defined over $\kappa_{\ms K}$ since $\T$ is $\ms K$-split. Using the above facts, we see that the maps between $\bar \F_q$ varieties and vector spaces in \eqref{DLMoyPrasad} are defined over $\kappa_{\ms K}$, and we have a diagram of $\kappa_{\ms K}$-points. 

\begin{equation}\label{diagFqpointsDLpar}
    \begin{tikzpicture}[baseline=(current bounding box.center)]
        \node (A) at (0,0) {$ \fg^*(k)_{x,-r}/\fg^*(k)_{x,-r+}$};
        \node (B) at (0,-2) {$ \fg^*(\ms K)_{x,-r}/\fg^*(\ms K)_{x,-r+}$};
        \node (B1) at (0,-4) {$ \fg^*(\ms E_{m_2})_{x,-r}/\fg^*(\ms E_{m_2})_{x,-r+}$};
        \node (C) at (0,-6) {$\ft^*(\ms E_{m_2})_{-r}/\ft^*(\ms E_{m_2})_{-r+}$};
        \node (D) at (6,-4) {$(\bg_{\ms M})^*_{x,0}(\kappa_{\ms K})$};
        \node (E) at (6,-6) {$\bt^*(\kappa_{\ms K})$};
        \node (G) at (12,-4) {$(\bg_{\ms M})^*_{x,0}//(\G_{\ms M})_{hx,0}(\kappa_{\ms K})$};
        \node (H) at (12,-6) {$\bt^*//W^{M}_{x}(\kappa_{\ms K})$};
        \node (I) at (12, -8) {$\bt^*//W(\kappa_{\ms K})$};
        \draw [right hook->] (A)--(B);
        \draw [right hook->] (B)--(B1);
        \draw [right hook->] (C)--(B1);
        \draw [right hook->] (E)--(D);
        \draw[->] (B1) edge node[above] {$\times \nu$} (D);
        \draw[->] (C) edge node[above] {$\times \nu$} (E);
        \draw[->] (D)--(G);
        \draw[->] (E)--(H);
        \draw[->] (H) edge node[above,rotate=90] {$\sim$} (G);
        \draw[->] (H)--(I);
    \end{tikzpicture}
\end{equation}

From the commutative diagram \eqref{diagFqpointsDLpar}, we can conclude that $i_{\nu, x}(\X)$ lies in $\bt^*//W(\kappa_{\ms K})$, where $\ms K$ is the splitting field of $G$. Similar to Remark \ref{remarkFrobstructure}, we have a natural $\kappa_{\ms K}$-structure on $\bt^*$ and $\bt^*//W$, while the $\kappa_{\ms K}$-structure on $(\bg_{\ms M})^*_{x,0}$ depends on choice of $\gamma_1$. However, if we choose some other $\gamma_1'\neq \gamma_1$, we will have $(\nu=\gamma_1^{in'}, i_{\nu,x}(\X)) \sim (\nu'=(\gamma'_1)^{in'}, i_{\nu',x}(\X))$. \paragraph{}

Now, if we return to the setting where $G$ is $k$-split, we have $k= \ms K$ and $\kappa_{\ms K}=\F_q$. So, $i_{\nu,x}(\X)\in \bt^*//W(\F_q)$, similar to the case of $\theta_{\nu}(\pi)$ for a smooth irreducible representation $\pi$. In fact, as shown in Proposition \ref{PropMP=DLpi}, if $(x,\X)$ is a Moy-Prasad type of depth $r$ contained in $\pi\in \Irr(G)_r$, then $\Theta_r(\pi)=\iota_x(\X)$. Further, since $r =i/m$, we can replace $x$ with $\sigma \in [\bar \mcC_m]$ such that $x \in \sigma$ (WLOG, we can assume $x \in \bar \mcC)$, and we observe that $i_{\nu,x}(\X)= t_\sigma\circ \chi_{\sigma,r}(\X)$. 

\begin{proposition}\label{PropMP=DLpi}
    Let $(\pi, V)$ be a smooth irreducible representation of depth $r \in \Z_{(p)}\cap \Q_{>0}$, and let $(x, \X)$ be a Moy-Prasad type contained in $(\pi,V)$. Then $\iota_x(\X)=\Theta_r(\pi)$. 
\end{proposition}

\begin{proof}

Let $\iota_x(\X)= \overline{(\nu', i_{\nu',x}(\X)}$ and $\Theta_r(\pi)=\overline{(\nu, \tilde\theta_\nu(\pi))}$.  Without loss of generality, we can assume that $x \in \Bar \mcC$ and $\nu'=\nu$. Let $r = \frac{i}{m}$ with $p \nmid m$. Then, we can pick $\sigma \in [\Bar{\mcC}_m]$ such that $x \in \sigma $ and $G_{x,r+}=G_{\sigma,r+}$. Pick $v \in V^{G_{\sigma,r+}}$ such that $\pi^{G_{\sigma,r+}}(X)v = \Tilde{\psi}(\X(X))v $ for $ X\in \bg_{\sigma,r}^F$, where $\pi^{G_{\sigma,r+}}$ is as defined in Theorem \ref{thmtmodWtocenter}. For any $\lambda \in \C[(\bt^*//W)^F]$, we have 
 \begin{equation*}
     \lambda(\tilde\theta_\nu(\pi))v = \xi^r(\lambda)v = \mu(G_{\sigma,r+})\sum_{X\in \bg_{\sigma,r}^F}i_{\sigma,r}(\lambda)(X)\pi^{G_{\sigma,r+}}(X)v= \mu(G_{\sigma,r+})\sum_{X\in \bg_{\sigma,r}^F}i_{\sigma,r}(\lambda)(X)\Tilde{\psi}(\X(X))v 
 \end{equation*}
Using Proposition \ref{fourierprop} $(i)$, we see that 
   \begin{equation*}
       \mu(G_{\sigma,r+})\sum_{X\in \bg_{\sigma,r}^F}i_{\sigma,r}(\lambda)(X)\Tilde{\psi}(\X(X))= \mu(G_{\sigma,r+})\sum_{X^*\in (\bg_{\sigma,r}^F)^*}\mcF_{\bg_{\sigma,r}^F}(i_{\sigma,r}(\lambda))(X^*)\mcF_{\bg_{\sigma,r}^F}(\Tilde{\psi}\circ(-\X))(X^*)
   \end{equation*}
From a simple calculation, it follows that 
   \begin{equation*}
       \mcF_{\bg_{\sigma,r}^F}(\Tilde{\psi}\circ(-\X))= \left|\bg_{\sigma,r}^F\right|^{1/2}\mathbbm{1}_{\X}
   \end{equation*}
   and using definition of $i_{\sigma,r}$, we have 
   \begin{equation*}
       \mcF_{\bg_{\sigma,r}^F}(i_{\sigma,r}(\lambda))= (c^{\mu}_{\sigma,r})^{-1}\mcF_{\bg_{\sigma,r}^F}(\rho_{\sigma,r}\circ t^*_\sigma(\lambda))=  (c^{\mu}_{\sigma,r})^{-1}\chi_{\sigma,r}^*\circ t^*_\sigma(\lambda)
   \end{equation*}
   Using the above results, we note that for any $\lambda \in \C[(\bt^*//W)^F]$
   \begin{align*}
       \lambda(\tilde\theta_\nu(\pi))&=\mu(G_{\sigma,r+})\sum_{X^*\in (\bg_{\sigma,r}^F)^*}\mcF_{\bg_{\sigma,r}^F}(i_{\sigma,r}(\lambda))(X^*)\mcF_{\bg_{\sigma,r}^F}(\Tilde{\psi}\circ(-\X))(X^*) \\
       &=\mu(G_{\sigma,r+})\cdot \left|\bg_{\sigma,r}^F\right|^{1/2}\cdot (c^{\mu}_{\sigma,r})^{-1}\sum_{X^*\in (\bg_{\sigma,r}^F)^*}\chi_{\sigma,r}^*\circ t^*_\sigma(\lambda)(X^*)\mathbbm{1}_{\X}(X^*)\\
       &=\chi_{\sigma,r}^*\circ t^*_\sigma(\lambda)(\X)= \lambda(i_{\nu,x}(\X))
   \end{align*}
   and hence $\Theta_r(\pi)= \iota_x(\X)$. 
\end{proof}

\begin{remark}\label{remark nontrivial and alternativeproof}
   Two Moy-Prasad types are defined to be stable associates if they have the same Deligne-Lusztig parameter attached to it. In \cite[Lemma 36]{chendebackertsai}, it is proved that two Moy-Prasad types which are associates of each other are stable associates. Since Moy-Prasad types contained in a smooth irreducible representation are associates of each other, this essentially attaches a Deligne-Lusztig parameter $\iota_\pi$ to $\pi$. The previous proposition gives an alternative proof of the fact that any two Moy-Prasad types contained in a smooth irreducible representation have the same Deligne-Lusztig parameter attached to it, and it is shown to be equal to the parameter $\Theta_r(\pi)$ attached to $\pi\in \Irr(G)_r$ in Section \ref{section:positiveDLparameters}. Further, \cite[Lemma 33]{chendebackertsai} shows that if the depth $r\in \Z_{(p)} \cap \Q_{>0}$, $\iota_x(\X)$ is non-zero if and only if $(x,\X)$ is non-degenerate. Since every smooth irreducible representation contains a non-degenerate Moy-Prasad type, we have $\iota_\pi=\Theta_r(\pi)$ and $\Theta_r(\pi) $ is non-trivial for any $\pi \in \Irr(G)_r$ with $ r \in \Z_{(p)}\cap \Q_{>0}$. 
   Let $\bar 0$ denote the image of $0$ in $(\bt^*//W)^F$ and $\DL_r^\circ$ denote the subset of $\DL_r$ containing non-trivial depth-$r$ Deligne-Lusztig parameters, i.e., $\overline{(\nu_r,\theta_{\nu_r})}$ such that $\theta_{\nu_r}\in (\bt^*//W)^F \setminus \{\bar 0\}$. Then, $\Theta_r$ is basically a map $\Theta_r: \Irr(G)_r\rightarrow \DL_r^\circ$ for $ r \in \Z_{(p)}\cap \Q_{>0}$.
\end{remark}

\subsection{Restricted depth-$r$ parameters}
 Let $G^\vee$ denote the dual connected reductive group over $\C$, and $T^\vee$ the dual complex torus. Let $W_k$ denote the absolute Weil group of $k$ ( see \cite[Section 6.3]{ME09}), and $I_k=\Gal(\bar k/K)$ denote the Inertia subgroup of the Weil group. We have the usual upper numbering filtration of the Galois group $\Gamma_k=\Gal(\bar k/k)$ denoted by $\Gamma^r_k$ (see \cite [Section 3 of Chapter IV] {serrelocal}). For infinite extensions, $\Gamma_k^r$ is defined as the inverse limit $\varprojlim \Gal(L/k)$ for finite Galois extensions of $k$. Similarly, we have the upper numbering filtrations of the Weil group, denoted $\{I_k^r\;|\: r\geq 0\}$, where $I_k^0=I_k$ is the inertia subgroup and $I_k^{0+}$ is the wild inertia subgroup. Let $I_k^{r+}$ denote the closure of $\cup_{s>r}I_k^s$. The group $I^r_k$ has the subspace topology from $I_k$, and we equip $I_k^r/I_k^{r+}$ with the quotient topology. 
 \begin{definition}\label{defn:positiveRPr}
     Let $r \in \Q_{>0}$. A continuous homomorphism $\varphi:I_k^r/I_k^{r+} \rightarrow G^\vee $, is a tame restricted depth-$r$ parameter if there exists a maximal torus $T^\vee \subset G^\vee $ and a continuous homomorphism $\tilde \varphi :I_k^{0+} \rightarrow T^\vee $, trivial on $I_k^{r+}$ such that $\tilde \varphi |_{I_k^r} \equiv \varphi$. Let $\mathrm{RP}_r$ denote the set of $G^\vee$-conjugacy classes of tame restricted depth-$r$ parameters.
 \end{definition}
Let $\pi\in \Irr(G)_r$ such that $\Theta_r(\pi)= \overline{(\nu,\tilde\theta_\nu(\pi))} \in \mathrm{DL}_r$ for $r \in \Z_{(p)}\cap \Q_{>0}$. Let $M$ denote a finite tamely ramified Galois extension of $k$ such that $r\cdot e(M/k) \in \Z$. Given the fixed additive character $\tilde \psi :\kappa_k=\F_q\rightarrow \C^{\times}$, we have an additive character $\tilde \psi_M : \kappa_M\rightarrow \C^{\times}$ given by $\tilde \psi_M (x) = \tilde \psi (\mathrm {Tr}_{\kappa_M/\F_q}(x))$. The map $\mathrm{p}: \bt^* \rightarrow \bt^*//W$ is surjective at the level of $\bar\F_q$-points. For $\tilde\theta_\nu(\pi)\in \bt^*//W(\F_q)$, there exists a finite extension $\mbk$ over $k$ and $X \in \bt^*(\mbk)$ such that $X$ maps to $\tilde\theta_\nu(\pi)$ under the map $\mathrm p$. Thus, possibly replacing $M$ with a finite unramified extension of itself, we have $X \in \ft^*(M)_{-r}/ \ft^*(M)_{-r+} \cong  \bt^*(\kappa_M)$ such that $X$ maps to $\tilde\theta_\nu(\pi)$ under the map 
 \[
\ft^*(M)_{-r}/ \ft^*(M)_{-r+} \xrightarrow{\simeq}\bt^*(\kappa_M) \rightarrow(\bt^*//W)(\kappa_M) 
 \]
Given such an $X$, we can define an additive character on $\ft(M)_r/\ft(M)_{r+}$ in the following way-
\begin{align*}
   \ft(M)_r/\ft(M)_{r+}& \longrightarrow\C^{\times} \\
   Y\:\:&\longmapsto \Tilde{\psi}_M(X(Y))
\end{align*}
Using the Moy-Prasad isomorphism $T(M)_r/T(M)_{r+} \cong \ft(M)_r/\ft(M)_{r+}$, we can pull it back to a character of $T(M)_r/T(M)_{r+}$ and hence a character of $T(M)_r$ denoted by $\chi_{X,M}$ via pullback along the quotient. Let $\psi_1,\psi_2$ are two characters of $T(M)$ whose restriction to $T(M)_r$ gives $\chi_{X,M}$, and denote their associated Langlands parameters via Local Langlands for Tori (see \cite[Section7.5]{Yutori}) by $\varphi_i:W_M\rightarrow T^\vee$, $i=1,2$ and they clearly have depth $\leq r$. Since $\psi_1^{-1}\psi_2$ has depth less than $r$, by depth preservation for tamely ramified tori (see \cite [Section 7.10]{Yutori}) $\varphi_1^{-1}\varphi_2$ will also have depth less than $r$, and hence $\varphi_1$ and $\varphi_2$ have the same restriction to $I_M^r/I_M^{r+}$. Further, since $r>0$ and $M/k$ is tamely ramified, we have that $I_M^r= I_M^r\cap \Gal(\bar k/M)^r=I^r_M\cap \Gal(\bar k/ M^t)= I^r_M \cap \Gal(\bar k/k^t)=I_k^r$. Hence, to $\overline{(\nu,\tilde\theta_\nu(\pi))}$, we can attach a continuous homomorphism $\varphi^T_{X,M}:I_k^r/I_k^{r+} \rightarrow T^\vee$. We have an embedding $T^\vee \hookrightarrow G^\vee$ determined upto $G^\vee$ conjugation, and composing with that, $\varphi_{X,M}$ gives a tame restricted depth-$r$ parameter.\paragraph{}

Observe that the process of attaching restricted depth-$r$ parameter $\varphi^T_{X,M}$ to $\Theta_r(\pi)$ is exactly the same as in \cite{chendebackertsai}, combining the steps mentioned in Section 3.2 and the proof of Lemma 40. Thus, $\varphi^T_{X,M}$ is the same as $\phi^T_{X,M}$ as defined in \cite[Section 5.1]{chendebackertsai} and using \cite[Corollary 19]{chendebackertsai}, we can conclude that $\varphi^T_{X,M}$ is independent of the choice of $M$. Hence, we can denote it by $\varphi^T_X$. Further, the $G^\vee$-conjugacy class of $\varphi_X^T$ depends only on $\Theta_r(\pi)$ (check \cite[Lemma 41]{chendebackertsai}, and hence gives an unique element $\varphi_{\iota_\pi}\in \mathrm{RP}_r$. There is in fact a bijection between $\mathrm{DL}_r$ and $\mathrm{RP}_r$, as shown in \cite[Lemma 45]{chendebackertsai}. Thus for $r \in \Q_{>0}\cap \Z_{(p)}$, we have a map 
\begin{equation*}
    \Irr(G)_r \xrightarrow{\Theta_r} \DL_r\xrightarrow{\simeq} \mathrm{RP}_r
\end{equation*}
and each element in the image of this map is non-trivial when $\Irr(G)_r \neq \emptyset$. \par

 \section{The depth zero case}\label{section: Depthzerocase}
In this section, we will construct stable functions on the depth zero Moy-Prasad filtraion quotient, and use that to construct elements in the depth zero center. We also define and attach depth-zero Deligne Lusztig parameters to smooth irreducible representations of depth zero using the constructed elements. This is basically a generalization of the results in \cite{cb24} to the reductive group case, and we use the complex dual torus instead of the $\bar\F_q$-dual.  Many of the proofs here can be directly lifted from \cite{cb24}, and we will refer to it as and when it seems fit. 
\subsection{Stable functions on the depth-zero quotient}
For $\sigma \in [\sA]$, let $C(\G_\sigma^F)$ denote the space of class functions (conjugation invariant functions) on the depth-zero quotient $\G_\sigma^F=G_{\sigma,0}/G_{\sigma,0+}$, equipped with the convolution product 
\begin{equation*}
    f*g(x)= \sum_{y \in \G_\sigma^F}f(xy^{-1})g(y)
\end{equation*}
For $f,g \in C(\G_\sigma^F)$, we have the standard inner product on $ C(\G_\sigma^F)$ given by 
\begin{equation*}
    (f,g)= |\G_\sigma^F|^{-1}\sum_{x\in \G_\sigma^F}f(x)\overline{g(x)}.
\end{equation*}
Note that in the depth-zero case, $\G_\sigma \cong \mbL_\sigma$ and for $\sigma \preceq\tau \in [\sA]$, $\mbP_{\sigma,\tau} \subset \G_\sigma$ as defined in Section \ref{subsectionstablefunction} is an $F$-stable parabolic subgroup with Levi decomposition $\mbP_{\sigma,\tau} \cong \U_{\sigma,\tau}\rtimes \G_\tau$, where $\G_\tau\subset \G_\sigma$ is an $F$-stable Levi subgroup. We have the parabolic restriction map $\res_{\G_\tau}^{\G_\sigma}: C(\G_\sigma^F)\rightarrow C(\G_\tau^F)$
\begin{equation*}
    \res_{\G_\tau}^{\G_\sigma}(f)(l)= \sum_{u \in \U_{\sigma,\tau}^F}f(lu)
\end{equation*}
We fix an isomorphism 
\begin{equation}\label{isom:F_qto Qmod Z}
    \bar\F_q^\times \simeq (\Q/\Z)_{p'}
\end{equation}
where $(\Q/\Z)_{p'}$ denotes the subgroup of elements in $\Q/\Z$ of order prime to $p$. Let $G_\sigma^\vee$ denote the complex dual of $\G_\sigma$ and $T^\vee = X^*(\T)\otimes \C^\times \cong X^*(T)\otimes \C^\times $ denote the complex dual of $\T$. Identifying $X^*(\T)$ and $X^*(T)$, it can be thought of as the complex dual of $T$ as well. Let $[q]$ denote the morphism $x \mapsto x ^q$ on $G_\sigma^\vee $. The set of semisimple conjugacy classes in $G_\sigma^\vee $ stable under $[q]$ are in bijection with $(T^\vee //W_\sigma)^{[q]}$. We have a surjective map 
\begin{equation*}
    \mc L_\sigma: \Irr(\G_\sigma^F)\rightarrow (T^\vee //W_\sigma)^{[q]}
\end{equation*}
where $\Irr(\G_\sigma^F) $ is the set of isomorphism classes of irreducible complex representations of the finite group $\G_\sigma^F$. This is essentially the map described in \cite[Section 5]{DL76}, but using the complex dual as described in \cite[Section 16]{lusztigtwelve} instead of the $\bar\F_q$-dual as in the original work. This depends on the chosen isomorphism \ref{isom:F_qto Qmod Z}. As mentioned in \cite{lusztigtwelve}, using the complex dual is more canonical since it involves choosing only one ismorphism as opposed to two as in \cite{DL76}. This decomposes $\Irr(\G_\sigma^F)$ into packets given by $\mc L_\sigma^{-1}(\theta)$ for $\theta \in (T^\vee //W_\sigma)^{[q]}$.
For each $f\in C(\G^F_\sigma)$ and $(\pi,V) \in \Irr(\G_\sigma^F)$, Schur's lemma implies that we have a function $\Irr(\G_\sigma^F) \rightarrow \C$ such that 
\begin{equation}\label{gammafndefn}
    \sum_{g \in \G_\sigma^F}f(g)\pi(g)= \gamma_f(\pi)\mathrm{Id}_V
\end{equation}
This gives a bijection between $C(\G^F_\sigma)$ and gamma functions $\gamma:\Irr(\G_\sigma^F) \rightarrow \C$ with inverse given by $ \gamma \mapsto f_\gamma$ such that
\begin{equation}\label{gammafntof}
    f_\gamma(x) :=|\G_\sigma^F|^{-1}\sum_{\pi \in \Irr(\G_\sigma^F)}\gamma(\pi)\chi_\pi(1)\overline{\chi_\pi(x)},
\end{equation}
where $\chi_\pi$ denotes the character of $\pi$. 

\begin{definition}
    A function $f\in C(\G^F_\sigma)$ is defined to be stable if $\gamma_f:\Irr(\G_\sigma^F) \rightarrow \C$ factors through $\Irr(\G_\sigma^F) \rightarrow (T^\vee //W_\sigma)^{[q]} \rightarrow\C$, i.e., $\gamma_f$ is constant on the packets $\mc L_\sigma^{-1}(\theta)$ and hence can be viewed as functions on the set $(T^\vee //W_\sigma)^{[q]}$. We denote the space of stable functions by $C^{st}(\G^F_\sigma)$. 
\end{definition}

Let $\C[(T^\vee //W_\sigma)^{[q]}]$ denote the space of complex functions on the set $(T^\vee //W_\sigma)^{[q]}$, with multiplictaion given by the ususal pointwise multiplication of functions. For $\sigma\preceq \tau \in [\sA]$, $W_\tau \subset W_\sigma$ and we have a canonical map $T^\vee//W_\tau \rightarrow T^\vee //W_\sigma$ compatible with $[q]$. We denote the map given by pullback along the natural map  $(T^\vee //W_\tau)^{[q]} \rightarrow (T^\vee //W_\sigma)^{[q]}$ by 
\begin{equation*}
    \Res^\sigma_\tau : \C[(T^\vee //W_\sigma)^{[q]}] \rightarrow \C[(T^\vee //W_\tau)^{[q]}]
\end{equation*}
\begin{proposition}[Properties of stable functions]
    \begin{itemize}
        \item [(1)] For $\sigma \in [\sA]$, there is an algebra isomorphism 
        \begin{equation*}
            \rho_{\sigma,0}: \C[(T^\vee //W_\sigma)^{[q]}]\xlongrightarrow[]{\simeq} C^{st}(\G^F_\sigma)
        \end{equation*}
        which sends the characteristic function $\mathbbm 1_\theta $ of $\theta \in \C[(T^\vee //W_\sigma)^{[q]}] $ to the idempotent projector $f_\theta \in C^{st}(\G^F_\sigma)$ for the packet $\mc L_\sigma^{-1}(\theta)$ with $\gamma_{f_\theta}(\pi)=1$ if $\mc L_\sigma(\pi)=1$, and $\gamma_{f_\theta}(\pi)=0$ otherwise. 
        \item[(2)] For $\sigma \preceq \tau \in [\sA]$ and $f  \in C^{st}(\G^F_\sigma)$, we have $\res^{\G_\sigma}_{\G_\tau} \in  \in C^{st}(\G^F_\tau)$, and the following diagram commutes. 
        \begin{equation}\label{restrictionstabledepth0}
            \begin{tikzcd}
             \C[(T^\vee //W_\sigma)^{[q]}] \arrow{r}{\rho_{\sigma,0}} \arrow{d}[swap]{\Res^{\sigma}_{\tau}} & C^{st}(\G^F_\sigma) \arrow{d}{\res^{\G_\sigma}_{\G_\tau}} \\
             \C[(T^\vee //W_\tau)^{[q]}]  \arrow{r}{\rho_{\tau,0}} & C^{st}(\G^F_\tau)  
            \end{tikzcd}
        \end{equation}
        \item[(3)] For any $f \in C^{st}(\G^F_\sigma)$ and $\sigma \preceq \tau \in [\sA]$, we have 
        \begin{equation*}
            \sum_{u \in \U_{\sigma,\tau}^F}f(xu)= 0 \text{   for   } x \not \in \mbP_{\sigma,\tau}^F
        \end{equation*}
    \end{itemize}
\end{proposition}
\begin{proof}These are basically restatments of the results in \cite[Section 4.1]{cb24} in the current setting. 
    \begin{itemize}
        \item [(1)] The proof follows from the definition of stable functions, \eqref{gammafntof} and the fact that $\gamma_{f*g}(\pi)=\gamma_f(\pi)\cdot \gamma_g(\pi)$. 
        \item[(2)] This is proved in \cite[Proposition 4.2.2(4)]{laumon}. 
        \item[(3)] This is \cite[Theorem 4.2]{cb24}. 
    \end{itemize}
\end{proof}

\begin{proposition}
    Let $\sigma \in [\sA]$ and $n \in N_G(T)(k)$. The adjoint action of $n$ induces an isomorphism 
    \begin{equation*}
        C^{st}(\G^F_\sigma) \xlongrightarrow{\simeq} C^{st}(\G^F_{n\sigma})
    \end{equation*}
\end{proposition}
\begin{proof}
    The adjoint action of $n$ induces an isomorphism $\T \xrightarrow{\Ad(n)}\T$, and hence an action on $X^*(\T)$. Thus we get an action of $n$ on $T^\vee=X^*(\T)\otimes \C^\times$ induced by the adjoint action, which gives an isomorphism $ T^\vee \rightarrow T^\vee$ which is equivariant with respect $W_\sigma$ action (with $W_\sigma$ acting on the r.h.s via the isomorphism $W_\sigma \xrightarrow{\Ad(n)} W_{n\sigma}$. This gives an isomorphism $T^\vee // W_\sigma \xrightarrow[\Ad(n)]{\simeq}T^\vee//W_{n\sigma}$ compatible with $[q]$. Further, the adjoint action of $n$ induces an isomorphism $\Irr(\G_\sigma^F) \xrightarrow[\Ad(n)]{\simeq} \Irr(\G_{n\sigma}^F)$, and hence the following diagram with the natural maps commutes. 
    \begin{equation*}
        \begin{tikzcd}
            \Irr(\G_\sigma^F)\arrow{r}{\mc L_\sigma} \arrow{d}{\wr}&(T^\vee // W_\sigma)^{[q]}\arrow{d}{\wr}\\
            \Irr(\G_{n\sigma}^F) \arrow{r}{\mc L_{n\sigma}}& (T^\vee // W_{n\sigma})^{[q]}
        \end{tikzcd}
    \end{equation*}
    Then, the definition of stable functions and part $(1)$ of the previous proposition give us the following commutative diagram
     \begin{equation}\label{adjnstabledepth0}
        \begin{tikzcd}
            \C[(T^\vee // W_\sigma)^{[q]}]\arrow[d,"\wr", "(\Ad(n^{-1}))^*"'] \arrow{r}{\rho_{\sigma,0}} & C^{st}(\G^F_\sigma)\arrow{d}{\Ad(n^{-1})^*} \\
            \C[(T^\vee // W_{n\sigma})^{[q]}]\arrow{r}{\rho_{n\sigma,0}}& C^{st}(\G^F_{n\sigma})
        \end{tikzcd}
    \end{equation}
    and hence the isomorphism $C^{st}(\G^F_\sigma) \xrightarrow{\simeq} C^{st}(\G^F_{n\sigma})$.
\end{proof}
\subsection{From stable functions to depth-zero Bernstein center}
Let $\sigma \in [\bar \mcC]$. Using the embedding $W_\sigma \hookrightarrow W$ as in Section \ref{section:stablefnstocenter}, we have a map 
\begin{equation*}
    t_{\sigma,0}:(T^\vee//W_\sigma)^{[q]}\longrightarrow (T^\vee//W)^{[q]}.
\end{equation*}
We have a natural inclusion map 
\begin{equation*}
    \iota_{\sigma,0}: C(\G_\sigma^F)\rightarrow C_c^\infty\left(\frac{G_{\sigma,0}/G_{\sigma,0+}}{G_{\sigma,0}}\right)\rightarrow C_c^\infty\left(\frac{G(k)/G_{\sigma,0+}}{G_{\sigma,0}}\right)=\M^0_\sigma
\end{equation*}
sending $ C(\G_\sigma^F) \ni f: \G_\sigma^F \rightarrow\C$ to $\iota_{\sigma,r}( f) \in \M^0_\sigma$ supported in $G_{\sigma,0} \subset G(k)$ given by 
\begin{equation*}
    \iota_{\sigma,0}( f) :G_{\sigma,0}\rightarrow G_{\sigma,0}/G_{\sigma,0+} \xrightarrow{\simeq} \G_\sigma^F \xlongrightarrow{f}\C
\end{equation*}
Note that $\tilde\iota _{\sigma,0}:=\mu(G_{\sigma,0+})^{-1}\iota_{\sigma,0}$ is an algebra morphism sending the identity $\mathbbm 1_e \in  C(\G_\sigma^F)$ to $\delta_{G_{\sigma,0+}}\in \M_\sigma^0$. Consider the composed map 
\begin{equation*}
   j_{\sigma,0}: \C[(T^\vee//W)^{[q]}] \xrightarrow{t^*_{\sigma,0}}\C[(T^\vee//W_\sigma)^{[q]}]\xrightarrow{\rho_{\sigma,0}} C^{st}(\G_\sigma^F)\xlongrightarrow{\mu(G_{\sigma,0+})^{-1}\iota_{\sigma,0}} \M_\sigma^0
\end{equation*}
We will follow steps similar to Section \ref{section:stablefnstocenter} to construct a map $\C[(T^\vee//W)^{[q]}] \rightarrow \ZZ^0(G)$.

\begin{lemma}\label{lemTmodWtostable}
    For $\sigma \in [\bar\mcC]$, we have a map 
    \[
    i_{\sigma,0}:\C[(T^\vee//W)^{[q]}] \longrightarrow C^{st}(\G_\sigma^F)
    \]
    such that for $\sigma \preceq\tau \in [\bar\mcC_m]$, we have the following commutative diagram. 
    \begin{equation}\label{TmodWtostable}
        \begin{tikzcd}
             \C[(T^\vee//W)^{[q]}] \arrow{r}{i_{\sigma,0}}\arrow{rd}{i_{\tau,0}}& C^{st}(\G_\sigma^F)\arrow{d}{\res^{\G_{\sigma}}_{\G_{\tau}}}\\
             &C^{st}(\G_\tau^F)
        \end{tikzcd}
    \end{equation}
\end{lemma}
\begin{proof}
    For $\sigma \in [\bar\mcC]$, define $i_{\sigma,0} :\C[(T^\vee//W)^{[q]}] \longrightarrow C^{st}(\G_\sigma^F)$ as the composition 
\begin{equation*}
    \C[(T^\vee//W)^{[q]}] \xlongrightarrow{t^*_{\sigma,0}} \C[(T^\vee//W_\sigma)^{[q]}]\xlongrightarrow{\rho_{\sigma,0}} C^{st}(\G_\sigma^F)
\end{equation*}
For $\sigma \preceq \tau \in [\bar \mcC_m]$, it immediately follows that the following diagram commutes.
\begin{equation*}
    \begin{tikzcd}
      \C[(T^\vee//W)^{[q]}] \arrow{r}{t^*_{\sigma,0}} \arrow{rd}[swap]{t^*_{\tau,0}}  & \C[(T^\vee//W_\sigma)^{[q]}]\arrow{d}{\Res^\sigma_\tau}\\
     & \C[(T^\vee//W_\tau)^{[q]}]
    \end{tikzcd}
\end{equation*}
Combining the above diagram and \eqref{restrictionstabledepth0} proves the lemma. 
\end{proof}
\begin{lemma}
    For $\sigma \preceq \tau \in [\bar\mcC]$, the map $\tilde\iota_{\sigma,0}=\mu(G_{\sigma,0+})^{-1}\iota_{\sigma,0}:C^{st}(\G_\sigma^F) \rightarrow \M^0_\sigma$ fits into the following commutative diagram of algebra morphisms. 
    \begin{equation}\label{stabletoM0sigma}
        \begin{tikzcd}
          C^{st}(\G_\sigma^F) \arrow{r}{\tilde\iota_{\sigma,0}}  \arrow{d}[swap]{\res^{\G_{\sigma}}_{\G_{\tau}}}&\M^0_\sigma\arrow{d}{\phi^0_{\sigma,\tau}}\\
          C^{st}(\G_\tau^F) \arrow{r}{\tilde\iota_{\sigma,0}}&\M^0_\tau
        \end{tikzcd}
    \end{equation}
\end{lemma}
\begin{proof}
    This is just a rescaled version of \cite[Lemma 5.1]{cb24}. 
\end{proof}

\begin{lemma}\label{lemmaadjointTmodWtoM0sigma}
    Let $\sigma_1,\sigma_2\in [\Bar{\mcC}]$ and $n \in N_G(T)(k)$ such that $n\mcC=\mcC$ and $ n \sigma_1 \preceq \sigma_2$. Then, we have a commutative diagram
    \begin{equation*}
    \begin{tikzcd}
         \C[(T^\vee//W)^{[q]}]\arrow{r}{j_{\sigma_1,0}} \arrow{rd}{j_{\sigma_2,0}} & \M^0_{\sigma_1}\arrow{d}{\phi^0_{\sigma_1,\sigma_2,n}}\\
        & \M^0_{\sigma_2}
    \end{tikzcd}
\end{equation*}
\end{lemma}
\begin{proof}
     The proof uses the same ideas as the proof of Lemma \ref{lemmaadjointtmodWtoMrsigma}. Using \eqref{adjnstabledepth0} , we have 
    \begin{equation*}
         \begin{tikzcd}
       \C[(T^\vee//W)^{[q]} \arrow{r}{t^*_{\sigma_1,0}} \arrow[dr,"t^*_{n\sigma_1,0}"'] & \C[(T^\vee//W_{\sigma_1})^{[q]}]\arrow{d}{\wr} \arrow{r}{\rho_{\sigma_1,0}} & C^{st}(\G_{\sigma_1}^F) \arrow{d}{\wr} \\
        &\C[(T^\vee//W_{n\sigma_1})^{[q]}]\arrow{r}{\rho_{n\sigma_1,0}} & C^{st}(\G_{n\sigma_1}^F)
    \end{tikzcd}
    \end{equation*}
Then, applying  Lemma \ref{lemTmodWtostable} to $n\sigma_1\preceq \sigma_2$ and combining with the above diagram, we have the following.
\begin{equation*}
    \begin{tikzcd}
    &  C^{st}(\G_{\sigma_1}^F) \arrow{d}{\wr} \\
     \C[(T^\vee//W)^{[q]} \arrow{ur}{i_{\sigma_1,0}} \arrow[r,"i_{n\sigma_1,0}" ]  \arrow[rd, "i_{\sigma_2,0}" ] & C^{st}(\G_{n\sigma_1}^F) \arrow{d}{\res^{\G_{n\sigma_1}}_{\G_{\sigma_2}}} \\
     & C^{st}(\G_{\sigma_2}^F)
     \end{tikzcd}
\end{equation*}
Further, the adjoint action of $n$ gives us an isomorphism $\Ad(n^{-1})^*:\M^0_{\sigma_1} \xrightarrow{\simeq}\M^0_{n\sigma_1}$ and hence a commutative diagram 
   \begin{equation*}
        \begin{tikzcd}
          C^{st}(\G_{\sigma_1}^F) \arrow{r}{\tilde\iota_{\sigma_1,0}} \arrow{d}{\wr} &\M^0_{\sigma_1}\arrow{d}{\wr} \\
          C^{st}(\G_{n\sigma_1}^F) \arrow{r}{\tilde\iota_{n\sigma_1,0}}  \arrow{d}[swap]{\res^{\G_{n\sigma_1}}_{\G_{\sigma_2}}}&\M^0_{n\sigma_1}\arrow{d}{\phi^0_{n\sigma_1,\sigma_2}}\\
          C^{st}(\G_{\sigma_2}^F) \arrow{r}{\tilde\iota_{\sigma_2,0}}&\M^0_{\sigma_2}
        \end{tikzcd}
    \end{equation*}
    Since $j_{\sigma,r}= \tilde\iota_{\sigma,r}\circ i_{\sigma,r} $, combining the above two diagrams proves the lemma. 
\end{proof}

\begin{theorem}\label{thmTmodWto0center}
    There is an algebra homomorphism 
    \begin{equation*}
        \xi^0: \C[(T^\vee//W)^{[q]}\longrightarrow \ZZ^0(G)
    \end{equation*}
    such that for any depth-zero representation $(\pi,V)$ and $0\neq v \in V^{G_{\sigma,0+}} $, $\sigma \in [\bar\mcC]$, we have 
    \begin{equation}\label{depth0xi}
         \xi^0(\lambda)(v) = \sum_{x\in \G_{\sigma}^F}i_{\sigma,0}(\lambda)(x)\pi^{G_{\sigma,0+}}(x)v
    \end{equation}
    where $i_{\sigma,0}: \C[(T^\vee//W)^{[q]} \rightarrow C^{st}(\G_{\sigma}^F)$ is the map defined in Lemma \ref{lemTmodWtostable} and $\pi^{G_{\sigma,0+}}$ denotes the natural representation of $\G_{\sigma}^F$ on $V^{G_{\sigma,0+}}$. 
\end{theorem}
\begin{proof}
    Using the previous lemmas, we see that we have a map from $\C[(T^\vee//W)^{[q]}$ to $\M^0_\sigma$ for each $\sigma \in [\bar \mcC]$ compatible with the inverse system maps $\phi^0_{\sigma,\sigma'}$. Hence, we have a map $\C[(T^\vee//W)^{[q]}\rightarrow \lim_{\sigma\in [\mcC]}\M^0_\sigma\cong \ZZ^0(G)$ and we define this to be $\xi^0$. Then, calculations similar to Theorem \ref{thmtmodWtocenter} give \eqref{depth0xi} and finishes the proof. 
\end{proof}
\subsection{Depth-zero Deligne-Lusztig parameters}\label{section:depth0DLparameters}
Let $S$ be a $k$-split maximal torus of $G$. For another $k$-split maximal torus $S'$, there exists $g \in G(k)$ such that $S'=\Ad_g(S)$, and the pulllback along the map $\Ad_g: S \rightarrow S'$ gives rise to a canonical isomorphism $(S')^\vee= X^*(S')\otimes \C^\times \rightarrow X^*(S)\otimes\C^\vee = S^\vee$ equivariant with respect to $[q]$. Let $W_S=W(G,S)$ denote the Weyl group of $G$ with respect to $S$. 
\begin{definition}
 We define the set of depth-zero Deligne-Lusztig parameters of $G(k)$ to be 
 \begin{equation*}
     \DL_0:= \lim_S (S^\vee//W_S)^{[q]}
 \end{equation*}
 where the limit is over all $k$-split maximal tori. 
\end{definition}

\begin{remark}
    Our definition is equivalent to the definition of $\mathrm{DL}_0$ in \cite[Section 7.1]{chendebackertsai}, when restricted to our setting. This is because we have assumed $G$ to be $k$-split. Further, since we had fixed a $k$-split maximal tori $T$, without loss of generality we can represent a depth zero Deligne-Lusztig parameter by an element in $(T^\vee//W)^{[q]}$, which we will do in some cases henceforth. So, we can equivalently define the set $(T^\vee//W)^{[q]}$ to be the set of depth-zero Deligne-Lusztig parameters, as done in \cite{cb24}.
\end{remark}
Let $(\pi, V)$ be a smooth irreducible complex representation of depth zero. Then, $z \in \ZZ^0(G)$ acts on $\pi$ via a scalar in $\C$, and composing $\xi^0$ with the evaluation map $\ZZ^0(G)\rightarrow \mathrm{End}(\pi)=\C$, we have a map 
\begin{equation*}
    \C[(T^\vee//W)^{[q]}] \xrightarrow{\xi^0}\ZZ^0(G)\rightarrow \mathrm{End}(\pi)=\C
\end{equation*}
and hence an element of $\C[(T^\vee//W)^{[q]}]$, denoted $\Theta_0(\pi)$ given by $\xi^0(\lambda)v=\lambda(\Theta_0(\pi))v$ for $v \in V$ and $\lambda \in \C[(T^\vee//W)^{[q]}]$. Thus, to each smooth irreducible representation of depth zero, we can attach a depth-zero Deligne-Lusztig parameter $\Theta_0(\pi)\in \mathrm{DL}_0$ which gives a map $\Theta_0:\Irr(G)_0\rightarrow \DL_0$. \par
Combining all the maps $\Theta_r$ for $r \in \Z_{(p)}\cap \Q_{\geq 0}$, we see that we have a map 
\begin{equation*}
    \Theta:\coprod_{r \in \Z_{(p)}\cap \Q_{\geq 0}} \Irr(G)_r \longrightarrow \coprod_{r \in \Z_{(p)}\cap \Q_{\geq 0}}\mathrm{DL}_r
\end{equation*}
with $\Theta(\pi) := \Theta_{\rho(\pi)}(\pi)\in \mathrm{DL}_{\rho(\pi)}$ for $\rho(\pi) \in \Z_{(p)}\cap \Q_{\geq 0}$, where $\rho(\pi)$ denotes the depth of $\pi$. Further, for $\rho(\pi) \in \Z_{(p)}\cap \Q_{> 0}$, $\Theta(\pi)\in \DL_{\rho(\pi)}^\circ$. Thus, we actually have a map 
\begin{equation}\label{DLparametersalltame}
     \Theta:\coprod_{r \in \Z_{(p)}\cap \Q_{\geq 0}} \Irr(G)_r \longrightarrow \DL_0\coprod \left(\coprod_{r \in \Z_{(p)}\cap \Q_{> 0}}\mathrm{DL}_r^\circ\right)
\end{equation}
\subsection{Deligne-Lusztig prameters attached to depth-zero Moy-Prasad types}
\begin{definition}
    A depth-zero Moy-Prasad type of $G(k)$ is a pair $(x,\mc X)$ where $x \in \B(G,k)$ and $\mc X$ is an irreducible cuspidal representation of $\G_x^F\cong G_{x,0}/G_{x,0+}$ inflated to the parahoric subgroup $G_{x,0}$. Let $\mathrm{MP}(0)$ denote the set of Moy-Prasad types of depth zero. 
\end{definition}
Two Moy-Prasad types of depth zero $(x,\mc X)$ and $(y,\mc Y)$ are said to be associates if there exists $g \in G(k)$ such that $G_{x,0}\cap G_{gy,0}$ surjects onto both $\G_x$ and $\G_{gy}$ and $\mc X$ is isomorphic to $\Ad(g)\mc Y$. A smooth irreducible representation of depth zero is said to contain a Moy-Prasad type $(x,\mc X)\in \MP(0)$ if the restriction $\res_{G_{x,0}}(\pi) $ of $\pi $
to the parahoric $G_{x,0}$ contains $\mc X$. From \cite[Theorem 5.2]{MP94}, we know that any smooth irreducible depth-zero representation $\pi$ contains a depth-zero Moy-Prasad type, and any two Moy-Prasad types contained $\pi$ are associates of each other. \paragraph{}

We recall briefly how to attach Deligne-Lusztig parameters to depth-zero Moy-Prasad types, following \cite[Section 5.4]{cb24} and \cite[Lemma 59]{chendebackertsai}. Let $(x,\mc X)$ be a depth-zero Moy-Prasad type, and let $S$ be a $k$-split maximal torus such that $x \in \A_S$. Then, $\mc X$ is the inflation of a irreducible cuspidal representation of $\G_x^F$. Using methods in \cite[Section 16]{lusztigtwelve} and identifying $X^*(S)$ and $X^*(\mathbb S)$, we can attach to $\mc X$ and element $\theta_{S,\mc X} \in (S^\vee//W_{x,S})^{[q]}$ where all the notations denote the ususal objects as defined earlier in this article and in \cite{chendebackertsai}. Let $\theta_{x,\mc X}$ denote the image of $\theta_{S,\mc X}$ along the map $(S^\vee//W_{x,S})^{[q]}\rightarrow (S^\vee//W_{S})^{[q]}\rightarrow \DL_0$ and $(x,\mc X) \mapsto \theta_{x,\mc X}$ gives us the desired map $\MP(0) \rightarrow \DL_0$. This is independent of the choice of apartment containing $x$, since all such apartments are $G_{x,0}$-conjugates. Equivalently, pick $g \in G(k)$ such that $gx \in \sA$. Then the image of $\theta_{T,\prescript{g}{}{}\mc X}\in (T^\vee//W_{gx})^{[q]}$ along  $(T^\vee//W_{gx})^{[q]} \rightarrow (T^\vee//W)^{[q]}$ gives the representative of the corresponding Deligne-Lusztig parameter in $(T^\vee//W)^{[q]}$ in accordance with the definition in \cite{cb24}. The fact that this is independent of the choice of $g$ can be shown similarly to \cite[Lemma 23]{chendebackertsai}. So, without loss of generality we can denote it by $\theta_{x,\mc X}$ as well, and we note that $\theta_{x,\mc X}=\theta _{gx,\prescript{g}{}{}\mc X}$. 
\begin{proposition}
    Let $(\pi,V)$ be a smooth irreducible depth-zero representation and let $(x,\mc X)$ be a depth-zero Moy-Prasad type contained in $(\pi,V)$. Then, $\theta_{x,\mc X}=\Theta_0(\pi)$. 
\end{proposition}
\begin{proof}
    Without loss of generality, we can assume $x \in \bar \mcC$. There exists $\sigma \in [\bar \mcC]$ with $x\in \sigma$ such that $G_{x,0+}=G_{\sigma,0+}$. Pick $v \in V^{G_{\sigma,0+}}$ such that $\pi^{G_{\sigma,0+}}(g)v= \mc X(g)v$ for $g \in \G_\sigma^F$. For $\lambda \in \C[(T^\vee//W)^{[q]}]$, we have 
    \[
    \lambda (\Theta_0 (\pi))v= \xi^0(\lambda)v = \sum_{g\in \G_{\sigma}^F}i_{\sigma,0}(\lambda)(g)\pi^{G_{\sigma,0+}}(g)v=\sum_{g\in \G_{\sigma}^F}i_{\sigma,0}(\lambda)(g)\mc X(g)v
    \]
    Using \eqref{gammafndefn} and the definition of $i_{\sigma,0}$, we have 
    \[
    \sum_{g\in \G_{\sigma}^F}i_{\sigma,0}(\lambda)(g)\mc X(g)v=\sum_{g\in \G_{\sigma}^F}\rho_{\sigma,o}\circ t^*_{\sigma,0}(\lambda)(g)\mc X(g)v= t^*_{\sigma,0}(\lambda)(\mc L_{\sigma}(\mc X))v= \lambda(\theta_{x,\mc X})v
    \]
    where the last step follows from the definition of $\theta_{x,\mc X}$ in the previous paragraph. Thus $ \lambda (\Theta_0 (\pi))=\lambda(\theta_{x,\mc X})$ for all $\lambda \in \C[(T^\vee//W)^{[q]}]$, and hence $\Theta_0(\pi)=\theta_{x,\mc X}$.
\end{proof}
\begin{remark}
    This gives an alternative proof of the fact that any two depth-zero Moy-Prasad types contained a smooth irreducible representation of depth zero have the same Deligne-Lusztig parameter attached to it. 
\end{remark}
\subsection{Restricted depth-zero parameters}
Let $\mf F$ denote a (geometric) Frobenius in $W_k$, and we denote its image in $W_k/I^{0+}_k$ by $\mf F$ as well. 
\begin{definition}\label{defn:depth0RPr}
    A depth-zero Langlands parameter is a continuous cocycle 
    $\phi:W_k/I^{0+}_k \rightarrow G^\vee$
    such that $\phi(\mf F)$ is semi-simple. A restricted depth-zero parameter is a continuous cocycle 
    $\varphi: I_k/I^{0+}_k\rightarrow G^\vee $
    which is a restriction from a depth-zero Langlands parameter. We denote by $\RP_0$ the set of $G^\vee$ conjugacy classes of restricted depth-zero parameters. 
\end{definition}
If we fix a Borel subgroup $T\subset B$, we get a based root datum of $G$ $\psi_0(G)=(X^*(T),\Delta,X_*(T),\Delta^\vee) $ where $\Delta$ is the set of positive simple roots determined by $B$. We fix a pinning $(G^\vee,B^\vee, T^\vee, e^\vee=\{x_\alpha\}_{\alpha\in \Delta}$, and the action of $W_k$ on $\psi_0(G)$ induces a $W_k$-action on $G^\vee$ denoted by $\mu_G: W_k \rightarrow \mathrm{Aut} ((G^\vee,B^\vee, T^\vee, e^\vee=\{x_\alpha\}_{\alpha\in \Delta}) \cong \mathrm{Aut}(\psi_0(G))$. If $G$ splits over a tamely ramified extension, the action $\mu_G$ of factors through $W_k/I^{0+}_k$ and hence we did not have to consider cocycles for restricted postive depth parameters. The set $\RP_0$ is independent of the choice of pinning since all pinnings are $G^\vee$-conjugate. Since we have assumed $G$ to be split over $k$, we immediately see that the action is trivial in our case and the depth-zero parameters (restricted or otherwise) are just homomprphisms $I_k/I^{0+}_k\rightarrow G^\vee $ and $W_k/I^{0+}_k \rightarrow G^\vee$ with required properties as mentioned in the definition. \par
 As per \cite[Lemma 62]{chendebackertsai}, there is a bijections $\RP_0\cong \DL_0$ and we get a map 
 \begin{equation*}
     \Irr(G)_0\xrightarrow{\Theta_0}\DL_0\xrightarrow{\simeq}\RP_0
 \end{equation*}
Further, the bijection $\RP_0\cong \DL_0$ depends on the choice of an isomorphism $\bar\F_q\simeq(\Q/\Z)_{p'}$, and making the same choice as in \ref{isom:F_qto Qmod Z} makes the composed map $\Irr(G)_0\rightarrow \RP_0$ independent of the choice.

\section{Decomposing the category of smooth representations}\label{section:decomposition}
\begin{definition}
    Let $\{\mathscr C_i\}_{i\in I}$ be a family of full subcategories of $R(G)$. We can decompose $R(G)$ as a product of full subcategories and write 
    \begin{equation*}
        R(G) = \prod_{i\in I} \mathscr C_i
    \end{equation*}
    if every $(\pi, V)$ decomposes as $V = \oplus_{i\in I}V_i$ with $V_i \in \mathscr C_i$ and for any $V_i \in \mathscr C_i$, $V_j \in \mathscr C_j$, we have $\mathrm{Hom}_{G(k)}(V_i,V_j)= 0$ if $i \neq j$.
\end{definition}
For $(\pi,V) \in R(G)$, let $\JH(\pi)$ (or $\JH(V)$) denote the set of (isomorphism classes of ) irreducible subquotients of $\pi$ (also called Jordan-H\"older factors of $\pi$). Let $\Irr(G) = \coprod_{\alpha \in A} S_\alpha$ be a partition of the set of (isomorphism classes of) smooth irreducible representations of $G(k)$. For $S_\alpha \subset \Irr(G)$, let $R(G)_{S_\alpha} $ denote the full subcategory of $R(G)$ defined as  $R(G)_{S_\alpha} := \{ (\pi,V)\in R(G)\:|\: \JH(\pi) \subseteq S_\alpha\}$. The subcategory $R(G)_{S_\alpha}$ is closed under the formation of subquotients, extensions and direct sums. For $(\pi, V)\in R(G)$, let $V_{S_\alpha}$ denote the sum of all $G(k)$-invariant subspaces of $V$ which lie in $R(G)_{S_\alpha}$. It is the unique maximal $G(k)$-subspace of $V$ in $R(G)_{S_\alpha}$. We immediately observe that $V_{S_\alpha} \cap V_{S_\alpha'}=\{0\}$ and $\mathrm{Hom}_{G(k)}(V_{S_\alpha},V_{S_\alpha'})=\{0\}$. For detailed proofs, check \cite[Section 16]{heyer2023smooth}.

\begin{definition}
   We say that $\{S_\alpha\}_{\alpha\in A}$ splits $(\pi,V)\in R(G)$ if $V$ can be written as a direct sum 
   \begin{equation*}
       V = \bigoplus_{\alpha \in A} V_{S_\alpha}
   \end{equation*}
   We say $\{S_\alpha\}_{\alpha\in A}$ splits $R(G)$ if it splits every $(\pi,V) \in R(G)$, i.e., if $R(G)$ decomposes into a product of full subcategories 
   \begin{equation*}
      R(G)= \prod_{\alpha \in A} R(G)_{S_\alpha}.
   \end{equation*}
\end{definition}

As mentioned in \cite[Remark 30]{chendebackertsai}, if $p$ does not divide the order of the absolute Weyl group of $G$, then \cite[Theorem 6.1]{fintzen2021types} implies that every non-degenerate positive depth Moy-Prasad type has depth $r \in \Z_{(p)}$. We henceforth impose that condition on the characteristic of the residue field, and hence for any $\pi \in \Irr(G)$, $\rho(\pi)\in \Z_{(p)}\cap \Q_{\geq 0}$. \par
For $p \nmid |W|$, \eqref{DLparametersalltame}, gives a map 
\begin{equation}\label{DL^tdefinition}
  \Theta:\Irr(G)=\coprod_{r \in \Z_{(p)}\cap \Q_{\geq 0}} \Irr(G)_r \longrightarrow \DL_0\coprod \left(\coprod_{r \in \Z_{(p)}\cap \Q_{> 0}}\mathrm{DL}_r^\circ\right)  =:\DL^t
\end{equation}
and hence we can partition the set of smooth irreducible representations of $G(k)$ into packets using the Deligne-Lusztig parameters. Let $\vartheta^r \in \DL_r$ for $ r \in \Z_{(p)}\cap \Q_{\geq 0}$ and $\Pi(\vartheta^r)= \{\pi \in \Irr(G)_r\:|\: \Theta_r(\pi)= \vartheta^r\}$. We have a partition  
\begin{equation}\label{DLpackets}
    \Irr(G) =\coprod_{\vartheta^0\in \DL_0}\Pi(\vartheta^0)\coprod \left( \coprod_{r \in \Z_{(p)}\cap \Q_{\geq 0} } \coprod_{\vartheta^r\in \DL^\circ_r} \Pi(\vartheta^r)\right)= \coprod_{\vartheta \in \DL^t}\Pi(\vartheta)
\end{equation}
For $\vartheta \in \DL^t$, let $R(G)_\vartheta$ denote the full subcategory of $R(G)$ such that $R(G)_\vartheta=\{ (\pi,V)\in R(G)\:|\: \JH(\pi)\subseteq \Pi(\vartheta)\}$ (instead of $R(G)_{\Pi(\vartheta)}$, to simplify notation). The aim of this section is to prove the following theorem.
\begin{theorem}\label{thm:decompositionDLpackets}
    We have a decomposition of $R(G)$ as a product of full subcategories 
    \begin{equation*}
        R(G) = \prod_{\vartheta\in \DL^t}R(G)_\vartheta
    \end{equation*}
\end{theorem}
In concrete terms, we want to show that for each $(\pi, V)\in R(G)$, we have $G(k)$-invariant subspaces $V_\vartheta \in R(G)_\vartheta$ such that $V = \oplus_{\vartheta \in \DL^t}V_\vartheta$ (i.e., the partition in \eqref{DLpackets} splits $R(G)$). We show this by producing projectors to $R(G)_\vartheta$ in the Bernstein center for each $\vartheta \in \DL^t$, and proving a finiteness condition. \paragraph{}

Let $X^*(G)= \mathrm{Hom}_k(G,\G_m)$ denote the group of $k$-rational characters of $G$. For $\chi \in X^*(G)$, considering it as a group homomorphism $G(k) \rightarrow k^\times$, $\rv(\chi(g))\in \Z$ for $g \in G(k)$. Let $G(k)^0$ denote the subgroup 
\begin{equation*}
     G(k)^0= \{ g \in G(k)\:|\: \rv(\chi(g))=0 \:\forall \: \chi \in X^*(G)\}
\end{equation*}
The subgroup $ G(k)^0$ is open, normal and contains all compact subgroups of $G(k)$. A character $\chi :G(k) \rightarrow \C^\times$ is called unramified if $G(k)^0$ lies in its kernel. We denote the group of unramified characters of $G(k)$ by $X^{ur}(G)$ and it can be identified with $\Hom(G(k)/G(k)^0, \C^\times)$. 
\begin{definition}
    We define a cuspidal pair (or cuspidal datum) of $G$ to be a pair $(L, \varrho)$ where $L$ is a $k$-Levi subgroup of $G$ and $\varrho $ is an irreducible supercuspidal representation of $L(k)$. 
\end{definition}
Two such pairs $(L, \varrho)$ and $(L',\varrho')$ are called associated if there exists $g\in G(k)$ such that $gLg^{-1}=L'$ and the map $\Ad(g) : L(k) \xrightarrow{\simeq} L'(k)$ induces an isomorphism $\prescript{g}{}{}\varrho \cong \varrho'$. We denote the $G(k)$-conjugacy class of $(L,\varrho)$ by $(L,\varrho)_G$, and let $\Omega(G)$ denote the set of equivalence classes of cuspidal pairs modulo association. From \cite[Proposition 1.7.2.1]{roche2009bernstein}, we see that for each $\pi \in \Irr(G)$, there exists a unique $(L,\varrho)_G\in \Omega(G)$ such that $\pi$ is isomorphic to a subquotient of $i^G_P(\varrho)$, where $P$ is a $k$-parabolic subgroup of $G$ with Levi component $L$ and $i^G_P$ denotes the normalized parabolic induction functor $i^G_P: R(L) \rightarrow R(G)$. Thus, the assignment $\pi \mapsto (L,\varrho)_G$ gives us a well defined surjective map $\CS: \Irr(G) \rightarrow \Omega(G)$, and $\CS(\pi)= (L,\varrho)_G$ is called the cuspidal support of $\pi$. \paragraph{}
We define another equivalence relation called inertial equivalence on the set of cuspidal data of $G$ in the following way: $(L,\varrho) \sim (L',\varrho')$ if there exists $g\in G(k)$ and $\omega \in X^{ur}(L') $ such that $\prescript{g}{}{}L=L'$, $\prescript{g}{}{}\varrho \cong \varrho' \otimes \omega$. We denote the inertial equivalence class of $(L,\varrho)$ by $ [L,\varrho]_G$. This is a coarser relation than association, and we denote the set of equivalence classes of cuspidal pairs modulo inertial equivalence by $\mf B(G)$. We denote by $\IS$ the composition 
\begin{equation*}
    \IS: \Irr(G) \xrightarrow{\CS} \Omega(G) \xtwoheadrightarrow{\Upsilon} \mf B(G)
\end{equation*}
and $\IS(\pi)= [L,\varrho]_G$ is called the inertial support of $\pi$, where $\pi $ is an irreducible subquotient of $i^G_{P'}(\varrho\otimes \omega)$ for some $k$-parabolic $P'$ with Levi component $L$ and $\omega \in X^{ur}(L)$. In fact, it is enough to take $P'=P$ (check \cite[Corollary 1.10.4.3]{roche2009bernstein}). \par

For $\mf a \in \mf B(G)$, let $\Omega(G)_{\mf a}= \Upsilon^{-1}(\mf a)$ and we have $\Omega(G) = \coprod_{\mf a \in \mf B(G)} \Omega(G)_{\mf a}$. The set $\Omega(G)$ has a natural structure of a complex algebraic variety, with connected components given by $\Omega(G)_{\mf a}$ (check \cite[Section 3.3]{haines}). We have an isomorphism $\ZZ(G) \xrightarrow{\simeq} \C[\Omega(G)]$ between the ring of regular functions on this variety and $\ZZ(G)$. 
\par
The set of inertial equivalence classes gives a partition of the $\Irr(G)$ given by $$\Irr(G)= \coprod_{\mf a \in \mf B(G)} \IS^{-1}(\mf a)$$ and we define $R(G)_{\mf a}:=\{ (\pi, V) \in R(G)\:|\: \JH(\pi) \subseteq \IS^{-1}(\mf a) \}$. This is an indecomposable full subcategory of $R(G)$, and we have the Bernstein decomposition theorem 
\begin{theorem}[Theorem 1.7.3.1 in \cite{roche2009bernstein}]\label{Thmbernsteindecomposition}
    We have a decomposition of $R(G)$ as a product of indecomposable full subcategories 
    \begin{equation*}
        R(G) = \prod_{\mf a \in \mf B(G)}R(G)_{\mf a}
    \end{equation*}
\end{theorem}

We fix a choice of $\nu_m \in k^t$ such that $\nu_m$ lies in a finite tamely ramified extension of $k$ and $\rv(\nu_{1/m})= 1 /m $ for $m \in \Z_{>0}$ and $m \nmid p$. For $r= i/m \in \Z_{(p)}\cap \Q_{>0}$, define $\nu_r= (\nu_{1/m})^i$. Further, we fix an isomorphism $\bar\F_q^\times \simeq (\Q/\Z)_{p'}$. With these choices and the choice of a fixed $k$-split maximal torus $T$ of $G$, we have fixed the maps $\xi^r$ as defined in Theorem \ref{thmTmodWto0center} (for $r=0)$ and Theorem \ref{thmtmodWtocenter} (for $r>0$) for each $r \in \Z_{(p)}\cap \Q_{\geq 0}$. Using the maps $\xi^r$, we will construct projectors to $R(G)_\vartheta$ for $\vartheta \in \DL^t$, and use the Bernstein decomposition theorem to give a finiteness condition, thereby proving Theorem \ref{thm:decompositionDLpackets}. \par

Since we have fixed $\nu_r$ for each $r \in \Z_{(p)}\cap \Q_{>0}$, we can represent $\vartheta^r \in \DL_r$ by the pair $(\nu_r, \theta_{\nu_r})$, $\theta_{\nu_r}\in (\bt^*//W)^F$ in the equivalence class of $\vartheta^r$. Similarly, for any $\pi \in \Irr(G)_r$, $\Theta(\pi)$ can be represented by the pair $(\nu_r, \tilde\theta_{\nu_r}(\pi))$, as defined in Section \ref{section:positiveDLparameters}. With the fixed maximal torus $T$, $\DL_0$ can be identified with the set $(T^\vee//W)^{[q]}$. 

\begin{proposition}\label{prop:projectorpackets}
    For $r \in \Z_{(p)}\cap \Q_{>0}$, let $\vartheta^r=\overline{(\nu_r, \theta_{\nu_r})} \in \DL_r^\circ$ and $\mbm 1_{\theta_{\nu_r}} \in \C[(\bt^*//W)^F]$ denote the characteristic function of $\theta_{\nu_r} \in (\bt^*//W)^F\setminus\{\bar 0\}$. Then, $e_{\vartheta^r}= \xi^r( \mbm 1_{\theta_{\nu_r}})$  is an idempotent in $\ZZ^r(G)$ and acts as the projector to $R(G)_{\vartheta^r}$, i.e., for $(\pi,V)\in \Irr(G)$, 
    \begin{equation*}
        e_{\vartheta^r}|_V=\begin{cases}
        \mathrm{Id}_V, & \text{ if } \rho(\pi)= r \text{ and } \Theta(\pi)=\vartheta^r\\
        0, & \text{ otherwise}
    \end{cases}
    \end{equation*}
    For $r=0$,  let $\vartheta^0 \in  (T^\vee//W)^{[q]}\equiv \DL_0 $ and let $\mbm 1_{\vartheta^0} \in \C[(T^\vee//W)^{[q]}]$ similarly denote the characteristic function of $\vartheta^0 \in (T^\vee//W)^{[q]}$ . Then, $e_{\vartheta^0}=\xi^r(\mbm 1_{\vartheta^0}) $ is an idempotent in $\ZZ^0(G)$ and acts as the projector to $R(G)_{\vartheta^0}$. 
\end{proposition}
\begin{proof}
    Let us first consider the case of $r \in \Z_{(p)}\cap \Q_{>0}$, and let $r \in \frac{1}{m}\Z$. Note that we only need to consider $\vartheta^r\in \DL^\circ_r$ because $\Theta(\pi)\in \DL^\circ_r$ for $\pi \in \Irr(G)_r$, as stated in Remark \ref{remark nontrivial and alternativeproof}.\par 
    Let $(\pi, V) \in \Irr(G)_{<r}$. Then, there exists $\sigma \in [\bar\mcC_m]$ such that $V^{G_{\sigma,r}}\ni v  \neq 0$. We have 
    \begin{equation*}
        e_{\vartheta^r}(v)= \xi^r(\mbm 1_{\theta_{\nu_r}})v= \mbm 1_{\theta_{\nu_r}}(\bar 0)v
    \end{equation*}
    where the last step follows from Theorem \ref{tmodWtostable}. Since $\theta_{\nu_r}\neq \bar 0$, we have $e_{\vartheta^r}(v)=0$ and hence $e_{\vartheta^r}|_V=0$ for $(\pi,V) \in \Irr(G)_{<r}$. Since the image of $\xi^r$ lies in $\ZZ^r(G)$, it immediately follows that $e_{\vartheta^r}|_V=0$ for $(\pi,V) \in \Irr(G)_{>r}$.  \par
    For $(\pi,V) \in \Irr(G)_r$, let $\sigma \in [\bar\mcC_m]$ be such that $\exists \:0\neq v \in V^{G_{\sigma,r+}}$. If $(\pi,V)\in \Pi(\vartheta^r) \subset \Irr(G)_r$, then $\Theta(\pi)= \overline{(\nu_r,\tilde \theta_{\nu_r}(\pi))}$ with $\tilde \theta_{\nu_r}(\pi)=\theta_{\nu_r}$ and 
    \begin{equation*}
         e_{\vartheta^r}(v)= \xi^r(\mbm 1_{\theta_{\nu_r}})v= \mbm 1_{\theta_{\nu_r}}(\tilde \theta_{\nu_r}(\pi))v=v
    \end{equation*}
    which shows that $e_{\vartheta^r}|_V= \Id_V$ for $(\pi,V)\in \Pi(\vartheta^r)$. If $(\pi,V)\in \Irr(G)_r \setminus\Pi(\vartheta^r)$, then $\tilde \theta_{\nu_r}(\pi)\neq \theta_{\nu_r}$ and we have 
     \begin{equation*}
         e_{\vartheta^r}(v)= \xi^r(\mbm 1_{\theta_{\nu_r}})v= \mbm 1_{\theta_{\nu_r}}(\tilde \theta_{\nu_r}(\pi))v=0. 
    \end{equation*}
    The fact that $e_{\vartheta^r}$ is an idempotent in $\ZZ^r(G)$ immediately follows since $\xi^r : \C[(\bt^*//W)^F] \rightarrow \ZZ^r(G)$ is an algebra morphism. This finishes the case of $r>0$. \par
    For $(\pi,V)\in \Irr(G)_0$, let $\sigma \in [\bar\mcC]$ be such that $\exists \:0\neq v \in V^{G_{\sigma,0+}}$. Let $\vartheta^0 \in \DL_0$. If $\pi \in \Pi(\vartheta^0)$, then $\Theta(\pi)=\vartheta^0$ and we have 
    \begin{equation*}
       e_{\vartheta^0}(v) = \xi^0(\mbm 1_{\vartheta^0})v= \mbm 1_{\vartheta^0}(\Theta(\pi))v= v
    \end{equation*}
    which shows $e_{\vartheta^0}|_V= \Id_V$ for $(\pi,V)\in \Pi(\vartheta^0)$. If $(\pi,V)\in \Irr(G)_0 \setminus\Pi(\vartheta^0)$, then $\Theta(\pi)\neq \vartheta^0$ and 
    \begin{equation*}
         e_{\vartheta^0}(v) = \xi^0(\mbm 1_{\vartheta^0})v= \mbm 1_{\vartheta^0}(\Theta(\pi))v= 0
    \end{equation*}
    The element $e_{\vartheta^0}(v) \in \ZZ^0(G)$ is an idempotent since $\xi^0$ is an algebra map. Thus, for $(\pi,V)\in \Irr(G)$ and $\vartheta \in \DL^t$
    \begin{equation*}
        e_{\vartheta}|_V=\begin{cases}
        \mathrm{Id}_V, & \text{ if }  \Theta(\pi)=\vartheta\\
        0, & \text{ otherwise}
    \end{cases}
    \end{equation*}
    which finishes the proof. 
\end{proof}
\begin{proposition}\label{prop:depth<r projector}
      For $r \in \Z_{(p)}\cap \Q_{>0}$, let $\bar 0^r=\overline{(\nu_r, \bar0)} \in \DL_r$ be the trivial Deligne-Lusztig parameter at depth-$r$ and $\mbm 1_{\bar 0} \in \C[(\bt^*//W)^F]$ denote the characteristic function of $\bar0 \in (\bt^*//W)^F$. Then, $e_{\bar 0^r}= \xi^r( \mbm 1_{\bar 0})$  is an idempotent in $\ZZ^r(G)$ and acts as the projector to $R(G)_{<r}$, i.e., for $(\pi,V)\in \Irr(G)$, 
    \begin{equation*}
        e_{\bar 0^r}|_V=\begin{cases}
        \mathrm{Id}_V, & \text{ if } \rho(\pi)< r \\
        0, & \text{ otherwise}
    \end{cases}
    \end{equation*}
\end{proposition}
\begin{proof}
     Let $(\pi, V) \in \Irr(G)_{<r}$ for $r \in \frac{1}{m}\Z$. Then, there exists $\sigma \in [\bar\mcC_m]$ such that $V^{G_{\sigma,r}}\ni v  \neq 0$. We have 
    \begin{equation*}
        e_{\bar 0^r}(v)= \xi^r(\mbm 1_{\bar 0})v= \mbm 1_{\bar 0}(\bar 0)v=v
    \end{equation*}
    where the last step follows from Theorem \ref{tmodWtostable}. Since the image of $\xi^r$ lies in $\ZZ^r(G)$, it immediately follows that $e_{\vartheta^r}|_V=0$ for $(\pi,V) \in \Irr(G)_{>r}$.  \par
    For $(\pi,V) \in \Irr(G)_r$, let $\sigma \in [\bar\mcC_m]$ be such that $\exists \:0\neq v \in V^{G_{\sigma,r+}}$. Then $\Theta(\pi)= \overline{(\nu_r,\tilde \theta_{\nu_r}(\pi))}\neq \bar0^r$ with $\tilde \theta_{\nu_r}(\pi)\neq \bar0$, and 
    \begin{equation*}
         e_{\bar0^r}(v)= \xi^r(\mbm 1_{\bar0})v= \mbm 1_{\bar0}(\tilde \theta_{\nu_r}(\pi))v=0
    \end{equation*}
    which shows that $e_{\bar0^r}|_V= 0 $ for $(\pi,V)\in \Irr(G)$.
\end{proof}

Now, we prove the main theorem of the section. 
\begin{proof}[Proof of Theorem \ref{thm:decompositionDLpackets}]
    Let $(\pi, V)\in R(G)$ and $\vartheta \in \DL^t$. Since $e_{\vartheta}\in \ZZ(G)$ is an idempotent, it is immediate that it projects onto a  $G(k)$-invariant subspace of $V$, and we define it to be $V_\vartheta := e_{\vartheta}(V)$. From Proposition \ref{prop:projectorpackets}, we see that $V_\vartheta\in R(G)_\vartheta$ and it is the unique maximal $G(k)$-subspace of $V\in R(G)_\vartheta$. It is immediate that $V_\vartheta \cap V_{\vartheta'}=\{0\}$ for $\vartheta\neq \vartheta' \in \DL^t$ and hence $V_\vartheta \oplus V_{\vartheta'} \subseteq V $ , and $\Hom_{G(k)}(V_1,V_1')=0$ if $V_1\in R(G)_\vartheta$, $V_1'\in R(G)_{\vartheta'}$. Note that this is enough to decompose representations of finite length, i.e., we have proved that $\{\Pi(\vartheta )\}_{\vartheta \in \DL^t}$ splits any representation of finite length. To extend this to all smooth representations, we need a finiteness condition. \par
    Using Theorem \ref{Thmbernsteindecomposition}, we know that we can write $V$ as a direct sum 
    \begin{equation*}
        V = \bigoplus_{\mf a \in \mf B(G)} V_{\mf a}
    \end{equation*}
    where $V_{\mf a}$ is the unique maximal $G(k)$-subspace of $V$ in $R(G)_{\mf a}$. Then, the set 
    $$\overline{\IS}(\pi)= \{ \mf a \in \mf B(G)\:|\: V_{\mf a} \neq 0\} $$
    is a finite set. For a $k$-Levi subgroup $L$ of $G$, we denote the depth of $\varrho \in \Irr(L)$ by $\rho_L(\varrho)$.   If $\varsigma\in \Irr(G) $ is an irreducible subquotient of $V_{\mf a}$ for $\mf a = [L,\varrho]_G$, then we know that $\varsigma $ is isomorphic to a subquotient of $i_P^G(\varrho \otimes \omega)$ for some $\omega \in X^{ur}(L)$ and $P$ a $k$-parabolic subgroup of $G$ with Levi component $L$. Since $L(k)^0$ contains all compact subgroups of $L(k)$, $\rho_L(\varrho)= \rho_L(\sigma \otimes \omega)$ for any $\omega \in X^{ur}(L)$. From \cite[Theorem 5.2]{MP96}, we know that depth is preserved under parabolic induction and hence $\rho(\varsigma)= \rho_L(\varrho)$. Since $\overline{\IS}(\pi)$ is finite, the set 
    $$\bar\rho(\pi)=\{ \rho(\varrho)\:|\: \mf a = [L,\varrho]_G \in \overline{\IS}(\pi)\}$$
    is a finite set consisting of elements in $\Z_{(p)}\cap \Q_{\geq 0}$. If $\vartheta^r \in \DL_r$ for $r \not \in \bar\rho(\pi)$, we can immediately observe that $V_{\vartheta^r} = e_{\vartheta^r}(V)=0$, and hence we only need to consider finite number of depths. Further note that $\DL_r$ is finite for $r \in \Z_{(p)}\cap \Q_{\geq 0}$. This gives the necessary finiteness condition and we see that 
    \begin{equation*}
        V = \bigoplus_{r \in \bar\rho(\pi)}\bigoplus_{\vartheta^r \in \DL_r} V_{\vartheta^r}
    \end{equation*}
    Since $(\pi, V)\in R(G)$ was arbitrary, we have that any $(\pi, V)\in R(G)$ can be written as a direct sum 
    \begin{equation*}
        V= \bigoplus_{\vartheta\in \DL^t}V_{\vartheta}
    \end{equation*}
    which finshes the proof of the theorem. 
    \end{proof}

Consider the idempotent elements $e_r :=\sum_{\vartheta^r \in \DL_r^\circ}e_{\vartheta^r}$ for $r\in \Z_{(p)}\cap \Q_{>0}$ and $e_0=\sum_{\vartheta^0 \in \DL_0}e_{\vartheta^0}=[A_{\delta_0}]$. Then, for $(\pi,V)\in R(G)$, $e_r(V)\subset V$ is a subrepresentation which has irreducible subquotients of only depth $r$. Further, the projector $e_{\bar0^r} $ constructed in Proposition \ref{prop:depth<r projector} projects onto the depth$<r$ part. Let $R(G)_r=\{(\pi,V)\in R(G)\:|\: \JH(\pi)\subset \Irr(G)_r\}$ and $R(G)_{<r}$ be defined similarly.  Then we have a decomposition $R(G)_{\leq r}=R(G)_r\oplus R(G)_{<r}$ via the projectors that we defined. Let us define $R(G)_{\bar 0^r}:=R(G)_{<r}$ consistent with the notation $R(G)_{\vartheta^r}$ for $\vartheta^r\in \DL_r^0$ with the corresponding projector $e_{\vartheta^r}$. Further, using the bijection $\DL_r \rightarrow \RP_r$, we can parametrize these subcategories by $\varphi^r \in \RP_r$. Let $\tilde\Theta_r$ denotes the composition $\Irr(G)_r\xrightarrow{\Theta}\DL_r\xrightarrow{\simeq}\RP_r$ for $r \in \Z_{(p)}\cap\Q_{\geq 0}$, $\RP_r^\circ$ denote the non-trivial restricted Langlands parameters for positive depth and $\RP^t$ be defined similarly to $\DL^t$ in \eqref{DL^tdefinition}. Then, we have a map $\tilde \Theta:\Irr(G) \rightarrow \RP^t$ and we can equivalently denote $R(G)_{\vartheta},\: \vartheta \in \DL^t$ by $R(G)_{\varphi},\:\varphi\in \RP^t$ if $\vartheta \mapsto \varphi$ via the bijection $\DL_r \xrightarrow{\simeq}\RP_r$ described in \cite{chendebackertsai}. Then we have the following decompositions as a corollary of Theorem \ref{thm:decompositionDLpackets} and the previous propositions:
\begin{corollary}\label{corollary:decomprestricted and <r}
 We have a decomposition of $R(G)$ as a product of full subcategories 
    \begin{equation*}
        R(G) = \prod_{\varphi\in \RP^t}R(G)_\varphi
    \end{equation*}
    Further, we have a decomposition for $r \in \Z_{(p)}\cap \Q_{>0}$
    \begin{equation*}
        R(G)_r=\bigoplus_{\varphi^r\in \RP_r^\circ}R(G)_{\varphi^r}=\bigoplus_{\vartheta_r\in \DL_r^\circ}R(G)_{\vartheta^r}
    \end{equation*}
    and the following decomposition for $r \in \Z_{(p)}\cap \Q_{\geq0}$
    \begin{equation*}
        R(G)_{\leq r}=\bigoplus_{\varphi^r\in \RP_r}R(G)_{\varphi^r}=\bigoplus_{\vartheta_r\in \DL_r}R(G)_{\vartheta^r}
    \end{equation*}
    where the subcategory corresponding to the trivial depth-$r$ parameter for positive depths contains the representations with all their irreducible subquotients in $\Irr(G)_{<r}$.
\end{corollary}

  \section{Some conjectures and comments}\label{section:conjectures}  
Most of the statements in this section are conjectural. Based on the results that we have proved and some existing conjectures about Langlands parameters and stable center, we make some predictions about the elements in $\ZZ^r(G)$ that we constructed as the images of the maps in Theorems \ref{thmtmodWtocenter} and \ref{thmTmodWto0center}. \par

We assume the characteristic of the residue field of $k$ is sufficiently large, atleast $p \nmid |W|$. Let $\phi: WD_k=W_k\ltimes\C\rightarrow \prescript{L}{}{}G$ be a Langlands parameter for $G(k)$. We have a notion of depth of a Langlands parameter, which we denote by $\rho(\phi)$. 
\begin{equation*}
    \rho(\phi)=\text{min}\{ r \in \Q_{\geq 0}\:|\: \phi|_{I_k^{r+}}\text{ is trivial }\}
\end{equation*}
Let $\Pi(\phi)$ denote the $L$-packet corresponding to $\phi$ and $\pi \in \Pi(\phi)$. The relation between $\rho(\pi)$ and $\rho(\phi)$ has been studied quite a bit, and under some assumptions (like the $G$ being split over tamely ramified extension and some conditions on the residue field), it has been conjectured that depth is preserved, i.e., $\rho(\pi)=\rho(\phi)$ for any $\pi \in \Pi(\phi)$ (check \cite[Conjecture 52]{chendebackertsai}). Depth-preservation for local Langlands under certain conditions is known in several cases like the tamely ramified tori (\cite[Section 7.10]{Yutori}), unitary groups (\cite{oi2023depthquasispliclassical},\cite{Oi2021Depthnon-quasi}), $\mathrm{GSp}_4$ (\cite{ganapathy2015GSP4}) and inner forms of $\GL_n$ (\cite{aubert2016depth}). \par

Assuming conjectural depth preservation basically implies that the restriction $\phi|_{W_k}$ of a Langlands parameter $\phi_\pi$ of $\pi \in \Irr(G)$ to $W_k$  factors through $\phi_\pi|_{W_k}: W_k/I_k^{\rho(\pi)+}\rightarrow G^\vee $. Further, the conjectures 47 and 65 (and some other conjectures which follow from them in Section 6) from \cite{chendebackertsai} suggest that the restriction of a Langlands parameter $\phi_\pi$ of $\pi $ to $I_k^{\rho(\pi)}$ is $G^\vee $ conjugate to the restricted Langlands parameter attached to if via the map $\Irr(G)_r \rightarrow \DL_r \xrightarrow[]{\simeq}\RP_r$ for $\rho(\pi)=r\in \Z_{(p)}\cap \Q_{\geq 0} $ described in this work and \cite{chendebackertsai}. These conjectures suggest that the packets $\Pi(\vartheta)$ for $\vartheta \in \DL^t$ that we constructed in the previous section are unions of $L$-packets conjectured by the Local Langlands correspondence. In fact, they are actually unions of infinitesimal classes as descibed in \ref{section: Bernstein center}. \par

We fix a choice of $\nu_m \in k^t$  and in isomorphism $\bar\F_q^\times \simeq (\Q/\Z)_{p'}$ as in Section \ref{section:decomposition}, which fixes the maps $\xi^r$ as described in Theorems \ref{thmTmodWto0center} and \ref{thmtmodWtocenter} for $r \in \Z_{(p)}\cap \Q_{\geq 0}$, and consider the elements in $\ZZ(G)$ which are in the image of $\xi^r$. If $z \in \mathrm{Im}(\xi^r)$ and $\pi_1,\pi_2 \in \Irr(G)$ such that $\Theta(\pi_1)=\Theta(\pi_2)$, then $f_z(\pi_1)=f_z(\pi_2)$ with $f_z$ as defined in \ref{section: Bernstein center}. In particular, $z \in \ZZ(G)$ acts by the same constant on all irreducible representations contained in an infinitesimal class, since they are contained in $\Pi(\vartheta)$ for some $\vartheta \in \DL^t$. Note that this is much coarser. Hence, the conjectural description of the stable center $\ZZ^{st}(G)$ suggests the following conjecture:
\begin{conjecture}
    Let $\ZZ^{st,r}(G)=\ZZ^{st}(G)\cap \ZZ^r(G)$ for $r \in \Z_{(p)}\cap \Q_{\geq 0}$ and $\xi^r$ be the maps described in Theorems \ref{thmtmodWtocenter} and \ref{thmTmodWto0center}. Then, $\mathrm{Im}(\xi^r) \subset \ZZ^{st,r}(G) $.
\end{conjecture}
By slight abuse of notation, we denote by $\mathbbm 1 $ the complex function in both $\C[(\bt^*//W)^F]$ and $\C[(T^\vee//W)^{[q]}]$ which takes the value $1$ at all points. Then, the idempotent element $\xi^r(\mbm 1)\in \ZZ^r(G)$ is exactly the depth-$r$ projector $[A_{\delta_r}]$ as described in \cite{BKV2}. The depth-$r$ projector was shown to be stable in \cite{BKV2}. A geometric approach was used in \cite{BKV} to show stability of the depth-zero projector. These provide some further evidence and possible approaches to prove the conjecture. Note that we would only need to prove that the projectors $e_\vartheta$ for $\vartheta \in \DL^t$ as described Proposition \ref{prop:projectorpackets} are stable.

\printbibliography

\bigskip
\textsc{School of Mathematics, University of Minnesota, Vincent Hall, Minneapolis, MN-55455}\\
\textit{E-mail address}: \texttt{bhatt356@umn.edu} 

\medskip

\textsc{School of Mathematics, University of Minnesota, Vincent Hall, Minneapolis, MN-55455}\\
\textit{E-mail address}: \texttt{chenth@umn.edu}

\end{document}